\documentclass{article}
\usepackage{arxiv}
\usepackage[utf8]{inputenc}
\usepackage[T1]{fontenc}
\usepackage{url}
\usepackage{mathptmx}
\usepackage{amsmath}
\usepackage{amsthm}
\usepackage{amssymb}
\usepackage{tikz}
\usepackage{float}
\newcommand{\llbracket}{[\mkern-3.5mu[}
\newcommand{\rrbracket}{]\mkern-3.5mu]}
\usepackage[hidelinks]{hyperref}

\newtheoremstyle{thm-break}{10pt}{3pt}{\normalfont}{}{\bfseries}{}{\newline}{\thmname{#1}\thmnumber{ #2}\thmnote{ (#3)}\par\vspace{0.5\baselineskip}}
\newtheoremstyle{remark-example-space}{10pt}{3pt}{\normalfont}{}{\bfseries}{.}{0.5em}{}
\theoremstyle{thm-break}
\newtheorem{theorem}{Theorem}[section]
\theoremstyle{remark-example-space}
\newtheorem{remark}[theorem]{Remark}
\theoremstyle{thm-break}
\newtheorem{definition}[theorem]{Definition}
\theoremstyle{thm-break}
\newtheorem{proposition}[theorem]{Proposition}
\theoremstyle{thm-break}
\newtheorem{lemma}[theorem]{Lemma}
\theoremstyle{thm-break}
\newtheorem{corollary}[theorem]{Corollary}
\theoremstyle{remark-example-space}
\newtheorem{example}[theorem]{Example}
\theoremstyle{remark-example-space}
\newtheorem{convention}[theorem]{Convention}

\newcounter{casei}
\newlength{\caselabelwidth}
\newenvironment{caselist}{%
  \begin{list}{(\roman{casei})}{%
    \usecounter{casei}%
    \setlength{\labelwidth}{\caselabelwidth}%
    \setlength{\labelsep}{0.4em}%
    \setlength{\leftmargin}{1.5em}%
    \setlength{\itemsep}{0.35\baselineskip}%
    \setlength{\topsep}{0.35\baselineskip}%
    \setlength{\parsep}{0pt}%
    \setlength{\partopsep}{0pt}%
  }%
}{\end{list}}
\providecommand{\keywords}[1]{}
\renewcommand{\keywords}[1]{%
\begin{center}
\small\textbf{Keywords:} #1
\end{center}
}
\renewenvironment{proof}[1][Proof]{\par\noindent\textbf{#1. }\normalfont}{\hfill$\square$\par}
\newcommand{\doi}[1]{\href{https://doi.org/#1}{\nolinkurl{https://doi.org/#1}}}
\newcommand{\MMLP}{\ifmmode\text{\textsc{MMLP}}\else\textsc{MMLP}\fi}
\newcommand{\MProlog}{\ifmmode\text{\textsc{MProlog}}\else\textsc{MProlog}\fi}
\newcommand{\KSys}{\ifmmode\mathrm{KDI4s5}\else KDI4s5\fi}
\newcommand{\Kmod}{\mathcal K}
\newcommand{\Amod}{\mathcal A}
\newcommand{\Idx}{I}
\newcommand{\Frames}{\mathfrak F}
\newcommand{\NNF}{\mathsf{nnf}}
\newcommand{\Corr}{\operatorname{Corr}}
\newcommand{\Comp}{\operatorname{Comp}}
\newcommand{\Inst}{\operatorname{Inst}}
\newcommand{\Flex}{\operatorname{Flex}}
\newcommand{\FV}{\operatorname{FV}}
\newcommand{\dom}{\operatorname{dom}}
\newcommand{\leadstoA}[1]{\mathrel{\leadsto_{#1}^{\Amod}}}
\newcommand{\eps}{\epsilon}
\newcommand{\sem}[1]{\llbracket #1\rrbracket}

\title{Multimodal Logic Programming with Full Formulas}
\author{%
\href{https://orcid.org/0000-0001-8878-9824}{Kenji Tokuo}\\
{\normalfont Department of Information Engineering, Oita College}\\
{\normalfont National Institute of Technology}\\
{\normalfont Oita 870-0152, Japan}\\
{\normalfont \texttt{tokuo@oita-ct.ac.jp}}%
}
\date{}
\renewcommand{\undertitle}{A preprint}

\begin{document}
\maketitle

\begin{abstract}
This paper presents a first-order multimodal logic programming system called \MMLP{}. The system accepts arbitrary formulas as both programs and queries, without restricting either side to Horn clauses or a separate goal grammar. Its declarative semantics is given independently by a Hilbert system for selected \(\mathrm{D}\), \(\mathrm{T}\), \(\mathrm{I}\), \(\mathrm{B}\), \(4\), and \(5\) modal principles. Execution uses a nested proof calculus with finite grammar certificates for modal propagation. Certificate reachability is equivalent to the associated Horn closure, certificate existence is decidable, and the calculus is cut-free complete. For quantified answer computation, we give a unification algorithm based on permission sets that controls eigenparameter scope. The algorithm always terminates and fails exactly when no admissible solution exists. A successful run returns a unifier that is itself admissible and through which all admissible solutions factor. Computed answers are declaratively correct, and each declaratively correct answer is an ordinary output instance of a computed answer. In proof search, \MMLP{} admits syntactic focalization and a fair and complete enumeration of focused answers. It represents exactly the same answer substitutions as Nguyen's pure logical \(\KSys\)-\MProlog{} on their common fragment, while admitting a strictly larger program and query language.
\end{abstract}


\section{Introduction}\label{sec:introduction}

\subsection{Motivation}\label{sec:motivation}

Strong syntactic restrictions on programs and goals contribute substantially to the operational clarity of classical logic programming. Definite programs use Horn clauses, while work on hereditary Harrop formulas and uniform proofs extends the language through controlled grammars for programs and goals~\cite{nadathur1993harrop,miller1991uniform,miller2022survey}. Classical goal-directed proof theory also studies fragments with disjunctive assumptions and answer substitutions~\cite{harland1997goal,nadathur1998uniform}. These restrictions permit structured proof search and yield familiar accounts of computed answers. They also exclude many formulas that are natural in a modal specification.

This paper adopts a different design choice. A program formula and a query formula may be any formula of a fixed first-order multimodal language. Boolean connectives, first-order quantifiers, and indexed modalities may occur at arbitrary depth. Thus a query may have an outer disjunction, implication may occur inside a program formula, and modalities may occur inside quantified contexts without a preliminary translation into a Horn clause language. We use the phrase \emph{full formulas} for this absence of a designated clause or goal fragment. The phrase is relative to the source language fixed in Section~\ref{sec:language}.

This change creates two proof-theoretic obligations. First, the modal proof theory must admit direct proof search over formulas without placing semantic accessibility atoms in the runtime syntax. Second, symbolic execution must preserve computed answer information across quantifiers and modal branches. Ordinary unification is insufficient for the second obligation because a proof metavariable created before an eigenparameter may not later acquire that eigenparameter through substitution.

We call the system \MMLP{}, for multimodal logic programming. Its proof theory addresses both requirements through a certificate calculus and scoped unification. The certificates are checked syntactically, and soundness and completeness of the calculus are established using Kripke semantics.

\subsection{Separation of declarative specification from execution}\label{sec:layers}

The declarative specification is independent of the execution calculus. A finite modal module $\Amod$ selects indexed schemes from the families $\mathrm{D}$, $\mathrm{T}$, $\mathrm{I}$, $\mathrm{B}$, $4$, and $5$. Throughout the paper, the first-order signature $\Sigma$ is fixed, and all worlds in each model share a single domain. We therefore suppress these fixed parameters from the notation for the proof systems. The symbol $H_{\Amod}$ denotes the Hilbert system, $\mathcal N_{\Amod}$ denotes the ordinary certificate calculus, and $\mathcal F_{\Amod}$ denotes its focused form. The Hilbert system contains classical first-order logic, normal indexed modalities, principles for constant domains, and the selected modal schemes. A user answer substitution $\theta$ for a query $G$ is declaratively correct for a finite program $P$ when $H_{\Amod}$ proves the implication from the conjunction of the closed program formulas to the universal closure of $G\theta$.

The operational development consists of a certificate calculus together with symbolic and focused execution. In the certificate calculus, formulas are arranged in indexed nested brackets. Symbolic execution adds substitutions and unification constraints, while focused execution adds scheduling data. A modal propagation step is annotated with a finite signed walk through the nested tree and a finite grammar derivation certifying the required modal reachability. Seriality remains a primitive rule that generates successors. The other selected modal principles contribute grammar productions and optional derived structural transformations.

This separation keeps declarative correctness, certificate derivability, and symbolic answer computation as distinct notions whose relationships are established by the results below.

\subsection{Main results}\label{sec:main-results}

The first group of results concerns the modal kernel. The selected $\mathrm{T}$, $\mathrm{I}$, $\mathrm{B}$, $4$, and $5$ conditions generate both a Horn closure on the edges of a finite nested tree and a closed semi-Thue grammar. Seriality $\mathrm D$ instead remains a primitive rule because it requires the existence of a successor rather than closure of an existing relation. We prove a constructive equivalence between the two presentations. Every Horn consequence yields a signed walk with a grammar derivation, and every such certificate yields the corresponding Horn consequence. This equivalence gives semantic soundness for certificate propagation, persistence under tree extension, and a decision procedure that returns a witness certificate when one exists. A fair saturation construction then gives cut-free semantic completeness for the ordinary certificate calculus. Independent Hilbert arguments yield full equivalence among Hilbert theoremhood, certificate derivability, and semantic validity.

The second group concerns scoped symbolic execution. Each flexible variable has a fixed permission set specifying the eigenparameters that may occur in its eventual value. The scoped unification algorithm combines ordinary decomposition, occurs checking, binding, and permission restriction. The central unification theorem proves termination and characterizes failure by the absence of admissible solutions. When the algorithm succeeds, the returned unifier is admissible and all admissible solutions factor through it.

The third group concerns computed answers. Symbolic atomic closure contributes first-order equations, while existential rules introduce fresh proof metavariables. After every branch ends in a symbolic truth or atomic leaf, scoped unification is run on the accumulated terminal constraints. Scoped instantiation converts every successful symbolic computation into an ordinary certificate derivation. The converse direction uses proof abstraction and scoped lifting. These results show that computed answers are declaratively correct and that each declaratively correct answer is an ordinary output instance of a computed answer. For every successful symbolic computation, fair iterative deepening eventually reaches a canonical reference derivation with the same instance closure on the user answer variables. It follows that the instance closure of computed answers coincides with the set of declaratively correct answers.

The fourth group concerns focused execution. The focused calculus alternates maximal asynchronous phases with progressive synchronous focuses. The syntactic normalization argument uses height-preserving invertibility, deep formula weakening, eigenparameter regularization, component renaming, and certificate persistence. Every ordinary certificate derivation has a focused derivation of the same root. Focused lifting establishes the corresponding computed answer results for the focused symbolic calculus.

The final group concerns conservativity. Nguyen writes $\KSys$ for the normal multimodal logic obtained from $\mathrm K$ by adding the $\mathrm D$ and $\mathrm I$ families, the strong positive introspection schemes $4_s\colon\Box_iA\to\Box_j\Box_iA$, and $5$. His $\KSys$-\MProlog{} provides a concrete comparison system with selective linear definite clause (SLD) resolution and computed answers~\cite{nguyen2006multimodal}. We compare its connected frame normal form with the frame class determined by a finite \MMLP{} modal module. The resulting equivalence relates \MMLP{} implication validity to Nguyen's local consequence. Nguyen's soundness and completeness theorems then show that the computed answers of \MProlog{} and \MMLP{} have the same instance closure on the common fragment. The same closure is obtained by focused \MMLP{} execution.

An implementation of \MMLP{} is publicly available. Repository information is given in the \hyperref[sec:code-availability]{Code availability section}.

\subsection{Proof-theoretic context}\label{sec:proof-context}

The modal kernel is formulated using deep and nested sequents~\cite{brunnler2009deep,brunnler2010nested}. Modal reasoning with grammars also has a substantial history in proof theory and decision procedures. Demri and de Nivelle study regular grammar logics with converse through a first-order translation~\cite{demri2005regular}. Tiu, Ianovski, and Gor\'e develop nested calculi for multimodal grammar logics with converse and closed semi-Thue systems~\cite{tiu2012grammar}. Their propagation method uses signed paths through nested trees and reduces rule applicability to language intersection. Lyon extends reachability rules parameterized by grammars to first-order modal calculi~\cite{lyon2022firstorder}. Lyon and Orlandelli develop quantified nested systems~\cite{lyon2023quantified}, and their later work provides reachability calculi for Horn-characterizable quantified modal logics~\cite{lyon2026horn}. \MMLP{} uses these proof-theoretic methods to obtain explicit witnesses for indexed modal propagation.

Work on modal logic programming includes direct systems and methods based on translation. Early direct work includes MOLOG and a declarative semantics for modal logic programs~\cite{farinas1986molog,balbiani1988declarative}. Debart, Enjalbert, and Lescot translate multimodal programs into order-sorted equational logic~\cite{debart1992multimodal}, while Nonnengart develops an optimized translation that preserves Horn form when the source is Horn~\cite{nonnengart1994modalities}. Orgun and Ma survey the early temporal and modal logic programming literature~\cite{orgun1994overview}. Baldoni, Giordano, and Martelli develop multimodal program structures with goal-directed proof procedures and sequent calculi~\cite{baldoni1996framework,baldoni1998modal}. Nguyen gives fixpoint and SLD semantics for modal logic programs~\cite{nguyen2003fixpoint}, extends the framework to multimodal logics~\cite{nguyen2006multimodal}, and later develops modal deductive databases and MDatalog~\cite{nguyen2007databases}. \MMLP{} follows the direct proof-theoretic tradition while permitting arbitrary source formulas in program and query positions and using a separate certificate kernel.

The use of richer formulas in logic programming has an independent proof-theoretic history. Hereditary Harrop formulas and uniform proofs support controlled implications and quantification in goals and programs~\cite{nadathur1993harrop,miller1991uniform}. Harland studies goal-directed provability and disjunctive information in classical logic~\cite{harland1997goal}, and Nadathur gives a related analysis of uniform provability in classical logic~\cite{nadathur1998uniform}. Gabbay and Olivetti develop goal-directed proof theory for a range of nonclassical logics, including modal systems~\cite{gabbay2000goal}. Stone gives a sound and complete goal-directed proof system for first-order multimodal logic with disjunction, modular goals, and local assumptions~\cite{stone2005disjunction}. \MMLP{} admits any formula of the source language fixed in Section~\ref{sec:language} in program and query positions, with no designated clause or goal fragment. Miller's survey gives a broader account of the structural proof-theoretic foundations of logic programming~\cite{miller2022survey}.

Scope constraints in unification are likewise an established concern in the proof theory of logic programming. Miller studies unification under a mixed prefix~\cite{miller1992unification}. Nadathur, Jayaraman, and Kwon treat mixed universal and existential quantification operationally by attaching scope information to constants and variables and modifying unification accordingly~\cite{nadathur1995scoping}. \MMLP{} addresses this first-order dependency problem with scoped unification based on permission sets that remain fixed after introduction.

Andreoli introduced focused proofs for linear logic~\cite{andreoli1992focusing}, and later work developed focusing for linear, intuitionistic, and classical systems~\cite{liang2009focusing}. Chaudhuri, Pfenning, and Price relate focusing ideas to forward and backward chaining in the inverse method~\cite{chaudhuri2008logical}. Focused calculi have also been developed directly for nested modal proof systems. Chaudhuri, Marin, and Stra{\ss}burger develop focused and synthetic nested sequents for classical modal logics~\cite{chaudhuri2016focusednested} and modular focused nested systems for intuitionistic modal logics~\cite{chaudhuri2016modularfocused}. The focalization theorem for \MMLP{} is specific to its certificate calculus and propagation side conditions. It establishes equivalence between ordinary and focused certificate derivability. The phase discipline does not distinguish program occurrences from query occurrences, so the result is a focalization theorem rather than a characterization in terms of uniform proofs.

\subsection{Paper structure}\label{sec:paper-structure}

Section~\ref{sec:language} fixes the source language, Hilbert semantics, and declarative answers. Section~\ref{sec:certificates} presents finite grammar certificates and the exact certificate calculus. Section~\ref{sec:unification} gives scoped unification. Section~\ref{sec:symbolic} develops symbolic proofs and computed answers. Section~\ref{sec:focusing} proves focalization and focused answer completeness. Section~\ref{sec:conservativity} establishes the $\KSys$-\MProlog{} comparison. The appendices contain normalization results, detailed Hilbert certification, and the detailed structural soundness argument.

\section{Declarative framework}\label{sec:language}

The source language distinguishes the language presented to users from the negation normal form (NNF) language used by the proof calculus. This section fixes the source syntax and the independent Hilbert specification. The normalization facts required by the calculus appear in Appendix~\ref{app:normalization}.

\subsection{Source syntax}\label{sec:source-language}

Fix a countably infinite effectively enumerable set $\mathsf{Var}$ of object variables. Let $\Sigma$ be a countable effectively presented first-order signature with predicate, constant, and function symbols. Here effective presentation means that the symbols and their arities admit a computable enumeration and that equality of symbols is decidable. Terms have the grammar $t ::= x \mid c \mid f(t_1,\ldots,t_n)$, where $x\in\mathsf{Var}$. For a finite tuple of terms $t_1,\ldots,t_n$, write $\bar t=(t_1,\ldots,t_n)$. For a term $t$, write $\FV(t)$ for its set of object variables.
Let the finite nonempty set of modal indices be
$\Idx=\{1,\ldots,m\}.$

\begin{definition}[Source formulas]\label{def:source-formulas}
The \emph{source formulas} of \MMLP{} are generated by
\[
 A ::= p(t_1,\ldots,t_n) \mid \bot \mid \neg A \mid A\land A \mid A\lor A
 \mid A\to A \mid \forall xA \mid \exists xA \mid \Box_iA \mid \Diamond_iA.
\]
Here $i\in\Idx$. Equality is not included in the object language. An arbitrary source formula may occur as a program formula or as a query formula.
\end{definition}

For a source formula $B$, write $\FV(B)$ for its set of free object variables. For source formulas, $A\equiv_\alpha B$ means that $A$ and $B$ are alpha equivalent.

At the Hilbert level, $A\leftrightarrow B$ abbreviates $(A\to B)\land(B\to A)$ throughout.

The proof calculus uses the following NNF language. We give its syntax here because the certificate calculus introduced later in the paper uses this language. The proofs of the normalization facts remain in Appendix~\ref{app:normalization}.

\begin{definition}[Core NNF language]\label{def:core-nnf}
The \emph{core formulas} are
\[
 C ::= p(\bar t)\mid\overline p(\bar t)\mid\bot\mid\top
 \mid C\land C\mid C\lor C\mid\forall xC\mid\exists xC
 \mid\Box_iC\mid\Diamond_iC.
\]
The proof calculus writes $\overline p(\bar t)$ for the negative literal corresponding to $\neg p(\bar t)$ in Hilbert formulas, while $\top$ corresponds to the abbreviation $\neg\bot$. Whenever a core formula occurs in a Hilbert derivability judgment, it denotes the source formula obtained by replacing each $\overline p(\bar t)$ with $\neg p(\bar t)$ and each $\top$ with $\neg\bot$, while leaving all other constructors unchanged. Their semantic clauses are fixed together with the model semantics in Definition~\ref{def:core-semantics}.
\end{definition}

\begin{definition}[NNF translations]\label{def:nnf-translations}
The two \emph{NNF translations} are defined simultaneously by
\[
\begin{array}{ll}
\NNF^+(p(\bar t))=p(\bar t), & \NNF^-(p(\bar t))=\overline p(\bar t),\\
\NNF^+(\bot)=\bot, & \NNF^-(\bot)=\top,\\
\NNF^+(\neg A)=\NNF^-(A), & \NNF^-(\neg A)=\NNF^+(A),\\
\NNF^+(A\land B)=\NNF^+(A)\land\NNF^+(B), & \NNF^-(A\land B)=\NNF^-(A)\lor\NNF^-(B),\\
\NNF^+(A\lor B)=\NNF^+(A)\lor\NNF^+(B), & \NNF^-(A\lor B)=\NNF^-(A)\land\NNF^-(B),\\
\NNF^+(A\to B)=\NNF^-(A)\lor\NNF^+(B), & \NNF^-(A\to B)=\NNF^+(A)\land\NNF^-(B),\\
\NNF^+(\forall xA)=\forall x\NNF^+(A), & \NNF^-(\forall xA)=\exists x\NNF^-(A),\\
\NNF^+(\exists xA)=\exists x\NNF^+(A), & \NNF^-(\exists xA)=\forall x\NNF^-(A),\\
\NNF^+(\Box_iA)=\Box_i\NNF^+(A), & \NNF^-(\Box_iA)=\Diamond_i\NNF^-(A),\\
\NNF^+(\Diamond_iA)=\Diamond_i\NNF^+(A), & \NNF^-(\Diamond_iA)=\Box_i\NNF^-(A).
\end{array}
\]
\end{definition}

\subsection{Hilbert specification}\label{sec:hilbert-specification}

The target modal logic is specified independently of the operational calculus. A finite modal module selects modal principles, and a Hilbert system specifies theoremhood.

\begin{definition}[Selected modal module]\label{def:modal-module}
A \emph{selected modal module} $\Amod$ is a finite set of indexed modal schemes chosen from the following families. Selecting one displayed scheme means including all of its formula instances:
\[
\begin{array}{rcl}
\mathrm{D}_i&:&\Box_iA\to\Diamond_iA,\\
\mathrm{T}_i&:&\Box_iA\to A,\\
\mathrm I(i,j)&:&\Box_iA\to\Box_jA,\\
\mathrm B(i,j)&:&A\to\Box_i\Diamond_jA,\\
4(i,j,k)&:&\Box_iA\to\Box_j\Box_kA,\\
5(i,j,k)&:&\Diamond_iA\to\Box_j\Diamond_kA.
\end{array}
\]
\end{definition}

\begin{definition}[Hilbert base]\label{def:hilbert-base}
For a selected modal module $\Amod$, the system $H_{\Amod}$ contains a classical propositional basis, the usual first-order schemes, modus ponens, and theorem generalization. For each $i\in\Idx$ it contains the normal modal axiom
$\mathrm{K}_i\colon \Box_i(A\to B)\to(\Box_iA\to\Box_iB),$
theorem necessitation, modal duality
$\Diamond_iA\leftrightarrow\neg\Box_i\neg A,$
and the constant-domain schemes given by the Barcan formula $\mathrm{BF}_i\colon \forall x\Box_iA\to\Box_i\forall xA$ and the converse Barcan formula $\mathrm{CBF}_i\colon \Box_i\forall xA\to\forall x\Box_iA$. The system also contains every formula instance of each modal scheme selected by $\Amod$. Both theorem generalization and theorem necessitation are applicable only when their premises are already theorems of the system.
\end{definition}

For a term $t$ that is free for $x$ in $A$, the notation $A(t/x)$ denotes capture-avoiding substitution of $t$ for the free occurrences of $x$ in $A$. We fix one concrete classical and first-order basis throughout the paper. Its schemes are
\[
\begin{array}{lll}
\mathrm{H1} & A\to(B\to A), & \\
\mathrm{H2} & (A\to(B\to C))\to((A\to B)\to(A\to C)), & \\
\mathrm{C1} & A\to(B\to A\land B), & \\
\mathrm{C2} & A\land B\to A, & \\
\mathrm{C3} & A\land B\to B, & \\
\mathrm{D1} & A\to A\lor B, & \\
\mathrm{D2} & B\to A\lor B, & \\
\mathrm{D3} & (A\to C)\to((B\to C)\to((A\lor B)\to C)), & \\
\mathrm{N1} & \neg A\to(A\to\bot), & \\
\mathrm{N2} & (A\to\bot)\to\neg A, & \\
\mathrm{CL} & \neg\neg A\to A, & \\
\mathrm{EFQ} & \bot\to A, & \\[4pt]
\mathrm{Q1} & \forall xA\to A(t/x), & \\
\mathrm{Q2} & \forall x(A\to B)\to(A\to\forall xB), & x\notin\FV(A),\\
\mathrm{Q3} & A(t/x)\to\exists xA, & \\
\mathrm{Q4} & \forall x(A\to B)\to(\exists xA\to B), & x\notin\FV(B).
\end{array}
\]
\begin{theorem}[Hilbert normalization]\label{thm:hilbert-nnf}
For every source formula $F$, both $H_{\Amod}\vdash F\leftrightarrow\NNF^{+}(F)$ and $H_{\Amod}\vdash \neg F\leftrightarrow\NNF^{-}(F)$ hold.
\end{theorem}

\begin{proof}
The complete proof is given in Appendix~\ref{app:hilbert-nnf}.
\end{proof}

\subsection{Frame specification}\label{sec:frame-specification}

This subsection fixes the relational structures used by the semantic metatheory. For binary relations $R$ and $S$ on the same set, relational composition is interpreted from left to right throughout the paper. Thus $(u,v)\in R\circ S$ exactly when some $w$ satisfies $uRw$ and $wSv$.

\begin{definition}[Selected frame class]\label{def:frame-class}
A \emph{multimodal frame} over $\Idx$ is a tuple $(W,(R_i)_{i\in\Idx})$, where $W$ is a nonempty set of worlds and each $R_i$ is a binary relation on $W$. A \emph{pointed multimodal frame} is a tuple $(W,\tau,(R_i)_{i\in\Idx})$ whose underlying multimodal frame is $(W,(R_i)_{i\in\Idx})$ and whose distinguished world $\tau$ belongs to $W$. The \emph{selected frame class} $\Frames(\Amod)$ consists of the multimodal frames that satisfy the following clauses whenever the corresponding scheme belongs to $\Amod$:
\[
\begin{array}{rcl}
\mathrm{D}_i&:&\forall u\,\exists v\,R_i(u,v),\\
\mathrm{T}_i&:&R_i(u,u),\\
\mathrm I(i,j)&:&R_j(u,v)\to R_i(u,v),\\
\mathrm B(i,j)&:&R_i(u,v)\to R_j(v,u),\\
4(i,j,k)&:&R_j(u,v)\land R_k(v,w)\to R_i(u,w),\\
5(i,j,k)&:&R_j(u,v)\land R_i(u,w)\to R_k(v,w).
\end{array}
\]
All variables in the displayed relational clauses range over $W$, and the free variables are universally quantified.
\end{definition}

\subsection{Constant-domain models}\label{sec:models}

The selected modal module determines the conditions on the accessibility relations. We now state the first-order semantic choices that remain constant throughout the proof theory, namely a shared domain and rigid interpretations of terms.

\begin{definition}[Multimodal model with rigid terms]\label{def:modal-model}
Let $(W,(R_i)_{i\in\Idx})$ be a frame in $\Frames(\Amod)$. A \emph{multimodal model with rigid terms} based on this frame is a structure
\[
\mathcal M=(W,(R_i)_{i\in\Idx},D,\mathcal J),
\]
where $D$ is a nonempty constant domain. The interpretation $\mathcal J$ assigns an element of $D$ to every constant symbol, a function $D^n\to D$ to every $n$-ary function symbol, and, at each world $w\in W$, a relation on $D^n$ to every $n$-ary predicate symbol. The interpretations of constants and function symbols are independent of the world. A \emph{pointed model} is a pair $(\mathcal M,\tau)$ with $\tau\in W$. An \emph{assignment} is a map $g:\mathsf{Var}\to D$. Satisfaction at a world $w\in W$ under an assignment $g$ uses the classical Boolean clauses and the usual first-order clauses over $D$. The modal clauses are
\[
\mathcal M,w,g\models\Box_iA
\Longleftrightarrow
\mathcal M,v,g\models A\text{ for every }v\in W\text{ with }wR_iv
\]
and
\[
\mathcal M,w,g\models\Diamond_iA
\Longleftrightarrow
\mathcal M,v,g\models A\text{ for some }v\in W\text{ with }wR_iv.
\]
\end{definition}

For a source formula $A$, write $\mathcal M\models A$ when $\mathcal M,w,g\models A$ for every $w\in W$ and every assignment $g$. The formula $A$ is \emph{valid on a frame} when $\mathcal M\models A$ for every model $\mathcal M$ based on that frame. For a class $\mathcal C$ of frames, $\mathcal C\models A$ means that $A$ is valid on every frame in $\mathcal C$.

\begin{proposition}[Modal correspondence]\label{prop:modal-correspondence}
Every formula instance of every modal scheme selected by $\Amod$ is valid on every frame in $\Frames(\Amod)$.
\end{proposition}

\begin{proof}
For $\mathrm{D}_i$, suppose $\Box_iA$ holds at $u$. By seriality, choose an $R_i$ successor $v$ of $u$. Then $A$ holds at $v$, so $\Diamond_iA$ holds at $u$. For $\mathrm{T}_i$, if $\Box_iA$ holds at $u$, reflexivity gives $uR_iu$, so $A$ holds at $u$. For $\mathrm I(i,j)$, every $R_j$ successor is also an $R_i$ successor. For $\mathrm B(i,j)$, from $uR_iv$ the frame clause gives $vR_ju$, so $u$ witnesses the required $\Diamond_jA$ at $v$. For $4(i,j,k)$, a path $uR_jvR_kw$ gives $uR_iw$. For $5(i,j,k)$, an $R_i$ witness $w$ at $u$ remains an $R_k$ witness at every $R_j$ successor $v$ because the frame clause gives $vR_kw$.
\end{proof}

\begin{definition}[Core semantic clauses]\label{def:core-semantics}
Satisfaction is extended to the two additional core forms by
$\mathcal M,w,g\models\overline p(\bar t)$ exactly when $\mathcal M,w,g\not\models p(\bar t)$, and by declaring $\mathcal M,w,g\models\top$ for every model $\mathcal M$, world $w$, and assignment $g$. The preceding notions of model, frame, and class validity extend to core formulas using these clauses.
\end{definition}

For a first-order term $t$, write $\sem{t}^{\mathcal M}_g$ for its ordinary denotation in $\mathcal M$ under assignment $g$. When the model is fixed, we abbreviate this by $\sem{t}_g$.

All results below use this constant-domain semantics with rigid constants and function symbols.

With the source and core semantics in place, we can state the semantic and substitution properties of normalization.

\begin{theorem}[Normalization adequacy]\label{thm:normalization-adequacy}
For every source formula $A$, every constant-domain model with rigid terms $\mathcal M$, every world $w$, and every assignment $g$,
$\mathcal M,w,g\models \NNF^{+}(A)\Longleftrightarrow \mathcal M,w,g\models A$
and
$\mathcal M,w,g\models \NNF^{-}(A)\Longleftrightarrow \mathcal M,w,g\not\models A$.
The two translations preserve free variables and commute with every finite simultaneous capture-avoiding substitution, up to alpha equivalence.
\end{theorem}

\begin{proof}
The complete proof is given in Appendix~\ref{app:nnf-adequacy}.
\end{proof}

The substitution lemmas for rigid terms are required in the soundness proof of the quantified rules. Their full statements and proofs are given in Appendix~\ref{app:substitution}.

\subsection{Declarative answers}\label{sec:declarative-answers}

Let a program be a finite set of source formulas
$P=\{P_1,\ldots,P_n\}.$
Free variables in program formulas are universally closed before the formulas are added to the program. Henceforth, each $P_i$ denotes its universal closure. Put
$\widehat P=P_1\land\cdots\land P_n$
when $n>0$, and put $\widehat P=\neg\bot$ when $P$ is empty.

Program assumptions occur in the antecedent $\widehat P$, while necessitation applies only to theorems of $H_{\Amod}$. Answer correctness is therefore stated as Hilbert theoremhood of an implication with antecedent $\widehat P$.

For each query $G$, fix a finite set $A_G\subseteq\FV(G)$ of object variables whose values form the user answer.

\begin{convention}[Ordinary substitutions]\label{conv:ordinary-substitutions}
An \emph{ordinary substitution} is a map from $\mathsf{Var}$ to source terms that differs from the identity at only finitely many variables. Its nonidentity domain is written $\dom(\theta)$. It extends simultaneously to terms and formulas, with bound variables renamed as necessary to avoid capture. Unless a later convention explicitly specifies a normalized update, substitutions act in postfix order. For substitutions $\theta$ and $\gamma$, the product $\theta\gamma$ means that $\theta$ is applied first and $\gamma$ second. Thus $t\theta\gamma=(t\theta)\gamma$ and $(\theta\gamma)(X)=(\theta(X))\gamma$. For a set $B$ of object variables, $\theta\upharpoonright B$ denotes the restriction of $\theta$ to $B$, with identity bindings omitted.
\end{convention}

\begin{definition}[User answer substitution]\label{def:user-answer}
A \emph{user answer substitution} for $G$ is an ordinary substitution $\theta$ with $\dom(\theta)\subseteq A_G$. Thus its range contains only object variables, constants, and function symbols. Variables in $A_G$ may occur in the range, so aliases such as $X\mapsto Y$ are permitted.
\end{definition}

In the rest of the paper, a term of the source language is called an \emph{ordinary user term}.

\begin{definition}[Hilbert correct answer]\label{def:hilbert-correct}
Let $\theta$ be a user answer substitution for $G$, and let $Y_1,\ldots,Y_r$ be any enumeration of $\FV(G\theta)$. We say that $\theta$ is \emph{Hilbert correct} for $(P,G)$ when
$H_{\Amod}\vdash \widehat P\to \forall Y_1\cdots\forall Y_r\,G\theta$.
The set of such substitutions is written $\Corr^{\mathrm H}_{\Amod}(P,G)$.
\end{definition}

We also say that the members of $\Corr^{\mathrm H}_{\Amod}(P,G)$ are \emph{declaratively correct}.

\begin{example}[Variable alias]\label{ex:variable-alias}
Let $P=\{\forall z\,p(z,z)\}$, $G=p(X,Y)$, and $A_G=\{X,Y\}$.
The substitution $\theta=\{X\mapsto Y\}$ is a legal user answer. The instantiated query is $p(Y,Y)$, whose universal closure follows from the program.
\end{example}

\begin{remark}[Existential success]\label{rem:existential-success}
A successful proof can establish a query without fixing every answer variable to a ground term. The declarative condition closes the remaining free variables universally after the user substitution. This convention permits general symbolic answers and ensures that symbols local to the proof do not occur in the user answer.
\end{remark}

\section{Certificate kernel}\label{sec:certificates}

This section defines the proof objects and syntactic data used to justify modal propagation for a multimodal module $\Amod$.

\subsection{Nested syntax}\label{sec:nested-syntax}

The recursive structures defined below have nodes called components. Fix a countably infinite effectively enumerable supply of names for these components, disjoint from the first-order and modal syntax. Fresh component names used below are always chosen from this supply.

\begin{definition}[Indexed one-sided nested sequent]\label{def:nested-sequent}
A \emph{nested sequent} is a finite expression generated recursively by
\[
 \Gamma=A_1,\ldots,A_n,[\Delta_1]_{i_1},\ldots,[\Delta_r]_{i_r}.
\]
The formulas $A_1,\ldots,A_n$ are core formulas, and the indices $i_1,\ldots,i_r$ belong to $\Idx$. Its formula interpretation is
\[
 \sem{\Gamma}=
 A_1\lor\cdots\lor A_n
 \lor\Box_{i_1}\sem{\Delta_1}
 \lor\cdots\lor\Box_{i_r}\sem{\Delta_r},
\]
with $\sem{\varnothing}=\bot$. Each component receives one persistent syntactic name from the fixed supply, called its \emph{position name}. A component equipped with this name is also called a \emph{position}. An immediate child $v$ of a component $u$ under an $i$ bracket gives an explicit tree edge $u\xrightarrow{i}v$.
\end{definition}

The proof objects above are related to nested sequent methods in modal proof theory~\cite{brunnler2010nested}.

A nested sequent $\mathcal G$ is \emph{valid over} $\Frames(\Amod)$ when $\Frames(\Amod)\models\sem{\mathcal G}$.

Distinct components have distinct position names. Below we freely refer to a component by its position name and use the words component and position interchangeably when no confusion can arise. A \emph{one-hole nested context} $\mathcal G\{\ \}_u$ is a nested sequent with one distinguished multiset position inside component $u$. Filling the hole with formulas or bracketed child sequents yields notation such as $\mathcal G\{A,B\}_u$ or $\mathcal G\{[A]_i\}_u$. When the position is irrelevant, we write simply $\mathcal G\{\ \}$. Two distinguished holes are written $\mathcal G\{\ \}_u\{\ \}_v$, and filling them gives expressions such as $\mathcal G\{A\}_u\{B,\Delta\}_v$.

\subsection{Signed walks}\label{sec:signed-walks}

Let
$\Idx^{\pm}=\{i,\bar i\mid i\in\Idx\}$
with involution $\overline{i}=\bar i$, $\overline{\bar i}=i$, and
$\overline{l_1\cdots l_n}=\bar l_n\cdots\bar l_1.$
Every explicit edge $u\xrightarrow{i}v$ determines the forward signed step $u\xrightarrow{i}v$ and the reverse signed step $v\xrightarrow{\bar i}u$. A finite sequence of signed steps is a \emph{signed tree walk}. Its label is the concatenation of the signed edge symbols, and we write $\operatorname{lab}(\omega)$ for the label of a walk $\omega$. The empty walk has label $\eps$.

\begin{definition}[Closed semi-Thue grammar]\label{def:semi-thue}
The selected nonserial modal schemes generate the following productions:
\[
\begin{array}{rcl}
\mathrm{T}_i&:&i\to\eps,\\
\mathrm I(i,j)&:&i\to j,\\
\mathrm B(i,j)&:&\bar j\to i,\\
4(i,j,k)&:&i\to jk,\\
5(i,j,k)&:&k\to\bar j i.
\end{array}
\]
For every production $\xi\to\zeta$, the grammar also contains the converse production $\bar\xi\to\bar\zeta$. The resulting finite \emph{closed semi-Thue system} is written $S_{\Amod}$. For signed words $\xi_1,\xi_2$ and a production $l\to\zeta$ in $S_{\Amod}$, write $\xi_1l\xi_2\Rightarrow_{\Amod}\xi_1\zeta\xi_2$ for one contextual rewrite, and let $\Rightarrow_{\Amod}^*$ be its reflexive transitive closure. Here \emph{closed} means closure of the production set under the bar involution on signed words. The rewrite relation itself remains directed. Seriality contributes no production because it requires existence of a successor rather than a relation inclusion.
\end{definition}

The grammar construction above is related to propagation and reachability methods based on closed semi-Thue systems~\cite{tiu2012grammar,lyon2022firstorder,lyon2026horn}.

\begin{definition}[Modal propagation certificate]\label{def:certificate}
Let $\mathcal G$ be a nested sequent and let $u,v$ be positions, possibly equal. A \emph{certificate} for modality $i$ from $u$ to $v$ is a pair $(\omega,\mathfrak d)$ such that $\omega$ is a finite signed walk from $u$ to $v$ and $\mathfrak d$ is a finite derivation
$i\Rightarrow_{\Amod}^*\operatorname{lab}(\omega).$
We write
$u\leadstoA{i}v$
when a certificate exists. Every proof rule that uses this judgment includes an actual pair $(\omega,\mathfrak d)$ in the annotated proof object.
\end{definition}

\begin{convention}[Coincident certificate endpoints]\label{conv:coincident-endpoints}
In the notation for a nested context with two holes introduced above, the annotated positions may coincide. When $u=v$, every displayed formula insertion takes place in the same component and the inserted occurrences remain distinct occurrences. This convention covers the empty walk associated with a $\mathrm{T}_i$ grammar production and any other certificate with equal endpoints.
\end{convention}

A certificate checker verifies the walk against the current nested tree and the grammar derivation against $S_{\Amod}$.

\subsection{Finite Horn closure}\label{sec:horn-closure}

The frame clauses in Definition~\ref{def:frame-class} also define a syntactic Horn closure on a set of positions. We formulate the closure for an arbitrary position set because the canonical countermodel will use an infinite union of finite nested trees. At each stage, the construction is effective.

Let $\mathcal C_{\Amod}$ be the set of relation clauses for the selected $\mathrm{T}$, $\mathrm{I}$, $\mathrm{B}$, $4$, and $5$ schemes. Seriality is excluded because it requires existence of a successor rather than closure under a relation clause. Let $N$ be a set of positions and let $E=(E_i)_{i\in\Idx}$ be a family of explicit edges with $E_i\subseteq N^2$. Put $E^0=E$. Given $E^n=(E_i^n)_{i\in\Idx}$, let $E^{n+1}$ be the family obtained by adjoining the head of every instance of a clause in $\mathcal C_{\Amod}$ whose premises belong to $E^n$. Write $\operatorname{Cl}_{\Amod}(E)=\bigcup_{n\in\mathbb N}E^n$. A \emph{relation fact} over $N$ is an expression $R_i(u,v)$ with $i\in\Idx$ and $u,v\in N$. A \emph{Horn derivation} of a relation fact from $E$ is a finite rooted tree whose nodes are labeled by relation facts. Each node is justified either because its label $R_i(u,v)$ corresponds to an explicit edge $(u,v)\in E_i$, or because its label is the head of an instance of a clause in $\mathcal C_{\Amod}$ and the labels of its children are exactly the premises of that instance. The \emph{height} of a Horn derivation is the maximum number of clause applications on a root-to-leaf path. A derivation justified only by an explicit edge has height zero. A clause instance with no premises gives a leaf of height one.

\begin{lemma}[Derivational characterization of the Horn closure]\label{lem:horn-derivational}
For every $n\in\mathbb N$ and every relation fact $R_i(u,v)$, $(u,v)\in E_i^n$ if and only if $R_i(u,v)$ has a Horn derivation from $E$ of height at most $n$. Consequently, $(u,v)\in\operatorname{Cl}_{\Amod}(E)_i$ if and only if $R_i(u,v)$ has a Horn derivation from $E$.
\end{lemma}

\begin{proof}
We use induction on $n$.

Base case. At stage zero, $E^0=E$, and the facts represented by pairs in $E$ are exactly those with a derivation justified by an explicit edge and having height zero.

Induction step. Assume the equivalence at stage $n$.

\begin{caselist}
\item Case $(u,v)\in E_i^{n+1}$. If $(u,v)\in E_i^n$, the induction hypothesis applies. Otherwise $R_i(u,v)$ is the head of a clause instance whose premises all correspond to pairs in $E^n$. The induction hypothesis gives derivations of those premises of height at most $n$. Placing the clause instance above them gives a derivation of the head of height at most $n+1$. A clause with no premises yields a clause derivation consisting of one node and having height one.
\item Case $R_i(u,v)$ has a Horn derivation of height at most $n+1$. If its root is justified by an explicit edge, then $(u,v)\in E_i^0\subseteq E_i^{n+1}$. Otherwise the root is a clause instance. Every premise subtree has height at most $n$, so the induction hypothesis places every premise pair in the corresponding family at stage $n$. The construction of $E^{n+1}$ therefore adds $(u,v)$ to $E_i^{n+1}$.
\end{caselist}

The consequence for $\operatorname{Cl}_{\Amod}(E)$ follows by taking the union over all finite stages.
\end{proof}

The next two lemmas establish finite stabilization and characterize the closure as the least relation family containing the explicit edges and satisfying the selected Horn clauses. On a finite tree, stabilization makes the closure effectively computable. This characterization also makes the result independent of the order in which those clauses are applied.

\begin{lemma}[Finite stabilization]\label{lem:finite-stabilization}
If $N$ is finite, the closure construction stabilizes after finitely many strict growth stages.
\end{lemma}

\begin{proof}
There are $|\Idx|\,|N|^2$ possible indexed triples $(i,u,v)$. A strict stage adds at least one new triple, and no stage removes a triple. Hence only finitely many strict stages can occur.
\end{proof}

\begin{lemma}[Least Horn closure]\label{lem:least-horn-closure}
The family $\operatorname{Cl}_{\Amod}(E)$ is the least relation family that contains $E$ and satisfies every selected Horn clause.
\end{lemma}

\begin{proof}
Every stage contains $E$. If the premises of a Horn instance occur in the union, they occur together at some finite stage, so the head occurs at the next stage. The union is therefore closed. Conversely, let $T$ contain $E$ and satisfy every clause. We prove $E^n\subseteq T$ for every $n$ by induction on $n$.

Base case. Since $E^0=E$ and $E\subseteq T$, we have $E^0\subseteq T$.

Induction step. Assume $E^n\subseteq T$. Every pair added at stage $n+1$ is either already in $E^n$ or is the head of a selected Horn clause whose premises belong to $E^n$. In the first case it belongs to $T$ by the induction hypothesis. In the second case all premises belong to $T$ by the induction hypothesis, and closure of $T$ under the selected clause places the head in $T$. Hence $E^{n+1}\subseteq T$.

Taking the union gives $\operatorname{Cl}_{\Amod}(E)\subseteq T$.
\end{proof}

\subsection{Equivalence with Horn closure}\label{sec:certificate-horn}

For a finite nested tree $\mathcal G$, let $E_i(\mathcal G)$ be its explicit $i$ edge family.

\begin{theorem}[Certificate--Horn equivalence]\label{thm:certificate-horn}
For every $i\in\Idx$ and all positions $u,v$,
\[
 u\leadstoA{i}v
 \Longleftrightarrow
 (u,v)\in
 \operatorname{Cl}_{\Amod}
 ((E_j(\mathcal G))_{j\in\Idx})_i.
\]
\end{theorem}

\begin{proof}
We first pass from Horn closure to certificates. By Lemma~\ref{lem:horn-derivational}, it is enough to start with a Horn derivation of $R_i(u,v)$. We use induction on its height.

Base case. If the derivation has height zero, its final fact is an explicit edge $u\xrightarrow{i}v$. Take the walk consisting of that single edge. Its label is $i$, and the grammar derivation of length zero $i\Rightarrow_{\Amod}^* i$ gives the certificate.

Induction step. Assume the claim for Horn derivations of smaller height and consider the clause used at the root.

\begin{caselist}
\item Case $\mathrm{T}_i$. The conclusion is $R_i(u,u)$. Take the empty walk at $u$ and the production $i\to\eps$.
\item Case $\mathrm I(i,j)$. The rule derives $R_i(u,v)$ from $R_j(u,v)$. By the induction hypothesis, there is a walk $\omega$ from $u$ to $v$ with label $\xi$ together with a derivation $j\Rightarrow_{\Amod}^*\xi$. The production $i\to j$ extends this derivation to $i\Rightarrow_{\Amod}^*\xi$.
\item Case $\mathrm B(i,j)$. The rule derives $R_j(v,u)$ from $R_i(u,v)$. By the induction hypothesis, there is a walk $\omega$ from $u$ to $v$ with label $\xi$ and a derivation $i\Rightarrow_{\Amod}^*\xi$. Reverse the walk. Its label is $\bar\xi$. Closedness of the grammar gives $\bar i\Rightarrow_{\Amod}^*\bar\xi$. The converse of the production $\bar j\to i$ is $j\to\bar i$, which yields $j\Rightarrow_{\Amod}^*\bar\xi$.
\item Case $4(i,j,k)$. The premises are $R_j(u,w)$ and $R_k(w,v)$. By the induction hypotheses, there are walks with labels $\xi_1$ and $\xi_2$ such that $j\Rightarrow_{\Amod}^*\xi_1$ and $k\Rightarrow_{\Amod}^*\xi_2$. Concatenation gives a walk from $u$ to $v$ with label $\xi_1\xi_2$. The production $i\to jk$ and contextual closure of rewriting yield $i\Rightarrow_{\Amod}^*\xi_1\xi_2$.
\item Case $5(i,j,k)$. The premises are $R_j(w,u)$ and $R_i(w,v)$ and the conclusion is $R_k(u,v)$. By the induction hypotheses, choose a walk from $w$ to $u$ with label $\xi_1$ and a walk from $w$ to $v$ with label $\xi_2$ such that $j\Rightarrow_{\Amod}^*\xi_1$ and $i\Rightarrow_{\Amod}^*\xi_2$. Reverse the first walk and concatenate it with the second. The resulting walk goes from $u$ to $v$ and has label $\bar\xi_1\xi_2$. Closedness gives $\bar j\Rightarrow_{\Amod}^*\bar\xi_1$, while the production $k\to\bar j i$ gives $k\Rightarrow_{\Amod}^*\bar\xi_1\xi_2$. These cases exhaust the selected Horn clauses.
\end{caselist}

We now pass from certificates to Horn closure. Put $E=(E_j(\mathcal G))_{j\in\Idx}$ and $C_i=\operatorname{Cl}_{\Amod}(E)_i$. For a binary relation $R$, write $R^{-1}$ for its converse, and let $\operatorname{Id}$ be the identity relation on the position set. Define $C(\bar i)=C_i^{-1}$ and $C(\eps)=\operatorname{Id}$. For a word $l_1\cdots l_n$, put $C(l_1\cdots l_n)=C(l_1)\circ\cdots\circ C(l_n)$, with $C(i)=C_i$ for $i\in\Idx$. By Lemma~\ref{lem:least-horn-closure}, the family $(C_i)_{i\in\Idx}$ satisfies every selected Horn clause. Every grammar production $l\to\zeta$ therefore satisfies $C(\zeta)\subseteq C(l)$. For $\mathrm{T}_i$ this follows from reflexivity. For $\mathrm I(i,j)$ it is the relation inclusion $C_j\subseteq C_i$. For $\mathrm B(i,j)$, generalized symmetry gives $C_i\subseteq C_j^{-1}=C(\bar j)$. For $4(i,j,k)$, closure gives $C_j\circ C_k\subseteq C_i$. For $5(i,j,k)$, closure gives $C_j^{-1}\circ C_i\subseteq C_k$. Taking inverses gives the same property for every converse production.

Relational composition is monotone. Hence a one-step rewrite $\xi l\upsilon\Rightarrow_{\Amod}\xi\zeta\upsilon$ gives $C(\xi\zeta\upsilon)\subseteq C(\xi l\upsilon)$. We prove $i\Rightarrow_{\Amod}^*\xi\Longrightarrow C(\xi)\subseteq C_i$ by induction on the length of the semi-Thue derivation.

Base case. A derivation of length zero has $\xi=i$, and $C(i)=C_i$.

Induction step. Assume the claim for a derivation ending at a word $\chi$, and append one rewrite $\chi\Rightarrow_{\Amod}\xi$. Monotonicity for the final rewrite gives $C(\xi)\subseteq C(\chi)$, while the induction hypothesis gives $C(\chi)\subseteq C_i$. Hence $C(\xi)\subseteq C_i$.

Let a certificate use a signed walk $\omega$ from $u$ to $v$ with label $\xi$. Every forward step belongs to the corresponding $C_i$, and every reverse step belongs to $C_i^{-1}$. Therefore $(u,v)\in C(\xi)$. The preceding inclusion yields $(u,v)\in C_i$.
\end{proof}

An assignment of proof positions to worlds that respects every edge must respect every certified reachability fact.

\begin{corollary}[Semantic soundness of certificates]\label{cor:certificate-semantic}
Let $\mathcal M$ be based on a frame in $\Frames(\Amod)$, and let $\iota$ map positions of $\mathcal G$ to worlds while respecting every explicit tree edge. If $u\leadstoA{i}v$, then $\iota(u)R_i\iota(v)$.
\end{corollary}

\begin{proof}
By Theorem~\ref{thm:certificate-horn}, the certificate gives a Horn closure fact. By Lemma~\ref{lem:horn-derivational}, that fact has a Horn derivation. We use induction on its height.

Base case. A Horn derivation of height zero is justified by an explicit edge. The hypothesis that the position map respects every edge gives the corresponding accessibility fact.

Induction step. Assume the claim for the premise derivations of the root clause. Apply the induction hypotheses to those premises and then apply the corresponding frame clause from Definition~\ref{def:frame-class} to obtain the root relation fact.
\end{proof}

\subsection{Effective certificate search}\label{sec:certificate-search}

The certificate relation is defined by a finite proof object, but execution also requires an effective decision procedure for certificate existence. On a finite nested tree, the Horn presentation and the grammar presentation give two terminating procedures for this decision problem.

\begin{theorem}[Effective certificate search]\label{thm:certificate-decidability}
For a finite nested tree $\mathcal G$, positions $u,v$, and $i\in\Idx$, the judgment $u\leadstoA{i}v$ is decidable. A positive decision yields a certificate.
\end{theorem}

\begin{proof}
The finite Horn closure gives one terminating decision procedure by Lemma~\ref{lem:finite-stabilization} and Theorem~\ref{thm:certificate-horn}. Storing the Horn instance that first produced each relation fact and then following those justifications yields a certificate through the constructive first half of Theorem~\ref{thm:certificate-horn}.

The runtime procedure uses grammar reachability directly. Since every left side in $S_{\Amod}$ is one signed symbol, construct a context-free grammar with a nonterminal $N_l$ for each signed symbol $l$. For every production $l\to l_1\cdots l_n$, add $N_l\to N_{l_1}\cdots N_{l_n}$, and add the production $N_l\to l$, which terminates the grammar derivation at the terminal symbol $l$. The language generated from $N_i$ is $\{\xi\mid i\Rightarrow_{\Amod}^*\xi\}.$ The labels of signed walks from $u$ to $v$ form a regular language accepted by the finite automaton whose states are the positions of $\mathcal G$ and whose transitions are its signed steps. Emptiness of the intersection of a context-free language with a regular language is decidable~\cite{tiu2012grammar}. From a successful product derivation, one can extract a witness word, the corresponding signed walk, and a grammar derivation. We fix one deterministic implementation of this second procedure and write $\mathsf{CertSearch}_{\Amod}$ for it.
\end{proof}

For a diamond occurrence $e$ of $\Diamond_iA$ at a component $u$ and a target component $v$, call $v$ \emph{eligible for $e$} when $u\leadstoA{i}v$. The pair $(e,v)$ is then called an \emph{eligible propagation pair}.

\begin{corollary}[Finite modal choices]\label{cor:finite-modal-choices}
For every nested sequent, the eligible propagation pairs are finite in number and effectively enumerable.
\end{corollary}

\begin{proof}
There are finitely many diamond occurrences and finitely many target components. For each source occurrence and each target component, all certificates license the same formula insertion. Theorem~\ref{thm:certificate-decidability} decides each candidate pair.
\end{proof}

\subsection{Ordinary certificate calculus}\label{sec:certificate-calculus}

Let $\mathsf{Par}$ be a countably infinite effectively enumerable set of proof parameters disjoint from the set $\mathsf{Var}$ of object variables fixed in Section~\ref{sec:source-language}. The \emph{proof terms} are $t ::= x\mid a\mid c\mid f(t_1,\ldots,t_n)$, where $x\in\mathsf{Var}$ and $a\in\mathsf{Par}$. A \emph{proof core formula} is obtained from the core NNF grammar of Definition~\ref{def:core-nnf} by allowing proof terms in atomic arguments. A \emph{proof nested sequent} is obtained from Definition~\ref{def:nested-sequent} by allowing proof core formulas at its formula occurrences. Alpha equivalence extends to proof core formulas in the usual capture-avoiding way. In the rule system defined below, every nested sequent is a proof nested sequent unless the original core syntax of Definition~\ref{def:core-nnf} is explicitly intended.
A parameter introduced by a universal inference is called an \emph{eigenparameter}. A proof parameter that is not introduced by a universal inference is called a \emph{neutral proof parameter}. The rules below treat all proof parameters as rigid constants.

\begin{convention}[Rigid semantic expansion for proof parameters]\label{conv:parameter-expansion}
Whenever the proof parameters occurring in a nested sequent belong to a finite set $Q\subseteq\mathsf{Par}$, semantic evaluation uses a rigid expansion of the underlying source model in which each member of $Q$ is interpreted as an additional constant in the same domain, with an interpretation fixed across worlds. Validity of such a parameterized nested sequent means validity of its recursive interpretation under every such rigid expansion.
\end{convention}

For the syntactic counterpart, when $Q\subseteq\mathsf{Par}$ is finite, write $H_{\Amod}[Q]$ for the Hilbert system over the temporary signature obtained from $\Sigma$ by adjoining the members of $Q$ as fresh rigid constants while keeping exactly the same axiom schemes and inference rules as $H_{\Amod}$. In particular, $H_{\Amod}[\varnothing]=H_{\Amod}$. Whenever a proof core formula occurs in a derivability judgment for $H_{\Amod}[Q]$, the Hilbert interpretation fixed for core formulas above is used, with the members of $Q$ treated as temporary rigid constants.

\begin{lemma}[Conservativity for fresh constants]\label{lem:fresh-constant-conservativity}
Let $Q\subseteq Q'\subseteq\mathsf{Par}$ be finite, and let $F$ be a formula over the signature $\Sigma\cup Q$. If
$H_{\Amod}[Q']\vdash F,$
then
$H_{\Amod}[Q]\vdash F.$
\end{lemma}

\begin{proof}
It is enough to eliminate one parameter $a\in Q'\setminus Q$ at a time. Take a finite $H_{\Amod}[Q\cup\{a\}]$ proof of $F$ and choose an object variable $y$ absent from that proof after alpha conversion of bound variables. Uniformly replace the constant symbol $a$ by the free object variable $y$ throughout the proof. Every propositional, first-order, modal, Barcan, and selected modal axiom instance remains an instance of the same scheme. Modus ponens, theorem generalization, and necessitation are preserved. Since $a$ does not occur in $F$, the final formula remains $F$. The transformed proof uses only the temporary constants in $Q$. Repeating this finite elimination yields an $H_{\Amod}[Q]$ proof of $F$.
\end{proof}

Proof search applies the following rules from conclusion to premises.

\begin{definition}[Certificate rule system]\label{def:certificate-rules}
The \emph{certificate rule system} $\mathcal N_{\Amod}$ has initial sequents $\mathcal G\{\top\}$, labeled $(\top\mathrm{ax})$, and $\mathcal G\{p(\bar t),\overline p(\bar t)\}$, labeled $(\mathrm{ax})$.
The Boolean rules are
\[
\begin{array}{cc}
 \displaystyle\frac{\mathcal G\{A,B\}}{\mathcal G\{A\lor B\}}\;(\lor)
 &
 \displaystyle\frac{\mathcal G\{A\}\;\mathcal G\{B\}}
      {\mathcal G\{A\land B\}}\;(\land)
\end{array}
\]
The quantified rules are
\[
 \frac{\mathcal G\{\exists xA,A(t/x)\}}
      {\mathcal G\{\exists xA\}}\;(\exists)
\]
for any proof term $t$, and
\[
 \frac{\mathcal G\{A(a/x)\}}
      {\mathcal G\{\forall xA\}}\;(\forall)
\]
where $a\in\mathsf{Par}$ is fresh in the conclusion.

For every $i\in\Idx$, the box rule is
\[
 \frac{\mathcal G\{[A]_i\}}
      {\mathcal G\{\Box_iA\}}\;(\Box_i),
\]
where proof search creates a fresh $i$ child. Certificate propagation is
\[
 \frac{\mathcal G\{\Diamond_iA\}_{u}\{A,\Delta\}_{v}}
      {\mathcal G\{\Diamond_iA\}_{u}\{\Delta\}_{v}}\;(\Diamond_i^{\mathrm{cert}}),
\]
with an attached certificate for $u\leadstoA{i}v$. The premise keeps the principal diamond occurrence. The endpoints may coincide under Convention~\ref{conv:coincident-endpoints}. A rule is called \emph{retaining} when its principal formula occurrence also occurs unchanged in its premise, and an occurrence preserved in this way is called \emph{retained}. The existential rule and certificate propagation are retaining. Their principal occurrences may therefore be used again with another proof term or another certified target. For every selected $\mathrm{D}_i\in\Amod$, seriality uses
\[
 \frac{\mathcal G\{[\varnothing]_i\}}
      {\mathcal G\{\varnothing\}}\;(\mathrm d_i),
\]
which creates a fresh empty $i$ child during backward search. The calculus has no cut rule.
\end{definition}

The ordinary certificate calculus $\mathcal N_{\Amod}$ operates on proof core formulas. When $C$ is a proof core formula, the notation $\mathcal N_{\Amod}\vdash C$ abbreviates derivability of the proof nested sequent consisting of a single root component containing only $C$. We call its derivations \emph{ordinary certificate derivations}, or simply ordinary derivations when the calculus is clear.

\subsection{Certificate persistence}\label{sec:certificate-persistence}

The ordinary certificate calculus extends the nested tree only by adding fresh labeled child edges. Existing certificates remain valid under these edge extensions.

\begin{lemma}[Tree monotonicity]\label{lem:tree-monotonicity}
Every nonterminal rule of the ordinary certificate calculus either leaves the explicit nested tree unchanged or adds one fresh labeled child edge. No rule removes, relabels, or moves an existing edge.
\end{lemma}

\begin{proof}
Boolean and quantifier rules change formulas. Certificate propagation adds a formula at an existing component. The box rule and the seriality rule each add one fresh child. These are all nonterminal rule families of the kernel.
\end{proof}

\begin{corollary}[Certificate persistence]\label{cor:certificate-persistence}
A certificate valid at one stage of a proof branch remains valid at every descendant stage.
\end{corollary}

\begin{proof}
Every edge used by the signed walk remains present with the same label and endpoints by Lemma~\ref{lem:tree-monotonicity}. The grammar $S_{\Amod}$ does not change along the branch.
\end{proof}

Certificate persistence concerns such edge extensions. The answer theory also changes first-order terms without changing the tree. The next lemma states the independent invariance needed for substitution and abstraction.

\begin{lemma}[Term invariance of certificates]\label{lem:certificate-term-invariance}
Let $(\omega,\mathfrak d)$ certify $u\leadstoA{i}v$ in a nested sequent $\mathcal G$. Suppose $\vartheta$ changes only first-order terms inside formulas and preserves formula constructors, positions, edges between parents and children, modal indices, and edge labels. Then the same pair $(\omega,\mathfrak d)$ certifies the corresponding reachability judgment in $\vartheta(\mathcal G)$.
\end{lemma}

\begin{proof}
The signed walk $\omega$ uses only positions and labeled tree edges, all of which are preserved by $\vartheta$. Its label is therefore unchanged. The grammar $S_{\Amod}$ contains no first-order term data, so the same grammar derivation $\mathfrak d$ remains valid.
\end{proof}

For every finite derivation $\mathcal D$ considered below, its \emph{height} $h(\mathcal D)$ is the maximum number of inference steps on a root-to-leaf path. A derivation consisting of an initial sequent has height zero.

The quantified soundness argument uses the substitution lemmas in Appendix~\ref{app:substitution}. Nested sequent validity can be expressed either through the recursive formula interpretation or through an assignment of worlds to positions.

\begin{definition}[Validity by position maps]\label{def:position-map-validity}
Let $\Gamma_u$ be the formula multiset at component $u$. A \emph{position map} is a map $\iota$ from positions to worlds, and such a map is \emph{edge respecting} when every explicit edge $u\xrightarrow{i}v$ satisfies $\iota(u)R_i\iota(v)$. A nested sequent is false under $\mathcal M,g,\iota$ when, for every position $u$, every $A\in\Gamma_u$ is false at $\iota(u)$ under $g$.
\end{definition}

\begin{lemma}[Semantics of position maps]\label{lem:position-map-semantics}
Let $u_0$ be the root position of a nested sequent $\mathcal G$. For every model $\mathcal M$, assignment $g$, and world $w$,
$\mathcal M,w,g\not\models\sem{\mathcal G}$
if and only if there is a position map $\iota$ with $\iota(u_0)=w$ that respects every edge and makes every formula occurrence false at its component.
\end{lemma}

\begin{proof}
We use induction on the number of bracket occurrences.

Base case. With no brackets, $\sem{\mathcal G}$ is the disjunction of the root formulas. Its falsity is therefore equivalent to assigning the root to $w$ and falsifying every root formula.

Induction step. Write the root as $A_1,\ldots,A_n,[\Delta_1]_{i_1},\ldots,[\Delta_r]_{i_r}$ and assume the result for each child subtree.

\begin{caselist}
\item Case from semantic falsity to a position map. If the interpretation is false at $w$, every root formula is false there, and for each child $[\Delta]_i$ the formula $\Box_i\sem{\Delta}$ is false. Choose an $R_i$ successor $v$ of $w$ at which $\sem{\Delta}$ is false. The induction hypothesis gives a position map that respects every edge on each child subtree. The child position sets are disjoint, so these maps combine with the assignment of $w$ to the root. This combined map falsifies every formula occurrence.
\item Case from a position map to semantic falsity. Let a position map that respects every edge falsify every formula occurrence. Restrict it to each child subtree $[\Delta]_i$. The induction hypothesis makes $\sem{\Delta}$ false at the child world. The explicit parent edge makes that world an $R_i$ successor of $w$, so the box disjunct is false at $w$. Every root formula is also false, and therefore the whole root disjunction is false.
\end{caselist}
\end{proof}

The preceding lemma reduces rule soundness to a local counterinterpretation argument. For each backward rule step, a falsifying map for the conclusion can be extended to at least one premise.

\begin{theorem}[Rule soundness]\label{thm:kernel-rule-soundness}
Every rule of $\mathcal N_{\Amod}$ preserves validity over constant-domain models with rigid terms based on $\Frames(\Amod)$.
\end{theorem}

\begin{proof}
It suffices to extend a counterinterpretation of a conclusion to a counterinterpretation of at least one premise.

A component containing $\top$ cannot be false, and opposite atomic literals cannot both be false. If $A\lor B$ is false, both disjuncts are false. If $A\land B$ is false, one conjunct is false, so the appropriate premise remains false.

If $\exists xA$ is false, then $A$ is false under every value of $x$, including the denotation of the chosen proof term $t$. The substitution lemma in Lemma~\ref{lem:formula-substitution} gives falsity of $A(t/x)$. If $\forall xA$ is false, choose a domain element that witnesses the failure, interpret the fresh parameter $a$ rigidly as that element, and apply the same substitution lemma to $A(a/x)$.

If $\Box_iA$ is false at the active component, choose an $R_i$ successor where $A$ is false and map the fresh $i$ child to that world. If a certificate diamond conclusion is false, Corollary~\ref{cor:certificate-semantic} gives $\iota(u)R_i\iota(v)$. Falsity of $\Diamond_iA$ at $\iota(u)$ makes $A$ false at $\iota(v)$, so insertion of $A$ at the target preserves the counterinterpretation. Finally, if $\mathrm{D}_i$ is selected, an $R_i$ successor exists by seriality. Map the fresh empty child to that successor. No formula is added in that child.
\end{proof}

\begin{corollary}[Kernel soundness]\label{cor:kernel-soundness}
If $\mathcal N_{\Amod}\vdash\mathcal G$, then $\mathcal G$ is valid in every constant-domain model with rigid terms based on $\Frames(\Amod)$.
\end{corollary}

\begin{proof}
We use induction on derivation height.

Base case. A derivation of height zero is an initial sequent, which is valid by the truth or atomic clause of Theorem~\ref{thm:kernel-rule-soundness}.

Induction step. Assume every premise of the final inference is valid by the induction hypothesis. The corresponding clause of Theorem~\ref{thm:kernel-rule-soundness} preserves validity from the premises to the conclusion.
\end{proof}

\subsection{Fair saturation}\label{sec:fair-saturation}

Completeness uses a saturation schedule in which bookkeeping information is retained separately on each branch. Fix an effective enumeration $t_0,t_1,\ldots$ of the proof terms, together with effective enumerations of proof parameters, modal indices, component names, and finite tuples of syntactic data. Fix also an effective enumeration of a countably infinite supply of names for persistent formula occurrences, disjoint from the previously fixed syntactic name supplies. A name from this supply is called a \emph{persistent occurrence name}. For the saturation construction only, every formula occurrence is assigned a persistent occurrence name from this supply. The occurrences in the root nested sequent receive distinct names, every occurrence preserved by a rule keeps its name, and every newly created occurrence receives a fresh name. These annotations are part of the saturation bookkeeping only and are ignored by the certificate calculus and its semantics. The completeness argument depends on the scheduler eventually processing every obligation that remains permanently available, on finite branching of the search tree, and on the existence of an infinite fair branch whenever no finite proof exists. The next lemmas establish these properties in that order.

\begin{definition}[Saturation configuration]\label{def:saturation-configuration}
A \emph{saturation configuration} is a triple $(\mathcal G,H,N)$ consisting of a nested sequent $\mathcal G$, a finite set $H\subseteq\mathbb N$, and a scheduler stage $N\in\mathbb N$. It is \emph{open} when $\mathcal G$ is not an initial sequent of the certificate calculus.
\end{definition}

\begin{definition}[Saturation obligation]\label{def:saturation-obligation}
At an open saturation configuration $(\mathcal G,H,N)$, a \emph{saturation obligation} has one of the following forms.
\begin{enumerate}
\item An occurrence of $A\lor B$.
\item An occurrence of $A\land B$.
\item An occurrence of $\forall xA$.
\item A pair consisting of an occurrence of $\exists xA$ and an index $n$ requesting the predetermined proof term $t_n$.
\item An occurrence of $\Box_iA$.
\item A tuple $(e,i,u,v)$ where $e$ is an occurrence of $\Diamond_iA$ at $u$ and $u\leadstoA{i}v$ holds in the underlying nested sequent $\mathcal G$.
\item A pair $(i,u)$ for selected $\mathrm{D}_i$ requesting one explicit $i$ successor of $u$.
\end{enumerate}
Every possible obligation has a distinct natural number code. A code is processed exactly when it belongs to $H$. At scheduler stage $N$, the search processes the least currently available code at most $N$ that does not belong to $H$, or performs a bookkeeping step when no such code is available. Processing an obligation adds its code to $H$, while a bookkeeping step leaves $H$ unchanged. In either case, the next scheduler stage is $N+1$. Whenever processing an obligation requires a fresh proof parameter, component name, or persistent occurrence name, the scheduler chooses the first name in the corresponding fixed effective enumeration that is unused in the current saturation configuration and its occurrence annotations. Conjunction copies the resulting set $H$ to both branches. Certificate obligations use the fixed procedure $\mathsf{CertSearch}_{\Amod}$. For a root nested sequent $\mathcal G_0$, the saturation construction starts from $(\mathcal G_0,\varnothing,0)$, and the saturation tree consists of the configurations generated by these logical and bookkeeping transitions.
\end{definition}

\begin{lemma}[Saturation fairness]\label{lem:saturation-fairness}
Every obligation that becomes permanently available on an infinite search branch is eventually processed on that branch.
\end{lemma}

\begin{proof}
Let the obligation have code $n$, and choose a stage after which all syntactic data that define it remain present. Certificate obligations have this persistence by Corollary~\ref{cor:certificate-persistence}. Once the scheduler index reaches $n$, postponement can occur only through currently available unprocessed obligations with smaller codes. There are finitely many such codes. Once a processed code is added to $H$, it remains there, and conjunction copies $H$ to both branches. Hence only finitely many lower codes can precede the chosen obligation.
\end{proof}

\begin{lemma}[Finite branching of saturation]\label{lem:saturation-finite-branching}
The saturation tree is finitely branching.
\end{lemma}

\begin{proof}
The scheduler selects at most one obligation at each saturation configuration. Every selected rule has one premise except conjunction, which has two. Fresh child names are canonical. A source diamond occurrence and a target component determine one successor sequent, independently of the number of certificates for that pair. Bookkeeping nodes have one successor.
\end{proof}

\begin{lemma}[Infinite fair open branch]\label{lem:infinite-open-branch}
If a nested sequent $\mathcal G_0$ has no finite certificate proof, the saturation construction rooted at $\mathcal G_0$ has an infinite open branch that satisfies Lemma~\ref{lem:saturation-fairness}.
\end{lemma}

\begin{proof}
Every open saturation configuration has a logical or bookkeeping successor. If the search tree had finite height, finite branching from Lemma~\ref{lem:saturation-finite-branching} would make it finite. Every maximal branch would then terminate in an initial sequent, and erasing scheduler data would give a finite proof of the root. This contradicts the assumption. The tree therefore has arbitrarily large finite depth. By K\"onig's lemma, it has an infinite branch. Initial sequents cannot occur on such a branch, and the branch is fair by Lemma~\ref{lem:saturation-fairness}.
\end{proof}

\subsection{Canonical countermodel}\label{sec:canonical-countermodel}

Fix an infinite open fair branch $\mathcal B$. For every component name $u$, let $\Gamma_u$ be the union of the formulas that occur at $u$ along the branch. Let $W$ be the set of component names on the branch, and let $E_i(\mathcal B)$ be the union of the explicit $i$ edges. Define
$(R_i)_{i\in\Idx} =\operatorname{Cl}_{\Amod} ((E_i(\mathcal B))_{i\in\Idx}),$
where Lemma~\ref{lem:horn-derivational} shows that every membership fact has a Horn derivation whose branch edges all occur by some stage. Every selected seriality obligation is eventually processed, so each selected $R_i$ is serial. The other selected Horn clauses hold by Lemma~\ref{lem:least-horn-closure}. Hence $(W,(R_i))$ belongs to $\Frames(\Amod)$.

For the canonical construction, use the simultaneous rigid expansion by all proof parameters in $\mathsf{Par}$. Its restriction to any finite $Q\subseteq\mathsf{Par}$ is an expansion in the sense of Convention~\ref{conv:parameter-expansion}. Take the constant domain $D$ to be the set of all proof terms. Define the canonical assignment by $g_0(x)=x$. Interpret every source constant and every proof parameter by itself. Interpret each $n$-ary function symbol $f$ by the function that maps $(t_1,\ldots,t_n)$ to the proof term $f(t_1,\ldots,t_n)$. For every world $u$ and every $n$-ary predicate symbol $p$, define
\[
\mathcal J_u(p)=\{(t_1,\ldots,t_n)\in D^n\mid \overline p(t_1,\ldots,t_n)\in\Gamma_u\}.
\]
Let
\[
\mathcal M=(W,(R_i)_{i\in\Idx},D,\mathcal J)
\]
be the resulting model. Openness ensures that a component never contains opposite literals with the same tuple.

\begin{lemma}[Canonical term lemma]\label{lem:canonical-term}
Every proof term $t$ satisfies $\sem{t}^{\mathcal M}_{g_0}=t$.
\end{lemma}

\begin{proof}
We use structural induction on $t$.

Base case. Variables use $g_0(x)=x$. Constants and proof parameters use their canonical interpretations.

Induction step. Suppose $t=f(t_1,\ldots,t_n)$. Apply the induction hypotheses to the arguments and then the defining interpretation of $f$.
\end{proof}

The remaining task is to show that every formula belonging to some $\Gamma_u$ is false in this canonical model. Fairness first converts the scheduler obligations into the local closure conditions needed for the induction.

\begin{lemma}[Saturation consequences]\label{lem:saturation-consequences}
For every $u\in W$, the following properties hold.
\begin{enumerate}
\item $\top\notin\Gamma_u$.
\item If $A\lor B\in\Gamma_u$, then $A,B\in\Gamma_u$.
\item If $A\land B\in\Gamma_u$, then at least one of $A,B$ belongs to $\Gamma_u$.
\item If $\forall xA\in\Gamma_u$, then $A(a/x)\in\Gamma_u$ for some fresh parameter $a$.
\item If $\exists xA\in\Gamma_u$, then $A(t/x)\in\Gamma_u$ for every proof term $t$.
\item If $\Box_iA\in\Gamma_u$, an explicit $i$ child $v$ satisfies $uR_iv$ and $A\in\Gamma_v$.
\item If $\Diamond_iA\in\Gamma_u$ and $uR_iv$, then $A\in\Gamma_v$.
\item If $\mathrm{D}_i\in\Amod$, then some $v$ satisfies $uR_iv$.
\end{enumerate}
\end{lemma}

\begin{proof}
Clause 1 follows from openness. If $\top$ belonged to some $\Gamma_u$, then it would already occur at a finite stage of the branch, where the underlying nested sequent would be a $(\top\mathrm{ax})$ initial sequent. Clauses 2 through 4 and clause 6 follow from fairness of their corresponding obligations, and clause 8 follows from fairness of seriality obligations. For clause 5, fix a proof term $t=t_n$. Once the existential occurrence appears, it remains present because the existential rule is retaining. Hence the obligation for that occurrence and index $n$ remains permanently available, so fairness eventually inserts $A(t/x)$. For clause 7, $uR_iv$ has a finite Horn derivation by Lemma~\ref{lem:horn-derivational}. Its finitely many branch edges already occur at some finite stage. The constructive direction of Theorem~\ref{thm:certificate-horn} yields a certificate at that stage, and Corollary~\ref{cor:certificate-persistence} preserves it thereafter. Once this diamond occurrence appears on the branch, it remains present because certificate propagation is retaining and all other rules leave it as passive context. The corresponding certificate obligation is therefore permanently available. Fairness eventually processes it and inserts $A$ at $v$.
\end{proof}

\begin{theorem}[Truth lemma]\label{thm:truth-lemma}
For every $u\in W$ and every $A\in\Gamma_u$,
$\mathcal M,u,g_0\not\models A.$
\end{theorem}

\begin{proof}
We use induction on formula complexity.

Base case. The case $\bot$ follows from its semantics, while $\top$ is excluded by Lemma~\ref{lem:saturation-consequences}. If $p(\bar t)\in\Gamma_u$, openness gives $\overline p(\bar t)\notin\Gamma_u$, so the predicate definition and Lemma~\ref{lem:canonical-term} make the positive atom false. If $\overline p(\bar t)\in\Gamma_u$, the predicate definition makes the positive atom true, so the negative literal is false.

Induction step. Assume the claim for formulas of smaller complexity.

\begin{caselist}
\item Case $A=B\lor C$. Both disjuncts occur by Lemma~\ref{lem:saturation-consequences} and are false by the induction hypotheses.
\item Case $A=B\land C$. At least one conjunct occurs by Lemma~\ref{lem:saturation-consequences} and is false by the corresponding induction hypothesis.
\item Case $A=\exists xB$. An arbitrary domain element is a proof term $t$. By Lemma~\ref{lem:saturation-consequences}, $B(t/x)\in\Gamma_u$. The induction hypothesis, Lemma~\ref{lem:canonical-term}, and the substitution lemma make $B$ false under $x\mapsto t$. Since the domain element was arbitrary, the existential is false.
\item Case $A=\forall xB$. By Lemma~\ref{lem:saturation-consequences}, $B(a/x)\in\Gamma_u$ for some parameter $a$. The induction hypothesis and substitution make $a$ a witness to falsity of the universal.
\item Case $A=\Box_iB$. By Lemma~\ref{lem:saturation-consequences}, there is an $R_i$ successor $v$ with $B\in\Gamma_v$. The induction hypothesis makes $B$ false at $v$.
\item Case $A=\Diamond_iB$. Take any $v$ with $uR_iv$. By Lemma~\ref{lem:saturation-consequences}, $B\in\Gamma_v$. The induction hypothesis makes $B$ false there. Every $R_i$ successor therefore falsifies $B$.
\end{caselist}
\end{proof}

By the truth lemma, every infinite fair open branch yields a countermodel of its root. Completeness therefore follows by contraposition from the construction of an infinite branch.

\begin{theorem}[Cut-free semantic completeness]\label{thm:kernel-completeness}
If an indexed nested sequent $\mathcal G$ is valid in every constant-domain model with rigid terms based on $\Frames(\Amod)$, then
$\mathcal N_{\Amod}\vdash\mathcal G.$
\end{theorem}

\begin{proof}
Assume contrapositively that $\mathcal G$ has no finite certificate proof. Lemma~\ref{lem:infinite-open-branch} gives an infinite open fair branch. Construct the canonical model above. Its relation family satisfies all selected Horn clauses and seriality requirements. Its domain is constant, and its constants and function symbols are rigid free constructors.

Map every component of the initial sequent to the canonical world with the same position name. Every initial explicit edge belongs to its canonical relation. Every initial formula belongs to the associated $\Gamma_u$, and Theorem~\ref{thm:truth-lemma} makes every such formula false. Thus Lemma~\ref{lem:position-map-semantics} gives a countermodel to $\mathcal G$. Validity therefore entails a finite derivation. Since the calculus contains no cut rule, the resulting proof is cut-free without a separate cut elimination transformation.
\end{proof}

\subsection{Substitution stability}\label{sec:kernel-substitution}

The answer theory transforms terms while preserving the nested tree. The following two stability results clarify the ordinary calculus. The later regularization argument uses renaming of proof parameters, while substitution admissibility establishes closure of ordinary derivability under capture-avoiding instantiation of free object variables.

\begin{lemma}[Uniform renaming of proof parameters]\label{lem:uniform-parameter-renaming}
Let $\mathcal D$ be a finite derivation in $\mathcal N_{\Amod}$. Let $a\in\mathsf{Par}$ occur in $\mathcal D$, and let $b\in\mathsf{Par}$ occur nowhere in $\mathcal D$. Uniform replacement of $a$ by $b$ in all formulas and witness terms of $\mathcal D$ produces another legal derivation of the root obtained by the same renaming.
\end{lemma}

\begin{proof}
We use induction on the height of $\mathcal D$.

Base case. A derivation of height zero is a truth or atomic initial sequent. Uniform replacement preserves truth closure and the matching argument tuples in an atomic identity leaf.

Induction step. Assume the claim for the premise derivations of the final inference.

\begin{caselist}
\item Case Boolean rule. The rule commutes with uniform replacement, so the induction hypotheses reconstruct the renamed premises and the same rule applies.
\item Case existential rule. Replacement commutes with substitution for object variables because proof parameters and object variables are disjoint syntactic classes. The renamed witness is again a proof term, and the induction hypothesis reconstructs the premise.
\item Case universal rule. Let $a_0$ be its eigenparameter. If $a_0\neq a$, the same $a_0$ remains fresh after replacement because $b$ occurred nowhere in the original derivation. If $a_0=a$, the renamed rule uses $b$. The original conclusion contained no $a$ by freshness and contained no $b$ by choice, so $b$ is fresh in the renamed conclusion. The induction hypothesis reconstructs the premise.
\item Case box or seriality rule. These rules inspect only nested shape, which uniform replacement preserves. Apply the induction hypothesis to each premise and reapply the rule.
\item Case certificate propagation. Certificate propagation remains legal by Lemma~\ref{lem:certificate-term-invariance}. Apply the induction hypothesis to the premise and reapply the same propagation rule.
\end{caselist}
\end{proof}

\begin{proposition}[Substitution admissibility]\label{prop:free-substitution-admissibility}
If
$\mathcal N_{\Amod}\vdash\mathcal G$
and $\theta$ is a capture-avoiding substitution on free object variables by proof terms, then
$\mathcal N_{\Amod}\vdash\mathcal G\theta.$
\end{proposition}

\begin{proof}
We use induction on the height of a derivation of $\mathcal G$.

Base case. Initial sequents are preserved by substitution.

Induction step. Assume the claim for the premise derivations of the final inference.

\begin{caselist}
\item Case Boolean rule. Substitution commutes with the Boolean constructor. Apply the induction hypotheses to the premises and reapply the same rule.
\item Case existential rule. Let $t$ be the witness and standardize bound variables apart from $\theta$. Then $(A(t/x))\theta=(A\theta)(t\theta/x)$, and $t\theta$ is again a proof term. Apply the induction hypothesis to the premise and reapply the existential rule.
\item Case universal rule. Let $a$ be the eigenparameter. If $a$ occurs in the range of $\theta$, choose a proof parameter $b$ absent from the finite derivation and the finite substitution range and apply Lemma~\ref{lem:uniform-parameter-renaming} first. We may therefore assume that the chosen eigenparameter is absent from the range of $\theta$. It remains absent from the substituted conclusion and can be reused in the final universal inference.
\item Case box or seriality rule. These rules preserve their nested shape. Apply the induction hypothesis to each premise and reapply the rule.
\item Case certificate propagation. Substitution for free variables changes no position, edge, or modal label by Lemma~\ref{lem:certificate-term-invariance}. Apply the induction hypothesis to the premise and reapply the propagation rule.
\end{caselist}
\end{proof}

\subsection{Hilbert adequacy}\label{sec:hilbert-certificate-adequacy}

The semantic completeness theorem is sufficient for one direction of the Hilbert comparison. The reverse direction requires Hilbert certification of each certificate rule. The longer context lemmas and the propositional canonical argument appear in Appendix~\ref{app:hilbert-certification}.

\begin{theorem}[Hilbert soundness]\label{thm:hilbert-soundness}
If $H_{\Amod}\vdash F$, then $F$ is valid in every constant-domain model with rigid terms based on $\Frames(\Amod)$.
\end{theorem}

\begin{proof}
The classical propositional and first-order schemes are sound under standard constant-domain semantics. Modus ponens and theorem generalization preserve validity. The normal modal axiom $\mathrm{K}_i$ is valid for every binary relation, and theorem necessitation preserves global validity. Modal duality follows from the semantic clauses.

For the Barcan formula, fix a model $\mathcal M$, a world $w$, and an assignment $g$, and assume $\mathcal M,w,g\models\forall x\Box_iA$. Fix $wR_iv$ and $d\in D$. The antecedent gives $\mathcal M,w,g[x\mapsto d]\models\Box_iA$, so $\mathcal M,v,g[x\mapsto d]\models A$. Since $d$ was arbitrary and the domain is constant, $\mathcal M,v,g\models\forall xA$. Since $v$ was arbitrary, $\mathcal M,w,g\models\Box_i\forall xA$. For the converse Barcan formula, assume $\mathcal M,w,g\models\Box_i\forall xA$, fix $d\in D$, and take any $wR_iv$. Then $\mathcal M,v,g[x\mapsto d]\models A$, so $\mathcal M,w,g[x\mapsto d]\models\Box_iA$. Since $d$ was arbitrary, $\mathcal M,w,g\models\forall x\Box_iA$.

The selected modal schemes are sound by Proposition~\ref{prop:modal-correspondence}. Every Hilbert derivation therefore preserves validity.
\end{proof}

This establishes Hilbert soundness. For the converse direction from certificate proofs to Hilbert theoremhood, it is enough to certify the modal propagation rule and the remaining primitive kernel rules inside arbitrary nested contexts.

\begin{lemma}[Hilbert certification of certificate propagation]\label{lem:certificate-hilbert}
Let $Q\subseteq\mathsf{Par}$ be finite and contain every proof parameter occurring in an instance of $\Diamond_i^{\mathrm{cert}}$. The instance preserves $H_{\Amod}[Q]$ theoremhood of nested interpretations.
\end{lemma}

\begin{proof}
The complete propositional abstraction argument, uniformly relative to the temporary constants in $Q$, is given in Appendix~\ref{app:certificate-hilbert-proof}.
\end{proof}

\begin{lemma}[Hilbert certification of kernel rules]\label{lem:remaining-rule-hilbert}
Let $Q\subseteq\mathsf{Par}$ be finite and contain every proof parameter occurring in an instance of a primitive kernel rule. Truth, atomic identity, the Boolean rules, the first-order rules, box, and selected seriality preserve $H_{\Amod}[Q]$ theoremhood of nested interpretations.
\end{lemma}

\begin{proof}
The complete verification for deep contexts, including elimination of a fresh universal eigenparameter, is given in Appendix~\ref{app:remaining-hilbert-proof}.
\end{proof}

\begin{corollary}[Certificate to Hilbert theoremhood]\label{cor:certificate-to-hilbert}
Let $\mathcal D$ be a finite $\mathcal N_{\Amod}$ derivation of $\mathcal G$, and let $Q_{\mathcal D}$ be the finite set of proof parameters occurring anywhere in $\mathcal D$. Then
$H_{\Amod}[Q_{\mathcal D}]\vdash\sem{\mathcal G}.$
If the root $\mathcal G$ contains no proof parameter, then
$H_{\Amod}\vdash\sem{\mathcal G}.$
\end{corollary}

\begin{proof}
For the first statement, we use induction on derivation height inside the single temporary signature $\Sigma\cup Q_{\mathcal D}$. All proof parameters in every rule instance belong to $Q_{\mathcal D}$, so every intermediate theoremhood judgment is well typed in $H_{\Amod}[Q_{\mathcal D}]$.

Base case. A derivation of height zero is an initial sequent. Its Hilbert theoremhood follows from the corresponding clause for initial sequents in Lemma~\ref{lem:remaining-rule-hilbert}.

Induction step. Assume the interpretations of all premises are theorems of $H_{\Amod}[Q_{\mathcal D}]$. If the final rule is certificate propagation, apply Lemma~\ref{lem:certificate-hilbert}. For every other primitive kernel rule, apply Lemma~\ref{lem:remaining-rule-hilbert}. This yields theoremhood of the conclusion interpretation.

If the root contains no proof parameter, its interpretation is a formula over $\Sigma$, and Lemma~\ref{lem:fresh-constant-conservativity} eliminates the finitely many temporary constants in $Q_{\mathcal D}$.
\end{proof}

The preceding lemmas show that every certificate derivation yields Hilbert theoremhood in a temporary rigid expansion. For a root without proof parameters, Lemma~\ref{lem:fresh-constant-conservativity} then yields theoremhood in the source Hilbert system. Together with semantic completeness, these results establish the equivalence of the independently defined Hilbert and certificate systems.

\begin{theorem}[Full Hilbert--certificate adequacy]\label{thm:full-adequacy}
For every source formula $F$,
\[
 H_{\Amod}\vdash F
 \Longleftrightarrow
 \mathcal N_{\Amod}\vdash\NNF^{+}(F).
\]
\end{theorem}

\begin{proof}
Assume first $H_{\Amod}\vdash F$. By Theorem~\ref{thm:hilbert-soundness}, $F$ is valid on $\Frames(\Amod)$. Theorem~\ref{thm:normalization-adequacy} gives semantic equivalence between $F$ and $\NNF^+(F)$, and Theorem~\ref{thm:kernel-completeness} gives a finite certificate proof of the latter.

Conversely, assume the certificate derivation. Its root $\NNF^+(F)$ contains no proof parameter, so the second statement of Corollary~\ref{cor:certificate-to-hilbert} gives $H_{\Amod}\vdash\NNF^+(F)$. Theorem~\ref{thm:hilbert-nnf} gives $H_{\Amod}\vdash F\leftrightarrow\NNF^+(F)$, from which $H_{\Amod}\vdash F$ follows by classical Hilbert reasoning.
\end{proof}

\begin{corollary}[Adequacy equivalences]\label{cor:three-way-adequacy}
For every source formula $F$,
\[
 H_{\Amod}\vdash F
 \Longleftrightarrow
 \mathcal N_{\Amod}\vdash\NNF^+(F)
 \Longleftrightarrow
 \Frames(\Amod)\models F.
\]
\end{corollary}

\begin{proof}
Combine Theorems~\ref{thm:full-adequacy}, \ref{thm:hilbert-soundness}, \ref{thm:kernel-completeness}, and~\ref{thm:normalization-adequacy}.
\end{proof}

For a user answer substitution $\theta$, Theorem~\ref{thm:full-adequacy} also makes the Hilbert correctness condition in Definition~\ref{def:hilbert-correct} equivalent to
$\mathcal N_{\Amod}\vdash \NNF^+\left(\widehat P\to \forall Y_1\cdots\forall Y_r\,G\theta\right),$
where $Y_1,\ldots,Y_r$ are the free variables of $G\theta$. We retain $\Corr^{\mathrm H}_{\Amod}(P,G)$ as the single notation for correct answers.

\subsection{Raw structural boundary}\label{sec:raw-boundary}

For comparison, we now define a raw structural presentation in which some rules replace bracket labels. Such replacements can destroy the information that the same component is reachable through more than one modal relation. Its local soundness is proved in Appendix~\ref{app:raw-structural}.

\begin{definition}[Raw calculus]\label{def:raw-calculus}
For a finite modal module $\Amod$, the \emph{raw calculus} $\mathcal R_{\Amod}$ has the same initial sequents, Boolean rules, quantified rules, and box rule as $\mathcal N_{\Amod}$. Replace certificate propagation by immediate diamond propagation that retains its principal occurrence
\[
 \frac{\mathcal G\{\Diamond_iA\}_{u}\{A,\Delta\}_{v}}
      {\mathcal G\{\Diamond_iA\}_{u}\{\Delta\}_{v}}\;(\Diamond_i^{\mathrm r}),
\]
where $v$ is an immediate $i$ child of $u$. The calculus also contains the selected structural rules displayed below. It has no cut rule. All rules are applied bottom up during proof search.
\end{definition}

For an arbitrary deep one-hole context $\mathcal G\{\}$, the selected structural rules are the following.
\[
\begin{array}{cc}
 \displaystyle\frac{\mathcal G\{[\Delta]_i\}}{\mathcal G\{\Delta\}}\;[\mathrm{t}_i]
 &
 \displaystyle\frac{\mathcal G\{[\varnothing]_i\}}{\mathcal G\{\varnothing\}}\;[\mathrm{d}_i]
\end{array}
\]
\[
 \frac{\mathcal G\{[\Delta]_i\}}{\mathcal G\{[\Delta]_j\}}\;[\mathrm I(i,j)],
\]
\[
 \frac{\mathcal G\{[\Delta,[\Lambda]_j]_i\}}
      {\mathcal G\{[\Delta]_i,\Lambda\}}\;[\mathrm b(i,j)],
\]
\[
 \frac{\mathcal G\{[\Delta]_i,[\Lambda]_j\}}
      {\mathcal G\{[[\Delta]_k,\Lambda]_j\}}\;[4(i,j,k)],
\]
\[
 \frac{\mathcal G\{[[\Delta]_k,\Lambda]_j\}}
      {\mathcal G\{[\Delta]_i,[\Lambda]_j\}}\;[5(i,j,k)].
\]
Each structural rule is available exactly when its modal scheme belongs to $\Amod$.

\begin{proposition}[Semantic soundness of raw structural rules]\label{prop:raw-semantic-soundness}
Every selected raw structural rule preserves validity over $\Frames(\Amod)$.
\end{proposition}

\begin{proof}
The detailed counterinterpretation proof is given in Appendix~\ref{app:raw-structural}.
\end{proof}

The following finite modal module yields a concrete separation.

\begin{theorem}[Counterexample with a shared successor]\label{thm:fanout-counterexample}
Let
$\mathcal B=\{\mathrm{D}_3,\mathrm I(1,3),\mathrm I(2,3)\}.$
Then
$H_{\mathcal B}\vdash\Diamond_1p\lor\Diamond_2\neg p,$
while $\mathcal R_{\mathcal B}$ does not derive the one-sided root
$\{\Diamond_1p,\Diamond_2\overline p\}.$
\end{theorem}

\begin{proof}
For Hilbert derivability, $\mathrm I(1,3)$ instantiated with $\neg p$ gives $\Box_1\neg p\to\Box_3\neg p,$ and $\mathrm{D}_3$ gives $\Box_3\neg p\to\Diamond_3\neg p.$ From $\mathrm I(2,3)\colon\Box_2p\to\Box_3p$, contraposition and modal duality give $\Diamond_3\neg p\to\Diamond_2\neg p.$ Hence $\Box_1\neg p\to\Diamond_2\neg p$. Since $\Box_1\neg p$ is equivalent to $\neg\Diamond_1p$, classical reasoning yields the displayed disjunction.

For raw nonderivability, consider bottom-up proof search from $S=\{\Diamond_1p,\Diamond_2\overline p\}.$ No Boolean, quantifier, or box rule applies. The only applicable rule that can create a component is $\mathrm{d}_3$, which creates an empty $3$ child. The inclusion rules can relabel a displayed $3$ component to $1$ or to $2$. Once a component has label $1$ or $2$, no selected raw rule returns it to $3$ or changes one of those two labels to the other. The immediate diamond rule can insert $p$ only in a $1$ component immediately below the root and can insert $\overline p$ only in a $2$ component immediately below the root.

Consequently every raw sequent reachable from the root satisfies the following invariant. Each component containing $p$ has incoming label $1$, each component containing $\overline p$ has incoming label $2$, and no component contains both literals. The invariant is true at the root. The seriality rule adds only an empty component, inclusion changes one label in an allowed direction, and each diamond rule inserts its literal only in a component with the corresponding label. No truth formula is generated, and the invariant prevents atomic identity. Hence no raw branch closes.
\end{proof}

The raw inclusion rules may fail to preserve a single shared successor because they replace bracket labels. Certificate propagation instead keeps the explicit child fixed and certifies both derived reachability facts for that same child. The next theorem proves that this treatment is sufficient and also establishes admissibility of the raw transformations.

\begin{theorem}[Raw rule admissibility]\label{thm:raw-rule-admissibility}
The certificate calculus proves the root in Theorem~\ref{thm:fanout-counterexample} without cut. Every selected structural rule of $\mathcal R_{\Amod}$ is admissible for ordinary certificate derivability.
\end{theorem}

\begin{proof}
Apply $\mathrm{d}_3$ once to create an empty $3$ child $v$ of the root $u$. The grammar contains $1\to3$ from $\mathrm I(1,3)$, so the explicit $3$ edge gives a certificate from $u$ to $v$ for modality $1$. Certificate propagation inserts $p$ at $v$. The same edge and $2\to3$ from $\mathrm I(2,3)$ give a certificate for modality $2$, so a second propagation inserts $\overline p$ in the same component. Atomic identity closes the branch.

For admissibility, suppose a selected structural rule of $\mathcal R_{\Amod}$ has premise $\mathcal G_0$ and conclusion $\mathcal G_1$, and suppose the certificate calculus derives $\mathcal G_0$. By kernel soundness, $\mathcal G_0$ is valid. Proposition~\ref{prop:raw-semantic-soundness} then implies validity of $\mathcal G_1$. Kernel completeness yields a certificate derivation of $\mathcal G_1$.
\end{proof}

The calculi developed below for answer computation use only the certificate calculus, whose rules preserve all existing edges of the nested tree.
An implementation that uses a raw transformation as a search heuristic must reconstruct a certificate derivation before accepting the branch as successful.

\section{Scoped unification}\label{sec:unification}

Answer computation depends on substitutions that respect the order in which eigenparameters are introduced. This section gives a first-order unification procedure with explicit permissions. Miller's work on unification under a mixed prefix and earlier implementation work on scoping constructs in logic programming provide the closest background for this dependency problem~\cite{miller1992unification,nadathur1995scoping}. The algorithm below is specialized to the first-order term syntax used by \MMLP{}. Its analysis establishes the factorization property required by the answer theory.

\subsection{Permission discipline}\label{sec:symbolic-variables}

The object language has only the set $\mathsf{Var}$ of object variables fixed in Section~\ref{sec:source-language}. For symbolic execution, fix a countably infinite effectively enumerable set $\mathsf{FVar}$ disjoint from $\mathsf{Var}\cup\mathsf{Par}$ and partition it into three disjoint effectively enumerable reserves $\mathsf{FVar}^{\mathrm A}$, $\mathsf{FVar}^{\mathrm W}$, and $\mathsf{FVar}^{\mathrm U}$. Their members are collectively called \emph{flexible variables}.

\begin{definition}[Permanent permission map]\label{def:permission-map}
Let $\mathcal P_f(S)$ denote the set of finite subsets of a set $S$. A \emph{permanent permission map} is a computable map
$\pi:\mathsf{FVar}\longrightarrow\mathcal P_{f}(\mathsf{Par})$
such that, for every finite $Q\subseteq\mathsf{Par}$, the variables $V$ with $\pi(V)=Q$ form an infinite effectively enumerable subset of each of the three reserves. The permission set $\pi(V)$ is permanently defined for every flexible variable and never changes during proof search or unification. We fix one such map $\pi$ throughout the rest of the paper.
\end{definition}

Because infinitely many variables are available for every permission set in each reserve, finite prescribed renamings that respect permission fibers can be extended to global renamings. The next lemma states this extension property.

\begin{lemma}[Renaming extension for permissions]\label{lem:permission-renaming-extension}
The following assertions hold.
\begin{enumerate}
\item Every finite injective partial renaming $p_0:Q_0\to\mathsf{Par}$, with finite $Q_0\subseteq\mathsf{Par}$, extends to a permutation $\kappa^{p}$ of $\mathsf{Par}$.
\item Fix such a permutation $\kappa^{p}$. For each reserve $\mathsf{FVar}^{R}$, where $R\in\{\mathrm A,\mathrm W,\mathrm U\}$, let $g_R^0$ be a finite injective partial map from $\mathsf{FVar}^{R}$ to itself. Suppose every prescribed pair satisfies
\[
\pi(g_R^0(V))=\kappa^{p}[\pi(V)].
\]
Then the union of the three partial maps extends to a bijection
\[
\kappa^{f}:\mathsf{FVar}\longrightarrow\mathsf{FVar}
\]
that preserves each reserve and satisfies
\[
\pi(\kappa^{f}(V))=\kappa^{p}[\pi(V)]
\]
for every $V\in\mathsf{FVar}$. In particular, any prescribed finite subset of $\mathsf{FVar}^{\mathrm A}$ whose members have empty permission can be fixed by including the pairs $V\mapsto V$ in $g_{\mathrm A}^0$.
\end{enumerate}
\end{lemma}

\begin{proof}
For the first assertion, form the finite directed graph of the partial injection $p_0$. Every vertex has indegree and outdegree at most one, so every nontrivial connected component is a directed cycle or a finite directed path. Keep each existing cycle. For every path, map its terminal vertex back to its initial vertex. Fix every parameter outside these finitely many components. The resulting map is a permutation extending $p_0$.

For the second assertion, for a reserve $R$ and a finite set $Q\subseteq\mathsf{Par}$ put
\[
F^{R}(Q)=\{V\in\mathsf{FVar}^{R}:\pi(V)=Q\}.
\]
The permutation $\kappa^{p}$ induces a bijection on finite parameter sets, $Q\mapsto \kappa^{p}[Q]$. By Definition~\ref{def:permission-map}, both $F^{R}(Q)$ and $F^{R}(\kappa^{p}[Q])$ are countably infinite. The prescribed part of $g_R^0$ with domain in $F^{R}(Q)$ is a finite injection into $F^{R}(\kappa^{p}[Q])$. Remove its finite domain and range and pair the two remaining countably infinite sets by their effective enumerations. This extends the prescribed part to a bijection from $F^{R}(Q)$ onto $F^{R}(\kappa^{p}[Q])$. Perform this construction for every reserve and every finite $Q$. Distinct fibers have disjoint domains and, because $\kappa^{p}$ is a permutation, disjoint target fibers. Their union is therefore the required bijection preserving each reserve. The prescribed variables with empty permission remain fixed because $\kappa^{p}[\varnothing]=\varnothing$.
\end{proof}

For a query $G$, each answer variable $X\in A_G$ receives a fresh $X^\sharp\in\mathsf{FVar}^{\mathrm A}$ with $\pi(X^\sharp)=\varnothing$, called its \emph{symbolic answer representative}. The map $X\mapsto X^\sharp$ is injective. The grammar of symbolic terms extends the grammar of proof terms by adding flexible variables:
$t ::= x\mid a\mid V\mid c\mid f(t_1,\ldots,t_n),$
where $x\in\mathsf{Var}$, $a\in\mathsf{Par}$, and $V\in\mathsf{FVar}$. A \emph{symbolic core formula} is obtained from the proof core grammar by allowing symbolic terms in atomic arguments, and a \emph{symbolic nested sequent} is a nested sequent whose formula occurrences are symbolic core formulas. Alpha equivalence extends to symbolic core formulas in the usual capture-avoiding way. Only flexible variables are instantiated by scoped unification. Object variables and proof parameters are rigid for this algorithm, even though object variables retain their ordinary first-order meaning in the declarative semantics.

For a derivation node $\mathsf n$, its \emph{active eigenparameter set} $Q_{\mathsf n}$ is the finite set of eigenparameters introduced by universal inferences on the path from the root of the derivation tree to $\mathsf n$. A \emph{proof metavariable} is a fresh member $U\in\mathsf{FVar}^{\mathrm W}$ reserved for a witness proof term and chosen at its introduction node $\mathsf n$ with $\pi(U)=Q_{\mathsf n}$. Section~\ref{sec:symbolic-rules} specifies where such variables are introduced. \emph{Auxiliary flexible variables} are chosen from $\mathsf{FVar}^{\mathrm U}$. Their permission sets are fixed at introduction according to the constructions below.

Each proof metavariable has an introduction node. When a mandatory branching rule creates several children, variables introduced above the split are shared by the children, while variables introduced strictly inside one child subtree are fresh for that subtree. We call this condition \emph{branch locality}. It permits distinct existential witnesses in separate branches while forcing one common substitution for symbolic answer representatives across mandatory branches.

For a symbolic term $t$, write $\operatorname{Par}(t)$ for the set of proof parameters in $t$ and $\Flex(t)$ for the set of flexible variables in $t$.

\begin{convention}[Flexible substitutions]\label{conv:flexible-substitutions}
A \emph{finite substitution on flexible variables} is a map from $\mathsf{FVar}$ to symbolic terms that differs from the identity at only finitely many variables. Its nonidentity domain is written $\dom(\theta)$. The elementary substitution $[t/X]$ maps $X$ to $t$ and fixes every other flexible variable, and $\operatorname{id}$ denotes the identity substitution. The dependency graph of $\theta$ has an edge from $X$ to $Y$ when $X\in\dom(\theta)$ and $Y\in\Flex(\theta(X))$. The substitution is \emph{acyclic} when this graph has no directed cycle. For a finite acyclic substitution, $\theta^*(X)$ denotes the term obtained by fully dereferencing $X$ under $\theta$, and $t\theta^*$ denotes its homomorphic extension to a symbolic term $t$. These conventions extend componentwise to symbolic core formulas, symbolic nested sequents, and finite sets of equations.
\end{convention}

\begin{definition}[Permission safety]\label{def:permission-safety}
A term $t$ is \emph{safe for $X$} when
$\operatorname{Par}(t)\subseteq\pi(X)$
and
$\pi(Y)\subseteq\pi(X)$ for every $Y\in\Flex(t)$.
A finite acyclic substitution $\theta$ \emph{preserves scope} when $\theta(X)$ is safe for $X$ for every $X\in\dom(\theta)$.
\end{definition}

A flexible variable that remains in a fully dereferenced term because it lies outside the substitution domain is called \emph{residual}.

The second condition in Definition~\ref{def:permission-safety} prevents an indirect scope violation. A binding $X\mapsto Y$ may satisfy the immediate parameter check even when a later binding $Y\mapsto f(a)$ would make $a$ occur in the final value of $X$. Permission inclusion blocks this possibility whenever $a\notin\pi(X)$.

The first bookkeeping fact concerns permission safety rather than graph acyclicity. Later substitutions cannot introduce a parameter dependency that was already forbidden.

\begin{lemma}[Safety under dereferenced substitution]\label{lem:safety-composition}
Suppose $t$ is safe for $X$ and $\gamma$ preserves scope. Then $t\gamma^*$ is safe for $X$.
\end{lemma}

\begin{proof}
Because $\gamma$ is finite and acyclic, every dereference chain is finite. We use induction on the maximum length of a dereference chain starting from a flexible variable occurring in $t$.

Base case. At depth zero nothing is replaced, so safety is the hypothesis.

Induction step. Assume safety is preserved through dereference depth $n$. Consider dereference depth $n+1$.

\begin{caselist}
\item Case $Y\in\dom(\gamma)$. For a flexible variable $Y$ occurring in the current term, safety gives $\pi(Y)\subseteq\pi(X)$. Scope preservation makes $\gamma(Y)$ safe for $Y$. Hence every proof parameter in $\gamma(Y)$ belongs to $\pi(Y)\subseteq\pi(X)$, and every flexible variable $Z$ in $\gamma(Y)$ satisfies $\pi(Z)\subseteq\pi(Y)\subseteq\pi(X)$.
\item Case $Y\notin\dom(\gamma)$. Then $\gamma(Y)=Y$, which is safe for $Y$ directly, and the same permission inclusions hold.
\end{caselist}

Replacing all variables at the current dereference depth therefore preserves safety for $X$. The induction terminates because the substitution is acyclic.
\end{proof}

\begin{proposition}[Scope preservation under acyclic composition]\label{prop:scope-composition}
Let $\alpha$ and $\beta$ be substitutions that preserve scope, and let $C$ be a finite set of flexible variables. Suppose there is a finite acyclic substitution $\chi$ with $\dom(\chi)\subseteq C$ such that
$\chi(X)=\alpha^*(X)\beta^*$
for every $X\in\dom(\chi)$. Then $\chi$ preserves scope.
\end{proposition}

\begin{proof}
Take $X\in\dom(\chi)$. If $X\notin\dom(\alpha)$, then $\alpha^*(X)=X$, which is safe for $X$. If $X\in\dom(\alpha)$, Definition~\ref{def:permission-safety} makes $\alpha(X)$ safe for $X$, and Lemma~\ref{lem:safety-composition}, applied with $\gamma=\alpha$, makes its fully dereferenced form $\alpha^*(X)$ safe for $X$. A second application of Lemma~\ref{lem:safety-composition}, now with $\gamma=\beta$, makes $\alpha^*(X)\beta^*$ safe for $X$. By the defining hypothesis on $\chi$, this term is $\chi(X)$. Since $X$ was arbitrary in the domain of the finite acyclic substitution $\chi$, Definition~\ref{def:permission-safety} shows that $\chi$ preserves scope.
\end{proof}

The acyclicity hypothesis is needed separately because permission safety is compositional, whereas acyclicity of two arbitrary substitution graphs is not. The unification proof below therefore avoids treating a formal product of substitutions as a new substitution unless acyclicity has been established independently.

Permissions are permanent by Definition~\ref{def:permission-map}. A variable may later be removed from the set maintained by the unification algorithm or may serve only as a residual auxiliary symbol, but its value $\pi(V)$ remains available whenever it occurs in a factor substitution.

Scope preservation controls which parameters may occur. A related support principle clarifies when two substitutions have the same action on a term: agreement on the flexible variables supporting that term is sufficient.

\begin{proposition}[Agreement on supported terms]\label{prop:supported-term-agreement}
Let $C$ be a set of flexible variables. Let $\alpha$ and $\beta$ be finite acyclic substitutions such that
$\alpha^*(Y)=\beta^*(Y)$ for every $Y\in C$.
If $\Flex(t)\subseteq C$, then
$t\alpha^*=t\beta^*.$
\end{proposition}

\begin{proof}
We use structural induction on $t$.

Base case. Rigid constants, object variables, and proof parameters are fixed by both substitutions. If $t=Y$ is flexible, the support hypothesis gives $Y\in C$, so the assumed equality applies.

Induction step. Suppose $t=f(t_1,\ldots,t_n)$. Each subterm has flexible support contained in $C$. Apply the induction hypotheses componentwise and use homomorphic substitution through the rigid function symbol $f$.
\end{proof}

\subsection{Rewrite system}\label{sec:unification-rewrite}

A finite \emph{constraint set} has the form
$E=\{s_1\doteq t_1,\ldots,s_n\doteq t_n\}.$
For such a set $E$, write $\Flex(E)$ for the union of the flexible variables occurring in its equations.
An \emph{algorithm state} is $(E,C,\eta)$, relative to the permanent permission map $\pi$. The finite set $C\subseteq\mathsf{FVar}$ is the \emph{current set of flexible variables}, and every reachable state satisfies the support invariant $\Flex(E)\subseteq C$. A variable removed from $C$ retains its permanent value $\pi(V)$, while a flexible variable that has never belonged to $C$ still has its fixed permission set $\pi(V)$. The map $\eta$ is the \emph{accumulator}. For each original flexible variable, it gives its current term representation over the current variables.

\begin{convention}[Normalized accumulator]\label{conv:normalized-accumulator}
Let $C_0$ be the original set of flexible variables. At a state with current set $C$, the accumulator is an extensional map $\eta:C_0\to\text{symbolic terms}$ satisfying
$\Flex(\eta(Z))\subseteq C$
for every $Z\in C_0$. It extends homomorphically to terms whose flexible variables belong to $C_0$. The initial accumulator is the identity map on $C_0$, denoted $\operatorname{id}$. When a rewrite uses an elementary substitution $\delta$, replace $\eta$ by the pointwise updated accumulator $\eta'$ defined by
$\eta'(Z)=\eta(Z)\delta.$
The support condition keeps every accumulator image normalized over the current variables. Whenever an accumulator is later regarded as a substitution, it is extended by the identity outside $C_0$.
\end{convention}

\begin{definition}[Scoped unification rewrites]\label{def:unification-rules}
The \emph{scoped unification procedure} repeatedly applies the first applicable rule under deterministic orders for equations and fresh variables.

\textbf{Delete.} Remove $t\doteq t$.

\textbf{Decompose.} Replace
$f(s_1,\ldots,s_n)\doteq f(t_1,\ldots,t_n)$
by $s_r\doteq t_r$ for $1\le r\le n$.

\textbf{Clash.} Fail when the rigid head symbols or arities of the two terms differ. Object variables, rigid constants, and proof parameters count as nullary rigid symbols for this algorithm.

\textbf{Orient.} Replace $t\doteq X$ by $X\doteq t$ when $X$ is flexible and $t$ is not a flexible variable.

\textbf{Occurs failure.} Fail on $X\doteq t$ when $X\neq t$ and $X\in\Flex(t)$.

\textbf{Forbidden parameter failure.} Fail on $X\doteq t$ when $\operatorname{Par}(t)\nsubseteq\pi(X)$.

\textbf{Permission restriction.} Suppose $X\doteq t$ has passed the two preceding failure tests and some $Y\in\Flex(t)$ satisfies $\pi(Y)\nsubseteq\pi(X)$. Choose the first such $Y$ and choose a fresh auxiliary flexible variable $Y'\in\mathsf{FVar}^{\mathrm U}$ with
$\pi(Y')=\pi(Y)\cap\pi(X)$ and $\delta=[Y'/Y]$.
Apply $\delta$ to the entire constraint set and update the accumulator with Convention~\ref{conv:normalized-accumulator}. Replace $Y$ by $Y'$ in the current set of flexible variables. The old variable $Y$ is no longer current. It may later occur outside a substitution domain as a residual flexible variable, and its permission remains the permanent value $\pi(Y)$.

\textbf{Bind.} Suppose $X\doteq t$, $X\notin\Flex(t)$, and $t$ is safe for $X$. Put $\delta=[t/X]$, remove the selected equation, apply $\delta$ to all remaining equations, update the accumulator, and remove $X$ from the current set of flexible variables. The eliminated variable $X$ is no longer current. It may later occur outside a substitution domain as a residual flexible variable, and its permission remains the permanent value $\pi(X)$.
\end{definition}

The initial state is $(E_0,C_0,\operatorname{id})$, where $C_0$ is any finite set of original flexible variables with $\Flex(E_0)\subseteq C_0$. A run \emph{fails} when one of the three failure rules applies. A nonfailure run \emph{succeeds} when it reaches a state $(\varnothing,C_f,\mu)$ with empty constraint set, and it returns the normalized accumulator $\mu$.

\begin{lemma}[Current support]\label{lem:current-support}
Every nonfailure state $(E,C,\eta)$ reachable from the initial state satisfies $\Flex(E)\subseteq C$.
\end{lemma}

\begin{proof}
The property holds initially by the choice of $C_0$. Delete, Decompose, and Orient introduce no new flexible variable. In Permission restriction, the uniform replacement $[Y'/Y]$ changes the support of the constraints by replacing the current variable $Y$ with the new current variable $Y'$, exactly as the update of $C$ does. In Bind, the selected equation is removed, every remaining occurrence of $X$ is replaced by $t$, and $X\notin\Flex(t)$. The old support invariant gives $\Flex(t)\subseteq C$, so after removing $X$ from $C$ every flexible variable occurring in the transformed constraints lies in $C\setminus\{X\}$. Thus the invariant is preserved by every nonfailure rewrite.
\end{proof}

\subsection{Termination}\label{sec:unification-termination}

Before assigning a decreasing measure to the rewrite system, we first verify that every nonempty nonfailure state admits a next rewrite or exposes a failure condition. This separates rule coverage from the termination argument itself.

\begin{lemma}[Rewrite exhaustiveness]\label{lem:unification-exhaustiveness}
Let $(E,C,\eta)$ be a reachable nonfailure state with $E\neq\varnothing$. Then either one of the failure rules applies or one of Delete, Decompose, Orient, Permission restriction, or Bind applies. Consequently a maximal nonfailure run can stop only with empty constraint set.
\end{lemma}

\begin{proof}
Choose the first equation $s\doteq t$ under the equation order. If $s$ and $t$ are syntactically equal, Delete applies. Suppose next that the left side is a flexible variable $X$. If $X\in\Flex(t)$, then $X\neq t$ by the preceding case, so Occurs failure applies. If some proof parameter of $t$ lies outside $\pi(X)$, Forbidden parameter failure applies. Otherwise, if some flexible variable $Y$ in $t$ has $\pi(Y)\nsubseteq\pi(X)$, Permission restriction applies. If none of these cases occurs, then $X\notin\Flex(t)$ and $t$ is safe for $X$, so Bind applies. This also covers the case in which $t$ is a distinct flexible variable.

Suppose instead that the left side is rigid and the right side is a flexible variable. Then Orient applies. It remains that neither side is flexible. Object variables, proof parameters, constants, and compound function terms have rigid outer symbols for unification. If the two rigid outer symbols or arities differ, Clash applies. If they agree and have positive arity, Decompose applies. If they agree and have arity zero, the two terms are syntactically equal, contrary to the first case. These cases exhaust every equation.
\end{proof}

\begin{theorem}[Termination of scoped unification]\label{thm:unification-termination}
The rewrite procedure in Definition~\ref{def:unification-rules} terminates on every finite constraint set.
\end{theorem}

\begin{proof}
Let $A(E)$ be the flexible variables occurring in the current constraint set. Set $M_1(E)=|A(E)|$. Define $M_2(E,\pi)$ as $\sum_{X\in A(E)}|\pi(X)|$. Let $|t|$ be the number of syntactic nodes in $t$. Define $M_3(E)$ as $\sum_{s\doteq t\in E}(|s|+|t|)$. Let $M_4(E)$ count equations of the form $t\doteq X$ that require orientation, and set $M_5(E)=|E|$. Order the tuple $M(E,\pi)=(M_1,M_2,M_3,M_4,M_5)$ lexicographically in $\mathbb N^5$.

Delete either lowers $M_1$ by removing the last occurrences in $E$ of one or more flexible variables or leaves $M_1,M_2$ fixed and strictly lowers $M_3$. Decompose leaves the first two components fixed and removes the two outer function nodes, so it strictly lowers $M_3$. Orient leaves the first three components fixed and lowers $M_4$ by one.

Permission restriction replaces $Y$ by one variable $Y'$ and therefore leaves $M_1$ fixed. Its applicability gives $\pi(Y)\nsubseteq\pi(X)$, so $\pi(Y')=\pi(Y)\cap\pi(X)\subsetneq\pi(Y).$ Thus $M_2$ strictly decreases. Bind eliminates $X$ from all remaining equations because $X\notin\Flex(t)$, so $M_1$ strictly decreases. Term growth under binding affects only later components of the measure.

Every nonfailure step strictly decreases the lexicographic measure. No infinite descending sequence exists in $\mathbb N^5$, so no run contains infinitely many nonfailure rewrites. By Lemma~\ref{lem:unification-exhaustiveness}, a finite maximal run with a nonempty constraint set must end at a failure rule. Otherwise the constraint set is empty and the run succeeds. Hence every run terminates in success or failure.
\end{proof}

\subsection{Solution factorization}\label{sec:unification-solutions}

Most generality requires preservation and factorization of solutions across individual rewrites. Factorization is stated extensionally because residual aliases need not be representable by an acyclic union of substitution graphs.

\begin{definition}[Admissible solution]\label{def:admissible-unifier}
An \emph{admissible solution} of $(E,C)$ is a substitution $\rho$ such that $\rho$ preserves scope, its nonidentity domain is contained in $C$, and
$s\rho^*=t\rho^*$
for every $s\doteq t\in E$. A residual flexible variable in a fully dereferenced range may be any member of $\mathsf{FVar}$ outside the nonidentity substitution domain. It may be current, formerly current, or never current. Its permission is always the permanent value given by $\pi$. We always replace a solution by its equivalent fully dereferenced normal form on the current variables.
\end{definition}

The last normalization convention preserves the solution set. A finite acyclic substitution and the substitution obtained by replacing each of its images by the corresponding fully dereferenced image have the same action on all terms. It is useful because variables in the nonidentity domain then do not occur in the fully dereferenced range.

Most generality is understood extensionally. For substitutions $\mu$ and $\gamma$, the expression
$\mu^*(X)\gamma^*$
means first fully dereference $X$ through $\mu$ and then instantiate the resulting residual flexible variables through $\gamma$. We do not require the union of the two substitution graphs to be acyclic. This distinction is essential for aliases. For example, the unifier $[Y/X]$ is more general than the solution $[X/Y]$ of $X\doteq Y$ because both original variables have the same value after the residual instantiation $[X/Y]$, even though combining the two graphs would create a syntactic cycle.

\begin{definition}[Term valuation]\label{def:term-valuation}
Let $C$ be a finite set of flexible variables. A \emph{term valuation} on $C$ assigns a term to every member of $C$ and extends homomorphically to terms whose flexible support is contained in $C$. A valuation is applied only once and is not recursively dereferenced. It is not a substitution graph, and no acyclicity requirement is imposed on it.
\end{definition}

Every admissible solution $\rho$ determines the term valuation $X\mapsto\rho^*(X)$ on the current variables. The temporary valuations below are used only to prove preservation of equation satisfaction across a rewrite.

\begin{lemma}[Safety of subterms]\label{lem:safe-subterm}
Suppose every proof parameter of $r$ lies in a set $S$ and every flexible variable $Z$ in $r$ satisfies $\pi(Z)\subseteq S$. Every subterm of $r$ has the same two properties.
\end{lemma}

\begin{proof}
Let $s$ be a subterm of $r$. Then $\operatorname{Par}(s)\subseteq\operatorname{Par}(r)\subseteq S$ and $\Flex(s)\subseteq\Flex(r)$. Hence every proof parameter of $s$ lies in $S$, and every flexible variable $Z$ occurring in $s$ satisfies $\pi(Z)\subseteq S$ by the hypothesis on $r$.
\end{proof}

\begin{lemma}[Accumulator regularity]\label{lem:accumulator-regularity}
At every nonfailure state reachable from the initial state, with original set of flexible variables $C_0$, current set $C$, and accumulator $\eta$, the following assertions hold.
\begin{enumerate}
\item $\Flex(\eta(Z))\subseteq C$ for every $Z\in C_0$.
\item If $Z\in C_0\cap C$, then $\eta(Z)=Z$.
\item The term $\eta(Z)$ is safe for $Z$ for every $Z\in C_0$.
\end{enumerate}
Consequently, after extension by the identity outside $C_0$, the accumulator is a finite acyclic substitution that preserves scope, and its fully dereferenced value at each $Z\in C_0$ is $\eta(Z)$.
\end{lemma}

\begin{proof}
We use induction on the length of the rewrite sequence.

Base case. At the initial state all three assertions hold for the identity accumulator.

Induction step. Assume the assertions at the current state and consider the next nonfailure rewrite.

\begin{caselist}
\item Case Delete, Decompose, or Orient. These rules leave the accumulator and current set unchanged, so all three assertions are preserved.
\item Case Permission restriction. Let $\delta=[Y'/Y]$ and $C'=(C\setminus\{Y\})\cup\{Y'\}$. The normalized update replaces each occurrence of $Y$ in an accumulator image by the fresh variable $Y'$ and changes no other flexible variable. This proves the support assertion. If an original variable $Z$ remains current, then $Z\neq Y$, and its old value was $Z$ by the induction hypothesis. The update therefore leaves it equal to $Z$. For safety, if $Y$ occurs in $\eta(Z)$, the induction hypothesis gives $\pi(Y)\subseteq\pi(Z)$. Since $\pi(Y')=\pi(Y)\cap\pi(X)\subseteq\pi(Y)$, replacing $Y$ by $Y'$ preserves safety for $Z$. If $Y$ does not occur, nothing changes.
\item Case Bind. Let $\delta=[t/X]$ and $C'=C\setminus\{X\}$. The occurs condition gives $X\notin\Flex(t)$, and Lemma~\ref{lem:current-support} gives $\Flex(t)\subseteq C'$. Replacing occurrences of $X$ in the accumulator images therefore establishes the new support assertion. Every original variable that remains current differs from $X$ and had identity value, so the second assertion is preserved. If $X$ occurs in $\eta(Z)$, old safety gives $\pi(X)\subseteq\pi(Z)$. The binding side condition makes $t$ safe for $X$, so every proof parameter of $t$ lies in $\pi(X)\subseteq\pi(Z)$ and every flexible variable $V$ of $t$ satisfies $\pi(V)\subseteq\pi(X)\subseteq\pi(Z)$. Thus substitution of $t$ for $X$ preserves safety for $Z$.
\end{caselist}

For the consequence, the nonidentity domain of the normalized accumulator is contained in $C_0\setminus C$ by the second assertion, while every accumulator image has flexible support contained in $C$ by the first. Hence no variable in the nonidentity domain occurs in any image. The accumulator is therefore acyclic, and the third assertion gives scope preservation.
\end{proof}

We next give the factorization direction needed for completeness. It is stated with extensional evaluation and therefore remains valid when an old solution uses a variable as a residual alias even though a later binding eliminates that variable.

\begin{lemma}[Permission restriction factorization]\label{lem:permission-restriction-factorization}
Consider a permission restriction step on $X\doteq t$ using $Y\in\Flex(t)$ with $\pi(Y)\nsubseteq\pi(X)$. Let $Y'$ be fresh with $\pi(Y')=\pi(Y)\cap\pi(X)$ and let $\delta=[Y'/Y]$. For every admissible solution $\rho$ of the old problem there is an admissible solution $\rho'$ of the transformed problem such that
$\rho^*(Z)=(Z\delta)(\rho')^*$
for every old current flexible variable $Z$.
\end{lemma}

\begin{proof}
Because $\rho$ solves $X\doteq t$,
$\rho^*(X)=\rho^*(t).$
Since $Y$ occurs in $t$, the fully dereferenced value $\rho^*(Y)$ occurs as a subterm of $\rho^*(t)$. The value $\rho^*(X)$ is safe for $X$, so Lemma~\ref{lem:safe-subterm} shows that every proof parameter of $\rho^*(Y)$ lies in $\pi(X)$ and every residual flexible variable $V$ in that term satisfies $\pi(V)\subseteq\pi(X)$. Admissibility of $\rho$ gives the same statements with $\pi(Y)$ in place of $\pi(X)$. Hence $\rho^*(Y)$ is safe for the fresh variable $Y'$ with permission $\pi(Y)\cap\pi(X)$.

Define $\rho'$ in normalized form by
$\rho'^*(Y')=\rho^*(Y)$
and, for every other new current variable $Z$,
$\rho'^*(Z)=\rho^*(Z)$.
If $\rho^*(Y)=Y$, then $Y$ would remain as a residual subterm of $\rho^*(X)$, contradicting $\pi(Y)\nsubseteq\pi(X)$. Thus $Y$ does not occur in the displayed ranges. The values are already fully dereferenced under $\rho$, and $Y'$ is fresh, so the resulting substitution is finite and acyclic. Safety of $\rho^*(Y)$ for $Y'$ was established above. For every other variable, scope preservation follows from that of $\rho$.

We prove $(s\delta)(\rho')^*=s\rho^*$ for every old term $s$ by structural induction.

Base case. If $s=Y$, the equality is the definition of $\rho'$. Every other variable case is unchanged, and rigid nullary symbols are fixed.

Induction step. Suppose $s=f(s_1,\ldots,s_n)$. Apply the induction hypotheses to the arguments and use homomorphic evaluation through $f$.

Every transformed equation is therefore satisfied because the corresponding old equation is satisfied by $\rho$.
\end{proof}

\begin{lemma}[Binding factorization]\label{lem:binding-factorization}
Suppose a binding step uses $X\doteq t$ with $X\notin\Flex(t)$ and $t$ safe for $X$. Put $\delta=[t/X]$. For every admissible solution $\rho$ of the old problem there is an admissible solution $\rho'$ of the reduced problem such that
$\rho^*(Z)=(Z\delta)(\rho')^*$
for every old current flexible variable $Z$.
\end{lemma}

\begin{proof}
Define $\rho'$ on each new current variable $Z\neq X$ by the normalized value
$\rho'^*(Z)=\rho^*(Z)$,
omitting identity bindings. Since these are fully dereferenced values of the finite acyclic substitution $\rho$, no variable in the nonidentity domain of $\rho'$ occurs in its range. The eliminated variable $X$ may remain as a residual symbol, but it is outside the new domain and retains its permanent permission set $\pi(X)$. Thus $\rho'$ is finite and acyclic. Scope preservation for every variable in its domain follows directly from scope preservation of $\rho$.

We first prove $r(\rho')^*=r\rho^*$ for every term $r$ with $X\notin\Flex(r)$ by structural induction.

Base case. Every flexible variable occurring in $r$ differs from $X$, so its fully dereferenced value is the same under $\rho'$ and $\rho$ by construction. Rigid nullary symbols are fixed.

Induction step. Suppose $r=f(r_1,\ldots,r_n)$. Apply the induction hypotheses to the arguments and use homomorphic evaluation through $f$.

In particular, $t(\rho')^*=t\rho^*$.

We next prove $(s\delta)(\rho')^*=s\rho^*$ for every old term $s$ by structural induction.

Base case. The only additional variable case is $s=X$. There,
$(X\delta)(\rho')^*=t(\rho')^*=t\rho^*=\rho^*(X),$
where the last equality is the fact that $\rho$ solves $X\doteq t$. The other variable cases use the definition of $\rho'$, and rigid nullary symbols are fixed.

Induction step. Suppose $s=f(s_1,\ldots,s_n)$. Apply the induction hypotheses componentwise and use homomorphic evaluation through $f$.

Every reduced equation is the $\delta$ image of an old equation distinct from the selected one, so the displayed equality makes $\rho'$ a solution.
\end{proof}

Delete, Decompose, and Orient require no separate construction. An old solution remains a solution after Delete or Orient, and injectivity of free constructors makes it a solution after Decompose. The two factorization lemmas therefore yield the general one-step statement.

\begin{lemma}[One-step solution factorization]\label{lem:one-step-factorization}
Suppose a nonfailure rewrite transforms the old problem into a new problem and uses the elementary substitution $\delta$, with $\delta=\operatorname{id}$ for Delete, Decompose, and Orient. Every admissible old solution $\rho$ has an admissible new solution $\rho'$ such that
$\rho^*(Z)=(Z\delta)(\rho')^*$
for every old current flexible variable $Z$.
\end{lemma}

\begin{proof}
Delete and Orient are immediate. For Decompose, equality of two terms with the same rigid outer constructor implies equality of every corresponding argument under $\rho^*$. Permission restriction is Lemma~\ref{lem:permission-restriction-factorization}, and Bind is Lemma~\ref{lem:binding-factorization}.
\end{proof}

For soundness of the returned accumulator we need a different direction, but only at the level of term equality. This avoids creating a possibly cyclic formal composition.

\begin{lemma}[One-step equation preservation]\label{lem:one-step-equation-transport}
Suppose a nonfailure rewrite uses $\delta$ to transform an old constraint set $E$ into a new constraint set $E'$. Let $\nu'$ be any term valuation on the new current variables that satisfies every equation in $E'$. Define a valuation $\nu$ on the old current variables by
$\nu(Z)=\nu'(Z\delta),$
where $\nu'$ is extended homomorphically. Then $\nu$ satisfies every equation in $E$.
\end{lemma}

\begin{proof}
Delete adds back only an equation $t\doteq t$. Orient uses symmetry of equality. For Decompose, satisfaction of all equations $s_r\doteq t_r$ gives equality of the two terms with the common rigid outer constructor.

For Permission restriction, every old equation is transformed uniformly by $\delta$, so its two sides have equal $\nu$ values exactly when its transformed sides have equal $\nu'$ values. For Bind, every unselected equation is treated in the same way. The selected equation becomes
$X\doteq t$.
Its left side receives value $\nu'(X\delta)=\nu'(t)$, while its right side receives the same value because $X\notin\Flex(t)$ and $\delta$ therefore leaves $t$ unchanged.
\end{proof}

The accumulator update is compatible with these one-step valuations.

\begin{lemma}[Accumulator update compatibility]\label{lem:accumulator-transport}
Consider one nonfailure step with old current set $C$, new current set $C'$, elementary substitution $\delta$, and the normalized accumulator update from $\eta$ to $\eta'$. Assume
$\Flex(\eta(Z))\subseteq C$
for every original variable $Z$. Let $\nu'$ be any term valuation on $C'$ and define
$\nu(Y)=\nu'(Y\delta)$
for $Y\in C$. Then
$\nu(\eta(Z))=\nu'(\eta'(Z))$
for every original variable $Z$.
\end{lemma}

\begin{proof}
Put $r=\eta(Z)$. The support hypothesis gives $\Flex(r)\subseteq C$. We prove $\nu(r)=\nu'(r\delta)$ by structural induction on $r$.

Base case. Rigid nullary symbols are fixed. For a flexible variable $Y\in C$, the equality is the definition of $\nu(Y)$.

Induction step. Suppose $r=f(r_1,\ldots,r_n)$. Apply the induction hypotheses to the arguments and use homomorphic evaluation through $f$.

By Convention~\ref{conv:normalized-accumulator}, $\eta'(Z)=r\delta$, so the displayed equality is exactly the required conclusion.
\end{proof}

\begin{lemma}[Soundness of failure rules]\label{lem:failure-soundness}
If Clash, Occurs failure, or Forbidden parameter failure applies, the current problem has no admissible solution.
\end{lemma}

\begin{proof}
For Clash, substitution cannot change the outer rigid function symbol or its arity, so the two sides cannot become the same free term.

For Occurs failure, $X$ occurs at a proper subterm position of $t$. Under any finite acyclic substitution, $t\rho^*$ therefore contains $\rho^*(X)$ as a proper subterm and has strictly larger term size. It cannot equal $\rho^*(X)$.

For Forbidden parameter failure, choose $a\in\operatorname{Par}(t)\setminus\pi(X)$. The rigid proof parameter $a$ cannot disappear under substitution. Any solution of $X\doteq t$ would place $a$ in $\rho^*(X)$, contradicting the permission requirement for $X$.
\end{proof}

Together, termination, accumulator regularity, one-step equation preservation, soundness of the failure rules, and one-step factorization prove the main unification theorem.

\begin{theorem}[Scoped most general unifier in the extensional sense]\label{thm:scoped-mgu}
Apply the algorithm of Definition~\ref{def:unification-rules} to a finite first-order constraint problem $(E_0,C_0)$ with $\Flex(E_0)\subseteq C_0$, relative to the permanent permission map $\pi$.
\begin{enumerate}
\item If the algorithm fails, no admissible solution exists.
\item If the algorithm succeeds with $(\varnothing,C_f,\mu)$, then $\mu$ is an admissible solution of the original problem.
\item For every admissible solution $\sigma$ of the original problem, there is a substitution $\gamma$ that preserves scope such that
$\sigma^*(Z)=\mu^*(Z)\gamma^*$
for every $Z\in C_0$.
\end{enumerate}
Consequently, the algorithm fails if and only if the original problem has no admissible solution.
Thus $\mu$ is a scoped most general unifier in the sense of clause 3. Below we abbreviate this as a \emph{scoped mgu}.
\end{theorem}

\begin{proof}
Termination is Theorem~\ref{thm:unification-termination}. Let a nonfailure run have current constraint set $E$, current set of flexible variables $C$, and accumulator $\eta$. We maintain two semantic invariants in addition to the syntactic accumulator properties of Lemma~\ref{lem:accumulator-regularity}.

For the first invariant, every term valuation $\nu$ on $C$ that satisfies $E$ also satisfies the original equations after the accumulator. For every $s\doteq t\in E_0$,
$\nu(s\eta)=\nu(t\eta).$
We prove this invariant by induction on the length of the nonfailure rewrite prefix.

Base case. At the initial state, the displayed assertion is exactly the assumption that $\nu$ satisfies $E_0$.

Induction step. Assume the invariant before one rewrite and let $\nu'$ satisfy the new constraint set. Lemma~\ref{lem:one-step-equation-transport} defines an old valuation $\nu$ satisfying the old constraint set. The induction hypothesis applies to $\nu$, and Lemma~\ref{lem:accumulator-transport} relates both sides to evaluation through the updated accumulator. Hence the first invariant is preserved.

For the second invariant, for every admissible initial solution $\sigma$ there is an admissible current solution $\rho$ such that
$\sigma^*(Z)=\eta(Z)\rho^*$
for every $Z\in C_0$. We again use induction on the length of the nonfailure rewrite prefix.

Base case. At the initial state, take $\rho=\sigma$.

Induction step. Assume the invariant at the old state. By Lemma~\ref{lem:one-step-factorization}, there is an admissible new solution $\rho'$ such that
$\rho^*(Y)=(Y\delta)(\rho')^*$
for every old current variable $Y$. Since every $\eta(Z)$ is supported by the old current set, we lift the pointwise factorization to $\eta(Z)$ by structural induction on that term.

Base case. Both evaluations fix rigid symbols. For a flexible leaf $Y$, the required equality is $\rho^*(Y)=(Y\delta)(\rho')^*$.

Induction step. Suppose $\eta(Z)=f(t_1,\ldots,t_n)$. Apply the induction hypotheses to the arguments and use homomorphic evaluation through $f$.

Therefore
$\eta(Z)\rho^*=(\eta(Z)\delta)(\rho')^*.$
By the normalized accumulator convention, $\eta'(Z)=\eta(Z)\delta$. Hence
$\eta(Z)\rho^*=\eta'(Z)(\rho')^*.$
The second invariant is therefore preserved.

If the algorithm reaches a failure rule, the second invariant implies that any initial admissible solution would induce an admissible solution of the failing current problem. This contradicts Lemma~\ref{lem:failure-soundness}.

If the algorithm succeeds with empty constraint set and accumulator $\mu$, extend $\mu$ by the identity outside $C_0$, as in Convention~\ref{conv:normalized-accumulator}. The identity term valuation on the final current variables satisfies the empty problem. The first invariant therefore shows that this substitution satisfies every original equation. Lemma~\ref{lem:accumulator-regularity} shows independently that it is finite and acyclic, preserves scope, and has $\mu^*(Z)=\mu(Z)$ for every $Z\in C_0$. Hence $\mu$ is an admissible solution of the original problem.

Finally, let $\sigma$ be any admissible original solution. By the second invariant at the final state, there is an admissible final substitution $\gamma$ such that
$\sigma^*(Z)=\mu^*(Z)\gamma^*$
for every $Z\in C_0$.
\end{proof}

\section{Symbolic execution}\label{sec:symbolic}

The ordinary certificate calculus represents derivations in which all witness terms are explicit. Logic programming instead requires symbolic execution that can postpone those choices and return substitutions for the variables selected by the user. This section defines the corresponding symbolic proof objects and execution procedure. Symbolic execution never enumerates first-order witness terms at an existential rule. It introduces one fresh proof metavariable and assigns it the eigenparameters that are currently available. Atomic closure contributes equations. Once every branch ends in a symbolic truth or atomic leaf, scoped unification is applied to the accumulated terminal equation set.

Scoped instantiation proves answer soundness by converting a symbolic derivation and an admissible solution of its terminal equations into an ordinary certificate derivation. For completeness, proof abstraction produces a symbolic derivation together with a solution of its equations from an ordinary certificate proof and the original answer. Scoped lifting factors that solution through the computed scoped mgu. These results yield the answer characterization theorem at the end of the section.

\subsection{Symbolic roots}\label{sec:symbolic-roots}

Let $P=\{P_1,\ldots,P_n\}$ be a finite program whose formulas have been universally closed, let $G$ be a query, and let $A_G\subseteq \FV(G)$ be the finite set of variables whose values are requested as output. For every $X\in A_G$, use the symbolic answer representative $X^\sharp$ introduced in Section~\ref{sec:symbolic-variables}, with $\pi(X^\sharp)=\varnothing$. This choice reflects that the representative is already present at the symbolic root before any eigenparameter is introduced.

\begin{definition}[Root answer replacement]\label{def:root-answer-replacement}
The \emph{root answer replacement} map $j_G$ acts on free occurrences of variables in the query as follows:
$j_G(X)=X^\sharp$ for $X\in A_G$, and $j_G(x)=x$ for $x\notin A_G$.
It fixes bound variables after the usual standardization apart, and it acts homomorphically on function symbols, predicate symbols, connectives, quantifiers, modalities, and nested brackets.

For $n>0$, the root component of the symbolic root associated with $(P,G)$ contains the program formula occurrences $\NNF^{-}(P_1),\ldots,\NNF^{-}(P_n)$ followed by the query formula occurrence $j_G(\NNF^{+}(G))$. For the empty program, the root component contains $\bot$ followed by $j_G(\NNF^{+}(G))$.
Each $X^\sharp$ receives permission set $\varnothing$.
\end{definition}

The next lemma states the relation between this symbolic root and ordinary query substitution. It will be used both in answer soundness and in proof abstraction.

\begin{lemma}[Root concretization]\label{lem:root-concretization}
Let $\rho$ be a substitution on the flexible variables of a symbolic root that preserves scope. Let $q$ injectively replace every residual flexible variable occurring in $\rho^*(X^\sharp)$, $X\in A_G$, by a fresh object variable. Define
$\theta(\rho,q)(X)=q(\rho^*(X^\sharp))$ for $X\in A_G$.
If every term on the right is an ordinary user term, then
$q\bigl(\rho^*(j_G(\NNF^{+}(G)))\bigr) \equiv_\alpha \NNF^{+}(G\theta(\rho,q)).$
\end{lemma}

\begin{proof}
We first prove the term statement $q(\rho^*(j_G(t)))=t\theta(\rho,q)$ by structural induction on $t$.

Base case. If $t=X\in A_G$, the equality is the definition of $\theta(\rho,q)$. If $t=x\notin A_G$, all three operations fix $x$. Constants are fixed.

Induction step. Suppose $t=f(t_1,\ldots,t_n)$. Apply the induction hypotheses to the arguments and use homomorphic action of all maps through $f$.

Next proceed by structural induction on the core formula $\NNF^{+}(G)$.

Base case. Positive and negative atoms follow from the term statement. Truth constants are fixed.

Induction step. Assume the claim for the immediate subformulas.

\begin{caselist}
\item Case Boolean connective or indexed modality. Apply the induction hypotheses and homomorphic extension.
\item Case quantifier. Alpha convert the bound variable in advance so that it is fresh for the finite ranges of $\rho$, $q$, and $\theta(\rho,q)$. The induction hypothesis applies to the matrix, and capture-avoiding substitution then commutes with the quantifier.
\end{caselist}

Thus $q\bigl(\rho^*(j_G(\NNF^{+}(G)))\bigr) \equiv_\alpha \NNF^{+}(G)\theta(\rho,q)$. Finite substitution compatibility of NNF from Theorem~\ref{thm:normalization-adequacy} gives $\NNF^{+}(G)\theta(\rho,q) \equiv_\alpha \NNF^{+}(G\theta(\rho,q)).$ \end{proof}

\subsection{Terminal closure}\label{sec:symbolic-closure}

The two terminal forms below are the only source of unification equations.

\begin{definition}[Truth closure]\label{def:symbolic-truth}
A component containing $\top$ closes by \emph{truth closure}. It contributes no equation.
\end{definition}

\begin{definition}[Atomic closure]\label{def:symbolic-atomic}
\emph{Atomic closure} is available when one component contains $p(s_1,\ldots,s_n)$ and $\overline p(t_1,\ldots,t_n)$.
The component then closes symbolically and contributes the equations
$s_1\doteq t_1,\ldots,s_n\doteq t_n.$
Predicate symbol and arity must agree. No atomic closure rule applies when they differ.
\end{definition}

\begin{lemma}[Validity of symbolic closure]\label{lem:symbolic-closure-sound}
Let $\rho$ be a finite acyclic substitution that preserves scope and satisfies every equation contributed by a symbolic terminal leaf, and let $q$ replace residual flexible variables by fresh object variables. After application of $q\rho^*$, the corresponding ordinary leaf is an instance of the truth initial rule or the atomic identity rule of the certificate calculus.
\end{lemma}

\begin{proof}
The truth case is immediate because $\top$ is fixed by both substitutions. In the atomic case, the hypothesis that $\rho$ satisfies the contributed equations gives $\rho^*(s_i)=\rho^*(t_i)$ for every argument position. Applying $q$ preserves these equalities. The resulting component therefore contains complementary literals with syntactically equal argument tuples, which is the ordinary atomic identity condition.
\end{proof}

\subsection{Symbolic rules}\label{sec:symbolic-rules}

The symbolic calculus uses the same propositional and modal rules as the ordinary certificate calculus. Only witness selection and terminal atomic closure require additional data.

\begin{definition}[Symbolic derivation]\label{def:symbolic-derivation}
A \emph{symbolic derivation}, written $\mathcal D^\sharp$, is a finite derivation of symbolic nested sequents obtained from the ordinary certificate rules with the following changes. Truth and atomic leaves use Definitions~\ref{def:symbolic-truth} and~\ref{def:symbolic-atomic}. When an existential formula $\exists xA$ occurs at a derivation node $\mathsf n$, the symbolic existential rule chooses a fresh proof metavariable $U\in\mathsf{FVar}^{\mathrm W}$ with $\pi(U)=Q_{\mathsf n}$ and adds $A(U/x)$ while retaining the existential occurrence. No concrete term is enumerated at this step. For a universal formula, the ordinary universal rule is used subject to the global naming discipline that distinct universal inferences in one symbolic derivation use distinct eigenparameters and every such parameter is absent from the symbolic root. Proof parameters arise symbolically only through universal inferences. Conjunction retains its two mandatory branches, with variables introduced strictly below distinct branch children required to be distinct while variables introduced above the split remain shared. Box, seriality, and certificate diamond rules are the ordinary rules with symbolic terms left untouched.

The \emph{terminal equation set} $E(\mathcal D^\sharp)$ is the union of the equations contributed by all atomic terminal leaves, and $C(\mathcal D^\sharp)$ is the set of flexible variables occurring in the derivation. The global universal naming condition implies the ordinary freshness condition at every local conclusion. Permission sets of older flexible variables are never enlarged. The raw structural transformations of Definition~\ref{def:raw-calculus} are not inference rules of symbolic derivations.
\end{definition}

\begin{definition}[Successful symbolic derivation]\label{def:successful-symbolic-derivation}
A symbolic derivation $\mathcal D^\sharp$ is \emph{successful} when every branch ends in a symbolic truth or atomic leaf and scoped unification has a successful run from $(E(\mathcal D^\sharp),C(\mathcal D^\sharp),\operatorname{id})$. If that run returns $\mu$, we call $\mu$ the \emph{scoped mgu of the derivation}. A \emph{symbolic computation for $(P,G)$} is a symbolic derivation whose root is the symbolic root associated with $(P,G)$ in Definition~\ref{def:root-answer-replacement}. A successful symbolic computation is a symbolic computation whose derivation is successful.
\end{definition}

The side condition for certificate propagation depends only on the nested tree and the modal grammar. Therefore first-order substitutions do not alter certificate validity.

\subsection{Scoped instantiation}\label{sec:scoped-instantiation}

The next theorem establishes that an admissible solution of a symbolic derivation yields an ordinary certificate proof. For a symbolic derivation $\mathcal D^\sharp$ and a finite acyclic substitution $\rho$, let $R(\mathcal D^\sharp,\rho)$ denote the set of residual flexible variables occurring in $\rho^*(\mathcal D^\sharp)$.

\begin{theorem}[Scoped instantiation]\label{thm:scoped-instantiation}
Let $\mathcal D^\sharp$ be a symbolic derivation, let $E(\mathcal D^\sharp)$ be the union of the equations contributed by all its atomic leaves, and let $C(\mathcal D^\sharp)$ be the set of flexible variables occurring in the derivation. Let $\rho$ be an admissible solution of $(E(\mathcal D^\sharp),C(\mathcal D^\sharp))$. Choose an injection
$q:R(\mathcal D^\sharp,\rho)\longrightarrow \mathsf{Var}$
whose range consists of fresh object variables and is disjoint from all eigenparameters and names already used in the derivation. Extend $q$ homomorphically to symbolic terms and formulas and componentwise to nested sequents and derivations, fixing all other symbols, positions, rule annotations, and modal certificates. Then $q(\rho^*(\mathcal D^\sharp))$ is an ordinary certificate derivation of the concretized root.
\end{theorem}

\begin{proof}
We use induction on the height of $\mathcal D^\sharp$.

Base case. If the derivation is a symbolic truth leaf, Lemma~\ref{lem:symbolic-closure-sound} gives an ordinary truth initial leaf. If it is a symbolic atomic leaf, the same lemma gives an ordinary atomic identity leaf because $\rho$ solves every equation contributed there.

Induction step. Assume the claim for all premise derivations of the final inference.

\begin{caselist}
\item Case propositional inference. Application of $\rho^*$ and $q$ commutes with the connective constructors and with capture-avoiding substitution. The induction hypotheses therefore provide ordinary derivations of all premises, and the corresponding ordinary propositional rule reconstructs the conclusion.
\item Case box, certificate diamond, or seriality. The nested tree is unchanged by term substitution. Certificate validity under changes to first-order terms is Lemma~\ref{lem:certificate-term-invariance}. The induction hypothesis reconstructs the premise proof, after which the same ordinary modal rule applies.
\item Case existential inference. Let the inference be created at derivation node $\mathsf n$ with fresh proof metavariable $U\in\mathsf{FVar}^{\mathrm W}$ satisfying $\pi(U)=Q_{\mathsf n}$. Because $\rho$ preserves scope, Lemma~\ref{lem:safety-composition} makes $\rho^*(U)$ safe for $U$. Hence every eigenparameter in $\rho^*(U)$ belongs to $Q_{\mathsf n}=\pi(U)$. Residual variables become object variables under $q$, so $t=q(\rho^*(U))$ is a legal proof term at $\mathsf n$. The induction hypothesis gives an ordinary derivation of the instantiated premise, and the ordinary existential witness rule applies with $t$.
\item Case universal inference. Let $a$ be the fresh eigenparameter. Any flexible variable already present before the inference has a permission set that does not contain $a$. Scope preservation therefore prevents $a$ from occurring in the substituted conclusion through an older variable. The chosen range of $q$ is fresh and contains no proof parameter. Thus $a$ remains fresh for the concretized conclusion. The induction hypothesis gives the premise derivation, and the ordinary universal eigenparameter rule applies.
\item Case mandatory branching. The same $\rho$ solves $E(\mathcal D^\sharp)$ and therefore solves every equation contributed in each mandatory branch. Apply the induction hypothesis to every branch. Variables introduced below different branch children are distinct, while shared ancestor variables receive the same value under $\rho$. The ordinary mandatory branching rule therefore recombines the premise derivations.
\end{caselist}
\end{proof}

\subsection{Answer projection}\label{sec:answer-projection}

The scoped mgu returned by unification may contain residual flexible variables. A user answer must contain only ordinary user terms. The empty permission sets of the symbolic answer representatives provide the required separation from proof parameters.

\begin{lemma}[Residual variables at answer positions]\label{lem:zero-permission-residual}
Let $\mu$ preserve scope and let $X^\sharp$ be a symbolic answer representative with $\pi(X^\sharp)=\varnothing$. Then $\mu^*(X^\sharp)$ contains no proof parameter. Moreover, every residual flexible variable $V$ occurring in $\mu^*(X^\sharp)$ has $\pi(V)=\varnothing$.
\end{lemma}

\begin{proof}
Because $\mu$ is finite and acyclic and because it preserves scope, repeated use of Lemma~\ref{lem:safety-composition} along its dereference chains shows that the fully dereferenced term $\mu^*(X^\sharp)$ is safe for $X^\sharp$. Hence
$\operatorname{Par}(\mu^*(X^\sharp))\subseteq\pi(X^\sharp)=\varnothing.$
If $V$ is residual in $\mu^*(X^\sharp)$, safety of the same term gives $\pi(V)\subseteq\pi(X^\sharp)=\varnothing$, and therefore $\pi(V)=\varnothing$.
\end{proof}

The lemma makes projection possible. Every residual flexible variable occurring in the fully dereferenced value of an answer representative can be replaced by an object variable, and no proof parameter appears in the resulting answer term.

\begin{definition}[Answer projection]\label{def:answer-projection}
Let $\mu$ be the scoped mgu of a successful symbolic computation for $(P,G)$. Let $R(\mu)$ be the finite set of residual flexible variables that occur in some $\mu^*(X^\sharp)$ with $X\in A_G$. Its members are called \emph{answer residual variables}.

An \emph{ordinary output variable} is an object variable reserved for answer projection and chosen fresh for $P$, $G$, and the symbolic derivation, with distinct answer residuals assigned distinct output variables. For each $V\in R(\mu)$ choose an output variable $Y_V$ and put $r^{\mu}(V)=Y_V$. Extend $r^{\mu}$ homomorphically to symbolic terms whose flexible variables lie in $R(\mu)$, fixing object variables, proof parameters, and rigid signature symbols. Define $\operatorname{proj}(\mu)$ to be the ordinary substitution satisfying
$\operatorname{proj}(\mu)(X)=r^{\mu}(\mu^*(X^\sharp))$ for every $X\in A_G$ and fixing every object variable outside $A_G$. Identity bindings are omitted as in Convention~\ref{conv:ordinary-substitutions}. In particular, if $\mu^*(X^\sharp)=X^\sharp$, then $X^\sharp$ is itself an answer residual variable and projection maps $X$ to a fresh ordinary output variable, representing an unconstrained answer position.
\end{definition}

By Lemma~\ref{lem:zero-permission-residual}, the projected terms contain no proof parameter. They are ordinary user terms. For example, if $\mu^*(X^\sharp)=f(V)$, then the user receives $X\mapsto f(Y_V)$.

Projection assigns names only to residuals that occur in the user answer. Scoped instantiation must replace every residual flexible variable occurring after full dereferencing of the entire symbolic derivation, including auxiliary variables created during unification. The answer renaming is therefore extended once to that whole finite residual set.

\begin{lemma}[Global residual renaming]\label{lem:global-residual-renaming}
Let $\mathcal D^\sharp$ be a successful symbolic computation for $(P,G)$ with scoped mgu $\mu$. The injection $r^{\mu}$ on answer residual variables extends to an injection
$\widehat r:R(\mathcal D^\sharp,\mu)\longrightarrow\mathsf{Var}$
on every residual flexible variable occurring in $\mu^*(\mathcal D^\sharp)$. Its range can be chosen disjoint from the program, query, proof parameters, and all object variables already present in the derivation.
\end{lemma}

\begin{proof}
Every answer residual belongs to $R(\mathcal D^\sharp,\mu)$ because each symbolic answer representative occurs in the root of $\mathcal D^\sharp$. The set $R(\mathcal D^\sharp,\mu)$ is finite because both the derivation and the finite acyclic substitution $\mu$ are finite. Keep the prescribed values $r^{\mu}(V)$ for answer residuals and choose pairwise distinct fresh object variables for the remaining members of $R(\mathcal D^\sharp,\mu)$. Since $\mathsf{Var}$ is infinite, the required fresh names can be chosen.
\end{proof}

\subsection{Answer soundness}\label{sec:answer-soundness}

Projection converts answer residuals to ordinary output variables. For a successful symbolic computation $\mathcal D^\sharp$ with scoped mgu $\mu$, global residual renaming extends this assignment to every residual flexible variable occurring in $\mu^*(\mathcal D^\sharp)$. Scoped instantiation then yields an ordinary certificate derivation. We can therefore define the substitution returned to the user and compare it with the independent declarative answer relation.

\begin{definition}[Computed answer]\label{def:computed-answer}
Let $\mathcal D^\sharp$ be a successful symbolic computation for $(P,G)$ in the sense of Definition~\ref{def:successful-symbolic-derivation}, and let $\mu$ be its scoped mgu. The \emph{computed answer} is
$\theta=\operatorname{proj}(\mu).$
We write
$\MMLP_{\Amod}(P,G)\Downarrow\theta$
for this relation and let $\Comp_{\Amod}^{\mathrm M}(P,G)$ be the set of all such $\theta$.
The superscript $\mathrm M$ marks ordinary \MMLP{} execution.
\end{definition}

Since the declarative answer relation was defined independently in Definition~\ref{def:hilbert-correct}, the soundness theorem below proves that every computed answer belongs to that relation.

\begin{theorem}[Answer soundness]\label{thm:answer-soundness}
If
$\MMLP_{\Amod}(P,G)\Downarrow\theta,$
then
$\theta\in\Corr^{\mathrm H}_{\Amod}(P,G).$
\end{theorem}

\begin{proof}
Take a successful symbolic computation $\mathcal D^\sharp$ for $(P,G)$ with scoped mgu $\mu$ and projected answer $\theta$. By Theorem~\ref{thm:scoped-mgu}, $\mu$ is an admissible solution of $(E(\mathcal D^\sharp),C(\mathcal D^\sharp))$. Extend the answer renaming $r^{\mu}$ to $\widehat r$ by Lemma~\ref{lem:global-residual-renaming}. Scoped instantiation, Theorem~\ref{thm:scoped-instantiation}, gives an ordinary certificate derivation $\widehat r(\mu^*(\mathcal D^\sharp)).$ 
The closed program formulas contain no symbolic answer representatives, so their root formula occurrences are unchanged. By Lemma~\ref{lem:root-concretization}, the query formula occurrence in the concretized root is alpha-equivalent to $\NNF^{+}(G\theta)$. Hence, by Lemma~\ref{lem:root-normalization}, the ordinary derivation proves the normalized root associated with $\widehat P\to G\theta$. Kernel soundness, Corollary~\ref{cor:kernel-soundness}, gives validity of this normalized root over $\Frames(\Amod)$. By Theorem~\ref{thm:normalization-adequacy}, the source implication $\widehat P\to G\theta$ is therefore valid over $\Frames(\Amod)$.

Let $Y_1,\ldots,Y_r$ enumerate $\FV(G\theta)$. The program conjunction $\widehat P$ is closed. Fix an arbitrary model and world, and let an assignment satisfy $\widehat P$ there. Changing the values assigned to $Y_1,\ldots,Y_r$ does not affect the truth of $\widehat P$. Since $\widehat P\to G\theta$ is valid, $G\theta$ is true under each such reassignment. Hence $\widehat P\to\forall Y_1\cdots\forall Y_r\,G\theta$ is valid at every model and world. By Corollary~\ref{cor:three-way-adequacy}, this validity implies $H_{\Amod}\vdash \widehat P\to\forall Y_1\cdots\forall Y_r\,G\theta.$ This is the condition in Definition~\ref{def:hilbert-correct}.
\end{proof}

\subsection{Proof regularization}\label{sec:proof-regularization}

Completeness needs the converse construction. An arbitrary ordinary certificate proof may use concrete witness terms and free object variables. Before those choices can be abstracted into flexible variables, eigenparameter names and neutral proof parameters require a canonical discipline.

\begin{lemma}[Eigenparameter regularization]\label{lem:eigen-regularization}
Let $\mathcal D$ be a finite ordinary certificate derivation and let $Q$ be a finite set of proof parameters. There is a derivation $\mathcal D'$ obtained from $\mathcal D$ by systematic renaming of eigenparameters, with the same root and height, such that each eigenparameter is introduced by exactly one universal rule, occurs only inside the premise subtree rooted at that rule, and does not belong to $Q$.
\end{lemma}

\begin{proof}
List the finitely many universal inferences in a topological order from the root toward the leaves. At each inference choose a proof parameter outside $Q$, outside the finitely many names already chosen, and outside the finitely many parameters occurring elsewhere in the current derivation. Apply Lemma~\ref{lem:uniform-parameter-renaming} to the premise subderivation of that universal inference, replacing its designated eigenparameter by the fresh parameter, and then reattach the universal inference with the new eigenparameter. This preserves the derivation and its height because the new parameter is absent from the conclusion. Continue through the finite list. Later renamings are confined to their own premise subtrees and therefore do not disturb earlier choices. The root is unchanged at every step.
\end{proof}

Eigenparameter regularization controls names introduced by universal rules. Proof parameters that have no introducing universal rule require a different treatment. They must be eliminated as proof parameters before abstraction.

\begin{lemma}[Neutral parameters]\label{lem:neutral-parameter-normalization}
Let $\mathcal D$ be a finite ordinary certificate derivation. All proof parameters that are not introduced by universal eigenparameter rules can be replaced uniformly by distinct fresh free object variables. The result is an ordinary certificate derivation of the root obtained by the same replacement. If the root contains no such neutral proof parameter, the root is unchanged.
\end{lemma}

\begin{proof}
Choose distinct fresh object variables for the finitely many neutral proof parameters. Apply the replacements simultaneously. Truth and atomic initial rules are preserved because equality of complementary argument tuples is preserved. Boolean rules commute with the replacement. In an existential inference the witness term is replaced uniformly and remains a legal proof term. A universal inference retains its designated eigenparameter because only neutral parameters are replaced. Box, certificate diamond, and seriality rules depend on the nested tree and modal certificates, which contain no first-order term data. Hence they are unchanged. If the root contains no neutral proof parameter, the simultaneous replacement has no effect on it.
\end{proof}

We call an ordinary certificate derivation \emph{regularized} when it satisfies Lemma~\ref{lem:eigen-regularization} with $Q=\varnothing$ and contains no neutral proof parameter, as ensured by Lemma~\ref{lem:neutral-parameter-normalization}.

Regularization has one consequence that is essential for abstraction. An eigenparameter occurring in a witness can only have been introduced by a universal inference above that witness.

\begin{lemma}[Parameter ancestry]\label{lem:parameter-ancestry}
Assume the regularizations of Lemmas~\ref{lem:eigen-regularization} and~\ref{lem:neutral-parameter-normalization}. If an eigenparameter $a$ occurs in an existential witness term used at a derivation node $\mathsf n$, then the unique universal inference that introduces $a$ is an ancestor of $\mathsf n$. Consequently every eigenparameter in that witness belongs to the active set $Q_{\mathsf n}$.
\end{lemma}

\begin{proof}
By eigenparameter regularization, $a$ occurs only in the premise subtree of its unique introducing universal rule. Since the witness occurrence lies at $\mathsf n$, the node $\mathsf n$ lies in that subtree. The introducing rule is therefore an ancestor of $\mathsf n$. The definition of $Q_{\mathsf n}$ contains precisely the eigenparameters introduced on the path from the root of the derivation tree to $\mathsf n$.
\end{proof}

\subsection{Proof abstraction}\label{sec:proof-abstraction}

We extend $\equiv_\alpha$ componentwise from proof and symbolic core formulas to nested sequents and derivations. Throughout this subsection, $\mathcal D$ denotes a regularized ordinary certificate derivation. A regularized ordinary proof still contains concrete first-order choices. We replace those choices by flexible variables while retaining a substitution that reconstructs the original proof. Object variables already present in the root remain rigid symbolic terms. Free object variables that occur only internally are represented by residual flexible variables with empty permission sets. This separation preserves the universal role of nonanswer query variables while allowing choices internal to the proof to be abstracted.

\begin{definition}[Residual representation]\label{def:residual-map}
Let $V_{\mathcal D}$ be the finite set of free object variables that occur in a regularized ordinary derivation $\mathcal D$ but do not occur in its root sequent. For each $Z\in V_{\mathcal D}$ choose a fresh auxiliary flexible variable $\widehat Z\in\mathsf{FVar}^{\mathrm U}$ with
$\pi(\widehat Z)=\varnothing.$
For a proof term $t$, let $\widehat t$ be obtained by replacing each free occurrence of $Z\in V_{\mathcal D}$ by $\widehat Z$ and leaving every other symbol unchanged. Extend $B\mapsto\widehat B$ homomorphically to proof core formulas and componentwise to nested sequents. Define the \emph{concretization map} $q[\mathcal D]$ on all symbolic terms by
$q[\mathcal D](\widehat Z)=Z$
for the chosen representatives and $q[\mathcal D](V)=V$ for every other flexible variable $V$. Thus each selected residual representative is mapped to its corresponding ordinary variable, while every other flexible variable is fixed. Extend $q[\mathcal D]$ homomorphically to formulas and componentwise to nested sequents and derivations, while fixing object variables, function symbols, constants, proof parameters, nested positions, rule annotations, and modal certificates.
\end{definition}

A symbolic object is called \emph{residual normal} when every object variable internal to the proof in $V_{\mathcal D}$ occurs only through its chosen residual representative. A substitution is \emph{residual normal relative to $\mathcal D$} when every term in its fully dereferenced range is residual normal. The elementary reflection and injectivity properties needed below are collected in the next lemma.

\begin{lemma}[Properties of residual representation]\label{lem:residual-properties}
For every proof term or proof core formula $B$ occurring in $\mathcal D$,
$q[\mathcal D](\widehat B)=B.$
On residual normal symbolic terms, $q[\mathcal D]$ reflects outer function symbols and is injective. Residual normality is preserved by the abstraction substitutions used below.
\end{lemma}

\begin{proof}
For the first assertion, use structural induction on terms.

Base case. The only nonidentity variable case is $q[\mathcal D](\widehat Z)=Z$. All other atomic term symbols are fixed.

Induction step. Suppose $t=f(t_1,\ldots,t_n)$. Apply the induction hypotheses to the arguments and use homomorphic action through $f$.

Extend the result to formulas by structural induction.

Base case. Atomic formulas follow from the term result. The map $q[\mathcal D]$ fixes the truth constants.

Induction step. Suppose the outermost constructor is Boolean, quantified, or modal. Apply the induction hypotheses to the immediate subformulas and use homomorphic action of $q[\mathcal D]$.

For injectivity on residual normal symbolic terms, use structural induction on the term.

Base case. Distinct rigid nullary symbols and flexible variables remain distinct under $q[\mathcal D]$, and distinct residual representatives map to distinct internal object variables that do not occur directly in a residual normal term.

Induction step. Suppose $t=f(t_1,\ldots,t_n)$. Equality of concretizations forces equality of the outer function symbol and, by the induction hypotheses, equality of all corresponding arguments.

Preservation of residual normality follows by induction on the number of substitution extensions performed during proof abstraction.

Base case. The initial substitution is residual normal by construction.

Induction step. Each new binding assigns a residual normal term to a fresh flexible variable and never replaces a residual representative by a term containing an unrepresented object variable internal to the proof. Hence residual normality is preserved.
\end{proof}

The induction used for abstraction must combine substitutions constructed independently below mandatory branches. A finite substitution $\rho'$ is a \emph{conservative extension} of $\rho$ when $\dom(\rho)\subseteq\dom(\rho')$ and $\rho'(X)=\rho(X)$ for every $X\in\dom(\rho)$. For two such extensions $\rho_1$ and $\rho_2$, put $N_i=\dom(\rho_i)\setminus\dom(\rho)$. We say that they are \emph{independent over $\rho$} when the new domains are disjoint, neither new domain contains a variable occurring in the fully dereferenced range of $\rho$, and no variable from one new domain occurs in the fully dereferenced range of the other extension. Together, these three conditions allow the substitutions to be united without creating a dependency between the two branches.

\begin{lemma}[Branch extension union]\label{lem:branch-extension-union}
Let $\rho$ be a substitution at a derivation node that preserves scope and is residual normal relative to $\mathcal D$. Suppose $\rho_1$ and $\rho_2$ are conservative extensions of $\rho$ constructed in two distinct mandatory branch subtrees. Assume that both preserve scope, are residual normal relative to $\mathcal D$, and are independent over $\rho$. Then $\rho_1\cup\rho_2$ is a well-defined conservative extension of $\rho$ that preserves scope and is residual normal relative to $\mathcal D$.
\end{lemma}

\begin{proof}
Both extensions agree with $\rho$ on its domain. Independence over $\rho$ makes the new domains disjoint and prevents a new variable from one branch from occurring in the fully dereferenced range of the other branch or of the inherited substitution. The union is therefore acyclic and agrees with each branch extension on the fully dereferenced terms belonging to that branch. Since each extension preserves scope, the union does as well. Independence also ensures that every fully dereferenced image in the union is an inherited image or an image from one branch extension, so residual normality is preserved. The union remains a conservative extension of $\rho$.
\end{proof}

With branch combination controlled, the main abstraction construction can now proceed by induction on the ordinary proof. The inherited substitution is extended only by choices made at the matching symbolic rule.

\begin{lemma}[Proof abstraction]\label{lem:general-proof-abstraction}
Let $\mathcal D_{\mathsf n}$ be a regularized ordinary certificate subderivation rooted at a sequent $\mathcal G_{\mathsf n}$, and let $Q_{\mathsf n}$ be the active eigenparameter set there. Suppose a residual normal symbolic root $\mathcal G_{\mathsf n}^\sharp$, a substitution $\rho_{\mathsf n}$ that preserves scope and is residual normal relative to $\mathcal D$, and the map $q[\mathcal D]$ satisfy
$q[\mathcal D](\rho_{\mathsf n}^*(\mathcal G_{\mathsf n}^\sharp))\equiv_\alpha\mathcal G_{\mathsf n}.$
Then there exist a symbolic derivation $\mathcal D_{\mathsf n}^\sharp$, a conservative extension $\rho[\mathcal D_{\mathsf n}]$ of $\rho_{\mathsf n}$ that preserves scope and is residual normal relative to $\mathcal D$, and a finite union $E(\mathcal D_{\mathsf n}^\sharp)$ of terminal equations with the following properties.
\begin{enumerate}
\item $\rho[\mathcal D_{\mathsf n}]$ solves every equation in $E(\mathcal D_{\mathsf n}^\sharp)$.
\item every ordinary existential witness $t$ used at a derivation node $\mathsf n'$ is represented by a fresh proof metavariable $U_{\mathsf n'}\in\mathsf{FVar}^{\mathrm W}$ with $\pi(U_{\mathsf n'})=Q_{\mathsf n'}$ and
$q[\mathcal D](\rho[\mathcal D_{\mathsf n}]^*(U_{\mathsf n'}))=t$.
\item flexible variables introduced strictly below distinct mandatory branch children are disjoint.
\item symbolic terminal leaves concretize to the matching ordinary terminal leaves.
\item
$q[\mathcal D](\rho[\mathcal D_{\mathsf n}]^*(\mathcal D_{\mathsf n}^\sharp)) \equiv_\alpha \mathcal D_{\mathsf n}.$
\end{enumerate}
\end{lemma}

\begin{proof}
We use induction on the height of $\mathcal D_{\mathsf n}$. The construction and all side conditions are treated simultaneously.

Base case.

\begin{caselist}
\item Case truth initial leaf. Use the symbolic truth leaf with no new variable, no equation, and substitution $\rho_{\mathsf n}$. Concretization is immediate.
\item Case atomic identity leaf. Let the complementary atoms be $p(\bar s)$ and $\overline p(\bar s)$. The symbolic root concretizes to that leaf by hypothesis. Close it with the symbolic atomic rule and add equations between the corresponding symbolic argument terms. After full dereferencing by $\rho_{\mathsf n}$, the corresponding symbolic argument terms are residual normal by the inherited invariant. Their images under $q[\mathcal D]$ are the same proof terms, so injectivity from Lemma~\ref{lem:residual-properties} makes the dereferenced argument terms syntactically equal. Hence $\rho_{\mathsf n}$ solves the equations contributed by the symbolic atomic leaf. No branch or witness condition arises.
\end{caselist}

Induction step. Assume the five assertions for all premise subderivations of smaller height.

\begin{caselist}
\item Case unary propositional rule. Apply the same symbolic rule to the matching occurrence in $\mathcal G_{\mathsf n}^\sharp$. Concretization of the premise follows from homomorphicity, and the induction hypothesis applies to the ordinary premise subderivation. The resulting substitution is a conservative extension of $\rho_{\mathsf n}$, and the full derivation concretizes by the induction hypothesis followed by the same ordinary rule.
\item Case conjunction. Apply the symbolic conjunction rule. For each permission fiber in the reserves for flexible variables, partition the names not yet used in the symbolic construction into two disjoint infinite supplies, one for each branch. Choose both supplies outside the finitely many flexible variables already occurring in $\rho_{\mathsf n}$ or in its fully dereferenced range. During each recursive construction, every new flexible variable is drawn from the supply assigned to that branch. Apply the induction hypothesis to both branches with the same inherited substitution $\rho_{\mathsf n}$. The two returned substitutions have disjoint new domains. Neither new domain occurs in the inherited fully dereferenced range, and no variable allocated to one branch occurs in the fully dereferenced range produced in the other branch. Lemma~\ref{lem:branch-extension-union} therefore combines the two substitutions. The combined substitution solves the union of the two finite constraint sets. Branch locality and full concretization follow from the two induction hypotheses.
\item Case universal introduction. Let $a$ be the fresh eigenparameter. Apply the symbolic universal rule with the same $a$. Eigenparameter regularization makes $a$ fresh for the ordinary conclusion. Since the symbolic conclusion concretizes to that conclusion and every flexible variable already present there was created before the universal inference, its permission set excludes the newly introduced parameter $a$. Scope preservation therefore implies that $a$ does not occur in the substituted conclusion, so the symbolic eigen side condition is satisfied. Extend the active set from $Q_{\mathsf n}$ to $Q_{\mathsf n}\cup\{a\}$ and apply the induction hypothesis to the premise. Full concretization then follows by the ordinary universal rule.
\item Case existential witness introduction. Let the rule be applied at the root node $\mathsf n$ with witness proof term $t$. Choose a fresh proof metavariable $U_{\mathsf n}\in\mathsf{FVar}^{\mathrm W}$ with $\pi(U_{\mathsf n})=Q_{\mathsf n}$. Define an extension of $\rho_{\mathsf n}$ by $\rho'(U_{\mathsf n})=\widehat t$. By Lemma~\ref{lem:parameter-ancestry}, every eigenparameter of $t$ lies in $Q_{\mathsf n}$. Every residual representative in $\widehat t$ has empty permission. Hence the new assignment preserves scope. The term $\widehat t$ is residual normal by construction, so the extended substitution remains residual normal relative to $\mathcal D$. Apply the symbolic existential rule with $U_{\mathsf n}$. The concretization equation $q[\mathcal D](\rho'^*(U_{\mathsf n}))=t$ holds by Lemma~\ref{lem:residual-properties}. The induction hypothesis applied to the premise gives the required extension and proof.
\item Case indexed box. Use the same symbolic box rule and apply the induction hypothesis to its premise. No new term or permission data is introduced.
\item Case certificate diamond propagation. Retain the same certificate. The symbolic and ordinary nested trees have the same edge structure, and the certificate contains no first-order term data. Apply the induction hypothesis to the premise and reconstruct the symbolic certificate rule.
\item Case indexed seriality. Create the same fresh child label and apply the induction hypothesis to the premise. The rule contains no term data.
\end{caselist}

These cases cover every primitive rule of the ordinary certificate calculus. Each construction extends the inherited substitution only on fresh variables, preserves scope and residual normality, and preserves every terminal equation already solved. Mandatory branches use disjoint fresh domains, so the five assertions hold simultaneously at every induction step.
\end{proof}

\begin{corollary}[Root abstraction]\label{cor:query-proof-abstraction}
Let $\mathcal D$ be a finite regularized ordinary certificate derivation whose root component contains program formula occurrences $\NNF^{-}(P_1),\ldots,\NNF^{-}(P_n)$ and the query formula occurrence $\NNF^{+}(G)\sigma$, where $\sigma$ is a user answer substitution. Then there exist a symbolic derivation $\mathcal D^\sharp$ of the root from Definition~\ref{def:root-answer-replacement}, an admissible solution $\rho[\mathcal D]$ of its initial unification problem, and a concretization map $q[\mathcal D]$ such that
$q[\mathcal D](\rho[\mathcal D]^*(X^\sharp))=\sigma(X)$
for every $X\in A_G$. The witness, terminal, branch locality, and full concretization clauses of Lemma~\ref{lem:general-proof-abstraction} also hold.
\end{corollary}

\begin{proof}
For each symbolic answer representative define the initial substitution $\rho_0(X^\sharp)=\widehat{\sigma(X)}.$ Every symbolic answer representative has empty permission. Every object variable occurring in $\sigma(X)$ also occurs in the root query instance $G\sigma$ and therefore lies outside $V_{\mathcal D}$. Hence $\widehat{\sigma(X)}=\sigma(X)$. Since $\sigma(X)$ is an ordinary user term, $\widehat{\sigma(X)}$ contains neither a proof parameter nor a flexible variable. It is therefore safe for $X^\sharp$, so $\rho_0$ preserves scope. The fully dereferenced range of $\rho_0$ is residual normal relative to $\mathcal D$. The root concretization lemma gives $q[\mathcal D](\rho_0^*(j_G(\NNF^{+}(G)))) \equiv_\alpha \NNF^{+}(G\sigma).$ Closed program formula occurrences are unchanged. Apply Lemma~\ref{lem:general-proof-abstraction} at the root with $\rho_0$.

The returned substitution preserves scope and solves every terminal equation by the first clause of that lemma. It remains to check the domain condition from Definition~\ref{def:admissible-unifier}. Every $X^\sharp$ in the nonidentity domain of $\rho_0$ occurs in the symbolic root because $A_G\subseteq\FV(G)$. Every later nonidentity binding has a fresh flexible variable introduced in the symbolic derivation as its left side. Hence the nonidentity domain of $\rho[\mathcal D]$ is contained in $C(\mathcal D^\sharp)$, the set of flexible variables occurring in the complete symbolic derivation. The returned substitution is therefore an admissible solution of the initial unification problem.
\end{proof}

\subsection{Scoped lifting}\label{sec:scoped-lifting}

Proof abstraction provides an admissible solution of the symbolic constraints. The scoped mgu theorem factors that solution through the unifier computed by execution. It remains to verify the required output type. The factor substitution must assign ordinary user terms to the projected residual variables.

\begin{definition}[Output substitution]\label{def:ordinary-output-substitution}
An \emph{ordinary output substitution} for a query $G$ is a finite capture-avoiding substitution $\gamma$ on object variables whose range consists of ordinary user terms and whose domain satisfies
$\dom(\gamma)\cap(\FV(G)\setminus A_G)=\varnothing.$
Thus $\gamma$ may instantiate answer variables and fresh variables occurring in answer terms, but it does not instantiate a query variable that was excluded from $A_G$. When $A_G=\FV(G)$, the domain condition is vacuous.
\end{definition}

\begin{lemma}[Output instantiation]\label{lem:output-instantiation}
Assume the extensional factorization
$\rho[\mathcal D]^*(Z)=\mu^*(Z)\delta^*$
on the original flexible variables of an abstracted proof. Let $X\in A_G$ and let $V$ be residual in $\mu^*(X^\sharp)$. Assume
$q[\mathcal D](\mu^*(X^\sharp)\delta^*)=\sigma(X)$
for a user answer substitution $\sigma$. Then
$q[\mathcal D](\delta^*(V))$
is an ordinary user term.
\end{lemma}

\begin{proof}
The map $q[\mathcal D]$ is defined on every symbolic term by Definition~\ref{def:residual-map}, so $q[\mathcal D](\delta^*(V))$ is well-formed even when $\delta^*(V)$ contains flexible variables other than the chosen representatives introduced by proof abstraction. The variable $V$ occurs as a syntactic subterm of $\mu^*(X^\sharp)$. Because residual instantiation and concretization are homomorphic, $q[\mathcal D](\delta^*(V))$ is a syntactic subterm of
$q[\mathcal D](\mu^*(X^\sharp)\delta^*)=\sigma(X).$
The right side is an ordinary user term. Every syntactic subterm of an ordinary user term is again an ordinary user term. Hence the displayed image contains neither proof parameters nor flexible variables.
\end{proof}

Lemma~\ref{lem:output-instantiation} establishes the required property of the output terms. The completeness argument next obtains an ordinary witness proof from a declaratively correct answer. Proof abstraction then constructs the corresponding symbolic derivation.

\begin{lemma}[Witness proof]\label{lem:correct-answer-proof}
If $\sigma\in\Corr^{\mathrm H}_{\Amod}(P,G)$, then there is a finite cut-free ordinary certificate derivation of
$\NNF^{-}(P_1),\ldots,\NNF^{-}(P_n),\NNF^{+}(G)\sigma$
with the convention for the empty program from Definition~\ref{def:root-answer-replacement} when $P=\varnothing$.
\end{lemma}

\begin{proof}
Let $Y_1,\ldots,Y_r$ be the free variables remaining in $G\sigma$. Correctness gives $H_{\Amod}\vdash \widehat P\to\forall Y_1\cdots\forall Y_r\,G\sigma.$ Hilbert soundness gives validity. Universal instantiation in the semantics yields validity of $\widehat P\to G\sigma$. By normalization adequacy, its positive NNF is valid. Cut-free certificate completeness from Theorem~\ref{thm:kernel-completeness} gives a derivation of the corresponding one-sided root. Lemma~\ref{lem:root-normalization} gives the query formula occurrence as $\NNF^+(G\sigma)$. The substitution part of Theorem~\ref{thm:normalization-adequacy} gives $\NNF^+(G\sigma)\equiv_\alpha\NNF^+(G)\sigma$, so after alpha conversion the root is precisely the displayed sequent.
\end{proof}

We can now combine proof existence, regularization, abstraction, and the most general unifier property. This yields the lifting theorem that establishes the completeness direction for computed answers.

\begin{theorem}[Scoped lifting]\label{thm:scoped-lifting}
Let $\sigma\in\Corr^{\mathrm H}_{\Amod}(P,G)$. Then there exist a successful symbolic computation for $(P,G)$ with computed answer $\theta$ and an ordinary output substitution $\gamma$ for $G$ such that
$\sigma(X)=(\theta\gamma)(X)$
for every $X\in A_G$. Moreover, the nonidentity domain of $\gamma$ can be chosen among the ordinary output variables occurring in $\theta$.
\end{theorem}

\begin{proof}
By Lemma~\ref{lem:correct-answer-proof}, choose a cut-free ordinary certificate proof $\mathcal D$ of the query instance determined by $\sigma$. Apply Lemma~\ref{lem:eigen-regularization} with $Q=\varnothing$ and then Lemma~\ref{lem:neutral-parameter-normalization}. Then apply Corollary~\ref{cor:query-proof-abstraction}. We obtain a finite symbolic derivation $\mathcal D^\sharp$ with an admissible solution $\rho[\mathcal D]$ of its constraint set and $q[\mathcal D](\rho[\mathcal D]^*(X^\sharp))=\sigma(X)$ for each answer variable.

Run scoped unification on the finite constraint set. Since $\rho[\mathcal D]$ is an admissible solution, clause 1 of Theorem~\ref{thm:scoped-mgu} rules out failure. Let $\mu$ be the returned scoped mgu. The factorization clause of that theorem gives a substitution $\delta$ that preserves scope such that
$\rho[\mathcal D]^*(Z)=\mu^*(Z)\delta^*$
for every original flexible variable $Z$.

Let $\theta=\operatorname{proj}(\mu)$. For every answer residual variable $V$ with output name $Y_V=r^{\mu}(V)$ define $\gamma(Y_V)=q[\mathcal D](\delta^*(V)).$ This is an ordinary user term by Lemma~\ref{lem:output-instantiation}. Extend $\gamma$ by the identity outside its finite domain. The output variables were chosen fresh for $G$, so Definition~\ref{def:ordinary-output-substitution} makes $\gamma$ an ordinary output substitution for $G$. Homomorphicity gives, for each $X\in A_G$,
$(\theta\gamma)(X)=q[\mathcal D](\mu^*(X^\sharp)\delta^*)=q[\mathcal D](\rho[\mathcal D]^*(X^\sharp))=\sigma(X).$ The symbolic derivation is successful because scoped unification succeeds on its terminal constraints and returns $\mu$.
\end{proof}

\subsection{Fair enumeration}\label{sec:fair-symbolic-search}

The existence result in Theorem~\ref{thm:scoped-lifting} becomes an operational completeness result once the search discipline is specified. The deterministic policy below uses one canonical fresh name at every rule that needs freshness and one canonical certificate for each eligible propagation pair. Eligibility is decidable by Theorem~\ref{thm:certificate-decidability}.

\begin{definition}[Reference derivation]\label{def:reference-symbolic-derivation}
A \emph{reference derivation} is a symbolic derivation generated under a deterministic policy for fresh names and certificates. For a finite partial reference derivation, a proof parameter, proof metavariable, or component name is unused exactly when it does not occur anywhere in that partial derivation. The policy uses the effective enumeration of $\mathsf{Par}$, the fixed effective enumeration of component names, and the effective enumeration of each permission fiber in $\mathsf{FVar}^{\mathrm W}$. A universal rule chooses the first unused proof parameter. An existential rule chooses the first unused proof metavariable in the required permission fiber. A box or seriality rule chooses the first unused component name.

For each eligible propagation pair, certificate selection is deterministic. The canonical certificate returned by $\mathsf{CertSearch}_{\Amod}$ is used. The symbolic rules equipped with this policy form the \emph{reference rule set}. A \emph{reference computation for $(P,G)$} is a symbolic computation for $(P,G)$ whose derivation is a reference derivation. It is successful when that symbolic computation is successful in the sense of Definition~\ref{def:successful-symbolic-derivation}.
\end{definition}

The \emph{primitive inference height} of a search derivation is the maximum number of calculus inferences on a root-to-leaf path. Deterministic bookkeeping required by this policy does not contribute to this height.

\begin{theorem}[Finite branching]\label{thm:symbolic-finite-branching}
For any finite partial reference derivation and any chosen open leaf, there are only finitely many legal reference inferences that can be applied at that leaf.
\end{theorem}

\begin{proof}
At a chosen open leaf, the underlying nested sequent has finitely many formula occurrences and positions. Each propositional rule therefore has finitely many possible principal occurrences. A universal rule uses one canonical fresh eigenparameter for a chosen occurrence. An existential rule uses one canonical fresh proof metavariable. A box rule has one canonical fresh child for a chosen principal occurrence. Certificate propagation contributes finitely many successor choices by Corollary~\ref{cor:finite-modal-choices}. The reference procedure keeps one canonical witness certificate for each eligible propagation pair. Seriality has finitely many applicable indexed obligations and uses one canonical fresh child. Truth and atomic closure have finitely many candidate occurrences. Mandatory branching is determined by a chosen principal formula and has finite arity. Hence the union of all successor choices is finite.
\end{proof}

\begin{lemma}[Finiteness for a fixed height bound]\label{lem:bounded-height-finiteness}
For every symbolic root and every natural number $h$, there are only finitely many reference derivations from that root of primitive inference height at most $h$.
\end{lemma}

\begin{proof}
Fix the symbolic root and consider the construction tree whose nodes are finite partial reference derivations from that root. A child is obtained by choosing one open leaf and applying one legal reference inference there, or by applying one legal terminal closure there. Every node of this construction tree has finitely many children. A finite partial derivation has finitely many open leaves. At each chosen open leaf, Theorem~\ref{thm:symbolic-finite-branching} gives only finitely many legal reference inferences, and truth and atomic closure have only finitely many candidate occurrences. The canonical choices of fresh names are evaluated relative to the whole partial derivation, so the statement that the construction tree is finitely branching already includes the bookkeeping shared across distinct branches.

Fix $h\in\mathbb N$. Every primitive inference has at most two premises. Hence there is a bound $B_h$ on the number of nodes in the underlying rule tree of a derivation of primitive inference height at most $h$. For example, $B_h=2^{h+1}-1$ is sufficient when terminal leaves are counted. Such a derivation is obtained from the root partial derivation by at most $B_h$ construction steps. We prove that a finitely branching construction tree has only finitely many nodes reachable within any fixed finite number of steps by induction on the step bound.

Base case. At step bound zero, only the root partial derivation is reachable.

Induction step. Assume finitely many nodes are reachable within $n$ steps. Each of those nodes has finitely many children, so the union of their child sets is finite. Hence only finitely many nodes are reachable within $n+1$ steps.

Therefore only finitely many reference derivations from the fixed root can have primitive inference height at most $h$.
\end{proof}

The finiteness result for a fixed height bound controls the size of the reference search. The next lemma shows that every successful computation has a canonical representative and that canonical renaming preserves its scoped unification problem.

\begin{lemma}[Reference canonicalization]\label{lem:reference-canonicalization}
Every successful symbolic computation $\mathcal D^\sharp$ for $(P,G)$ has a successful reference computation $\mathcal D^{r}$ for $(P,G)$ with the same root, rule applications, branching structure, and primitive inference height. The two derivations differ only by systematic renaming of fresh proof parameters and witness flexible variables, by consistent renaming of each generated component throughout its descendant subtree, and by replacement of a certificate with the canonical certificate for the same source, target, and modality. The renaming determined by the derivation extends to a permutation $\kappa^{p}$ of all proof parameters and a bijection $\kappa^{f}$ of all flexible variables that preserves each reserve and satisfies
\[
\pi(\kappa^{f}(V))=\kappa^{p}[\pi(V)]
\]
for every $V\in\mathsf{FVar}$, while every symbolic answer representative occurring at the root is fixed. Under the induced homomorphic renaming, the atomic constraint problem of $\mathcal D^\sharp$ is transformed into that of $\mathcal D^{r}$. The global maps are also defined on every auxiliary variable that may later be generated by scoped unification. The computed answers of the two derivations represent exactly the same user answer substitutions after ordinary output substitution and restriction to $A_G$, up to renaming of fresh output variables.
\end{lemma}

\begin{proof}
Choose an order of the inference occurrences in which each inference precedes every inference in its premise subderivations, and define the renaming successively in that order. At each universal inference, rename the fresh proof parameter to the canonical choice determined by the partial reference derivation. At each existential inference, rename the fresh witness flexible variable in the same manner. At each box or seriality inference, rename the generated component to the canonical fresh component. Apply each renaming uniformly throughout the corresponding premise subderivation. The global symbolic eigenparameter discipline makes the induced map $p_0$ on proof parameters a finite injection, and freshness of witness flexible variables gives finite injective partial renamings on those variables. Generated component names require separate treatment because each generated child must be renamed together with every occurrence of its name in the corresponding premise subderivation rather than by one global map on name strings. Thus, if the same unused name was chosen independently in disjoint branches, each generated child is renamed within its own premise subderivation together with all occurrences of its name and all certificate vertices that use it. For a witness flexible variable $U$, first apply $p_0$ to its active eigenparameter set and then choose the canonical unused member $U'$ of the corresponding permission fiber. Such a fresh choice is always possible by Definition~\ref{def:permission-map} because the corresponding permission fiber is infinite. By the choice of $U'$,
\[
\pi(U')=p_0[\pi(U)].
\]
Branch locality is preserved because distinct branch subtrees receive distinct canonical fresh names. At certificate propagation, replace the attached witness by the certificate returned by $\mathsf{CertSearch}_{\Amod}$ for the renamed source, target, and modality. All certificates for that triple license the same formula insertion. This recursive construction gives a legal reference derivation with the same rule applications, branching structure, and primitive inference height.

Apply Lemma~\ref{lem:permission-renaming-extension}. Extend $p_0$ to a permutation $\kappa^{p}$ of $\mathsf{Par}$. The extension agrees with $p_0$ on every proof parameter occurring in the derivation, so every prescribed pair for a witness variable $U\mapsto U'$ still satisfies
\[
\pi(U')=\kappa^{p}[\pi(U)].
\]
Add the prescriptions $X^\sharp\mapsto X^\sharp$ for the symbolic answer representatives of the root. Their permission is empty. Lemma~\ref{lem:permission-renaming-extension} now extends these prescriptions to a bijection
\[
\kappa^{f}:\mathsf{FVar}\longrightarrow\mathsf{FVar}
\]
that preserves each reserve and satisfies $\pi(\kappa^{f}(V))=\kappa^{p}[\pi(V)]$ for every flexible variable $V$. This global extension acts in particular on every auxiliary variable that may later be created by Permission restriction, even though such a variable did not occur in the symbolic derivation before unification.

Let $\kappa$ act on symbolic terms by $\kappa^{p}$ on proof parameters and $\kappa^{f}$ on flexible variables, while fixing object variables and rigid signature symbols. Extend it homomorphically to formulas and constraints. Together with the already fixed renaming of component names, this sends the atomic constraint problem $(E,C)$ of $\mathcal D^\sharp$ to the atomic constraint problem $(\kappa(E),\kappa^{f}[C])$ of $\mathcal D^{r}$.

For a finite substitution $\rho$, define the renamed substitution $\kappa_*\rho$ by renaming the entire substitution graph:
\[
\begin{aligned}
\dom(\kappa_*\rho)&=\kappa^{f}[\dom(\rho)], &
(\kappa_*\rho)(\kappa^{f}(V))&=\kappa(\rho(V)).
\end{aligned}
\]
Because $\kappa^{f}$ is a bijection, graph acyclicity is preserved and reflected. We prove
\[
(\kappa_*\rho)^*(\kappa^{f}(V))=\kappa(\rho^*(V))
\]
by induction on dereference depth.

Base case. At depth zero, $V$ is residual and both sides equal $\kappa^{f}(V)$.

Induction step. Assume the equality for dereference depth at most $n$ and let $V$ have depth $n+1$. The first dereference step replaces $\kappa^{f}(V)$ by $\kappa(\rho(V))$ by definition of $\kappa_*\rho$. Apply the induction hypotheses to the flexible variables occurring in $\rho(V)$ and use homomorphic action of $\kappa$ to obtain the displayed equality.
Permission equivariance gives
\[
\begin{aligned}
\operatorname{Par}(\kappa(t))&=\kappa^{p}[\operatorname{Par}(t)], &
\pi(\kappa^{f}(V))&=\kappa^{p}[\pi(V)].
\end{aligned}
\]
Hence a term is safe for $V$ exactly when its image under $\kappa$ is safe for $\kappa^{f}(V)$. Equation satisfaction is likewise preserved and reflected because $\kappa$ is a bijective homomorphism on terms. Consequently $\rho$ is an admissible solution of $(E,C)$ exactly when $\kappa_*\rho$ is an admissible solution of $(\kappa(E),\kappa^{f}[C])$. The same statement in the reverse direction uses the inverse permutations.

Let $\mu$ be the scoped mgu of $\mathcal D^\sharp$. The renamed substitution $\kappa_*\mu$ is an admissible solution of the reference constraint problem. Theorem~\ref{thm:scoped-mgu}, together with termination, therefore forces the deterministic reference unification run to succeed. Let $\mu^{r}$ be its scoped mgu. By the factorization property of Theorem~\ref{thm:scoped-mgu}, $\kappa_*\mu$ factors through $\mu^{r}$. Conversely, $(\kappa^{-1})_*\mu^{r}$ is an admissible solution of the original problem and therefore factors through $\mu$. Thus the two mgus determine the same extensional solution set on the original flexible variables under the global renaming $\kappa^{f}$. This argument includes residual auxiliary variables generated during either unification run because both $\kappa^{f}$ and its inverse are defined on all of $\mathsf{FVar}$.

The renaming fixes every $X^\sharp$. Lemma~\ref{lem:zero-permission-residual} gives empty permission to every residual occurring at an answer position, so projection on either side merely assigns fresh names of ordinary object variables to corresponding residuals. The factor substitutions obtained in both directions preserve scope. Their values on answer residuals with empty permission therefore contain no proof parameter, and every residual flexible variable occurring there also has empty permission. Rename those residuals by fresh ordinary object variables exactly as in answer projection. Each extensional factorization then becomes a legal ordinary output substitution between the corresponding projected answers. Thus the two computed answers yield exactly the same user answer substitutions after ordinary output substitution and restriction to $A_G$, up to a bijective renaming of the fresh output variables.
\end{proof}

Bounding the primitive inference height makes the global reference search effective. Lemma~\ref{lem:bounded-height-finiteness} gives a finite reference tree at each bound, while Lemma~\ref{lem:reference-canonicalization} shows that restricting the search to reference derivations preserves every user answer obtainable by ordinary output instantiation.

\begin{definition}[Reference enumeration]\label{def:reference-enumeration}
For a fixed symbolic root, the \emph{reference enumeration} uses iterative deepening in primitive inference height. At each bound it effectively enumerates all reference derivations from that root whose primitive inference height is at most that bound. If every branch of a reference derivation ends in a symbolic truth or atomic leaf, scoped unification is applied to its terminal equation set. If unification succeeds, answer projection is applied.
\end{definition}

\begin{theorem}[Fair enumeration]\label{thm:symbolic-fairness}
For every successful symbolic computation for $(P,G)$, the reference enumeration reaches a successful reference computation for $(P,G)$ after finitely many stages. The computed answer of that reference derivation represents exactly the same user answer substitutions after ordinary output substitution and restriction to $A_G$, up to renaming of fresh output variables.
\end{theorem}

\begin{proof}
Apply Lemma~\ref{lem:reference-canonicalization} to obtain a successful reference computation $\mathcal D^{r}$ of primitive inference height $h$. By Lemma~\ref{lem:bounded-height-finiteness}, only finitely many reference derivations have height at most $h$, and Definition~\ref{def:reference-enumeration} enumerates every such derivation at the stage with bound $h$. Hence that stage includes $\mathcal D^{r}$. The statement about represented user answers is the final clause of Lemma~\ref{lem:reference-canonicalization}.
\end{proof}

\begin{theorem}[Answer completeness]\label{thm:answer-completeness}
For every $\sigma\in\Corr^{\mathrm H}_{\Amod}(P,G)$, the reference enumeration eventually returns a computed answer $\theta$ and there is an ordinary output substitution $\gamma$ for $G$ such that
$\sigma=\theta\gamma$
on $A_G$.
\end{theorem}

\begin{proof}
Scoped lifting, Theorem~\ref{thm:scoped-lifting}, provides a successful symbolic computation for $(P,G)$ whose computed answer has $\sigma$ as an ordinary output instance. By Theorem~\ref{thm:symbolic-fairness}, the reference enumeration reaches a successful reference computation for $(P,G)$ representing the same user answer substitutions under ordinary output instantiation, up to renaming of fresh output variables. Hence the answer returned from that reference derivation also has $\sigma$ as an ordinary output instance. Incorporate the renaming of output variables into the ordinary output substitution $\gamma$.
\end{proof}

The theorem establishes completeness of the enumeration. If no proof exists, iterative deepening may continue indefinitely.

\subsection{Instance closure}\label{sec:instance-closure}

Computed answers are most useful when a general answer represents all of its legal ordinary instances. The declarative relation has the same closure property.

\begin{lemma}[Declarative instance closure]\label{lem:correct-instance-closure}
Let $\theta\in\Corr^{\mathrm H}_{\Amod}(P,G)$ and let $\gamma$ be an ordinary output substitution for $G$. Then
$\theta'=(\theta\gamma)\upharpoonright A_G$
belongs to $\Corr^{\mathrm H}_{\Amod}(P,G)$.
\end{lemma}

\begin{proof}
Let $\bar Y$ enumerate $\FV(G\theta)$. Correctness gives $H_{\Amod}\vdash\widehat P\to\forall\bar Y\,G\theta.$ By Hilbert soundness, the displayed formula is valid. Fix an arbitrary model, world, and assignment satisfying $\widehat P$. Universal truth gives $G\theta$ under every reassignment of the variables in $\bar Y$. Evaluate each term in the finite range of $\gamma$ under the current assignment and reassign the corresponding variables to those denotations. Lemma~\ref{lem:formula-substitution} yields $G\theta\gamma$.

Because $\dom(\gamma)$ is disjoint from $\FV(G)\setminus A_G$, $\gamma$ changes no original free occurrence of a query variable excluded from the answer set. Hence applying the composed substitution to $G$ is the same as applying its restriction to the answer variables:
$G\theta\gamma=G\theta'.$
The restriction $\theta'$ is a user answer substitution because its domain is contained in $A_G$ and composition of ordinary user terms again yields ordinary user terms. Thus $\widehat P\to G\theta'$ is valid. Universally close the free variables remaining in $G\theta'$. The program is closed, so the same assignment argument makes the universally closed implication valid. Full Hilbert adequacy from Corollary~\ref{cor:three-way-adequacy} yields the theoremhood condition in Definition~\ref{def:hilbert-correct}.
\end{proof}

This gives the closure property on the declarative side. We now close computed answers under exactly the same class of ordinary output substitutions.

\begin{definition}[Instance closure of computed answers]\label{def:computed-instance-closure}
The \emph{instance closure} $\Inst(\Comp_{\Amod}^{\mathrm M}(P,G))$ is the set of user answer substitutions obtained by applying ordinary output substitutions for $G$ from Definition~\ref{def:ordinary-output-substitution} to computed answers and then restricting the result to $A_G$.
\end{definition}

\begin{corollary}[Answer characterization]\label{cor:computed-answer-characterization}
For every modal module $\Amod$, program $P$, and query $G$, we have
$\Inst(\Comp_{\Amod}^{\mathrm M}(P,G)) = \Corr^{\mathrm H}_{\Amod}(P,G).$
\end{corollary}

\begin{proof}
For the forward inclusion, answer soundness gives correctness of every computed answer and Lemma~\ref{lem:correct-instance-closure} preserves correctness under legal output instantiation. For the reverse inclusion, Theorem~\ref{thm:answer-completeness} gives a computed answer $\theta$ and an ordinary output substitution $\gamma$ for $G$ with $\sigma=\theta\gamma$ for every declaratively correct answer $\sigma$.
\end{proof}

\subsection{Diagnostic examples}\label{sec:diagnostic-examples}

The following examples illustrate three consequences of the symbolic discipline. They are small enough to inspect directly and will also serve as regression tests for an implementation.

\begin{example}[Shared answer across a program disjunction]\label{ex:shared-disjunction}
Let $P=\{p(c_1)\lor p(c_2)\}$ and $G=p(X)$, where $c_1$ and $c_2$ are distinct constants. At the symbolic root, the program occurrence is $\NNF^{-}(p(c_1)\lor p(c_2))=\overline p(c_1)\land\overline p(c_2)$. The conjunction rule therefore creates two mandatory branches. The first branch requires $X^\sharp\doteq c_1$ and the second requires $X^\sharp\doteq c_2$. Since $X^\sharp$ was introduced above the split, it is shared by both branches. The union of constraints is
$\{X^\sharp\doteq c_1,X^\sharp\doteq c_2\}$.
Binding either equation reduces the other to an equation between the distinct rigid constants $c_1$ and $c_2$, which fails by clash. Thus there is no substitution answer for $p(X)$ when $c_1$ and $c_2$ are distinct. By contrast, for the closed query $\exists x\,p(x)$, the existential rule may introduce distinct witness metavariables in the two branches after the split. Atomic closure can constrain them independently to $c_1$ and $c_2$, so scoped unification can succeed.
\end{example}

\begin{example}[General answer from universal information]\label{ex:universal-answer}
Let $P=\{\forall y\,p(f(y))\}$ and $G=p(X)$. At the symbolic root, the program occurrence is $\NNF^{-}(\forall y\,p(f(y)))=\exists y\,\overline p(f(y))$. The symbolic existential rule introduces a proof metavariable $U$ whose permission is the active eigenparameter set at that node. No universal inference precedes this step, so the active set is empty and $\pi(U)=\varnothing$. Atomic closure yields
$X^\sharp\doteq f(U).$
The term $f(U)$ is safe for $X^\sharp$. It contains no proof parameter, and $\pi(U)\subseteq\pi(X^\sharp)=\varnothing$. Hence the binding is permitted. The scoped unification run therefore returns an mgu $\mu$ with
$\mu(X^\sharp)=f(U).$
The variable $U$ is outside the nonidentity domain of $\mu$ and occurs in $\mu^*(X^\sharp)=f(U)$, so it is an answer residual variable. Projection renames it to an ordinary output variable $Y$. The computed answer is
$X=f(Y).$
Its declarative correctness is the theoremhood of the universal closure of $p(f(Y))$ under the program, which follows directly from the program formula.
\end{example}

\begin{example}[No substitution answer from an existential assumption]\label{ex:no-answer-from-existential-assumption}
Let $P=\{\exists y\,p(y)\}$. Take $G=p(X)$. At the symbolic root, the program occurrence is $\NNF^{-}(\exists y\,p(y))=\forall y\,\overline p(y)$. The universal rule therefore introduces a fresh eigenparameter $a$. Symbolic atomic closure then requires
$X^\sharp\doteq a.$
Since $\pi(X^\sharp)=\varnothing$, the forbidden parameter failure rule applies. No substitution answer is returned. This agrees with the declarative semantics because $\exists y\,p(y)$ does not entail $p(t)$ for any particular user term $t$ in general.
\end{example}

\section{Focused execution}\label{sec:focusing}

The symbolic search of Section~\ref{sec:symbolic} is complete, but it allows many interleavings of invertible and noninvertible rules. This section introduces a focused scheduling discipline. The general proof-theoretic background comes from focusing in sequent calculi and its later extension to nested modal systems~\cite{andreoli1992focusing,liang2009focusing,chaudhuri2016focusednested}. The proof here is specific to the certificate calculus and uses a finite syntactic normalization based on proof transformations and certificate persistence. Semantic completeness then follows as a corollary of ordinary completeness and focalization. The final subsection extends the scoped answer results to focused symbolic execution.

\subsection{Operational polarity}\label{sec:operational-polarity}

Throughout this section, a \emph{state} means an indexed nested sequent.

Define the \emph{logical rank} $r(A)$ recursively. Atomic literals, $\top$, and $\bot$ have rank zero. For binary formulas,
$r(A\land B)=r(A\lor B)=1+r(A)+r(B).$
For quantified formulas,
$r(\forall xA)=r(\exists xA)=1+r(A).$
For modal formulas,
$r(\Box_iA)=r(\Diamond_iA)=1+r(A).$

The \emph{asynchronous heads} are $\lor$, $\land$, $\forall$, and $\Box_i$.
The \emph{synchronous heads} are $\exists$ and $\Diamond_i$. The seriality rules $\mathrm d_i$ are synchronous. Atomic literals and $\bot$ have no decomposition rule. A component containing $\top$ may close in any phase.

Certificate persistence under backward proof steps is Corollary~\ref{cor:certificate-persistence}, and invariance under changes to first-order terms is Lemma~\ref{lem:certificate-term-invariance}.

\subsection{Maximal asynchronous phases}\label{sec:asynchronous-phases}

For scheduling purposes, fix an effective linear order on the formula occurrences of every state and use it to assign distinct \emph{creation indices} $0,1,2,\ldots$ when the scheduling discipline is initialized. Existing occurrences retain their indices along the derivation. Whenever an inference creates new formula occurrences in one premise, assign them the next consecutive natural numbers above every index in the conclusion state, in the order in which the new occurrences are displayed by the rule. In different premise branches, inherited indices are maintained independently. Creation indices are auxiliary scheduling annotations preserved by transformations that change only first-order names or terms.

\begin{definition}[Maximal asynchronous phase]\label{def:maximal-asynchronous}
A \emph{maximal asynchronous phase} from a state repeatedly applies only the rules for $\lor$, $\land$, $\forall$, and $\Box_i$ as long as an asynchronous occurrence remains. At each step choose the least creation index among the available asynchronous occurrences. A conjunction creates two mandatory branches, and maximal asynchronous expansion continues independently in both. A terminal branch may close immediately. An open state with no asynchronous occurrence is a \emph{border state}.
\end{definition}

For a state $S$, let $w(S)$ be the sum of $r(A)$ over all formula occurrences $A$ in $S$, counted with multiplicity.

\begin{lemma}[Termination of asynchronous phases]\label{lem:async-termination}
Every asynchronous inference strictly decreases $w$ along each resulting premise branch. Consequently every maximal asynchronous phase is a finite tree.
\end{lemma}

\begin{proof}
A disjunction step replaces rank contribution $1+r(A)+r(B)$ by $r(A)+r(B)$, so $w$ decreases by one. In each conjunction premise the contribution becomes either $r(A)$ or $r(B)$, so $w$ decreases strictly. A universal step replaces $\forall xA$ by $A(a/x)$, and term substitution leaves logical rank unchanged, so $w$ decreases by one. A box step removes the box occurrence and places its matrix in a fresh child, again decreasing $w$ by one. Thus $w$ strictly decreases along every branch of the phase. Branching degree is at most two, and the initial value of $w$ is a natural number. Hence the phase has bounded depth and finite branching, so its tree is finite.
\end{proof}

\subsection{Progressive synchronous focus}\label{sec:synchronous-focus}

The existential and certificate diamond rules retain their principal formulas. Repeated use of the same retained occurrence would allow a synchronous phase to continue indefinitely. The focus discipline therefore follows only the newly introduced matrix occurrence.

\begin{definition}[Progressive synchronous focus]\label{def:progressive-focus}
At an open border state, a \emph{progressive synchronous focus} begins with one of the following choices:
\begin{enumerate}
\item an occurrence of $\exists xA$ together with one proof term $t$.
\item an occurrence of $\Diamond_iA$ at component $u$, a target component $v$, and a certificate $u\leadstoA{i}v$.
\item a seriality trigger $\mathrm d_i$ at a component $u$ when $\mathrm{D}_i\in\Amod$.
\end{enumerate}
An existential focus performs one retaining existential step and makes the newly created occurrence $A(t/x)$ the distinguished occurrence. A diamond focus inserts $A$ at the certified target and makes that new occurrence distinguished. After either kind of step, the focus follows only the successive distinguished occurrences created from this one. It continues when the distinguished formula again has head $\exists$ or $\Diamond_j$ and the corresponding rule is currently applicable. It stops when the distinguished head is asynchronous, atomic, or $\bot$, and it also stops when a distinguished diamond has no eligible target. A seriality focus consists of one seriality inference.
\end{definition}

\begin{lemma}[Termination of synchronous focus]\label{lem:sync-termination}
Every progressive synchronous focus is finite.
\end{lemma}

\begin{proof}
Let $F_0,F_1,\ldots$ be the successive distinguished occurrences. If $F_n=\exists xA$, the next distinguished formula is $A(t/x)$, whose rank is $r(A)<1+r(A)=r(F_n)$. If $F_n=\Diamond_iA$, the next distinguished formula is $A$, again of strictly smaller rank. Hence rank strictly decreases whenever focus continues. No infinite descending sequence of natural numbers exists. Seriality focus has one step by definition.
\end{proof}

\begin{definition}[Focused ordinary derivation]\label{def:focused-ordinary}
A finite ordinary derivation is \emph{focused} when every open branch alternates a maximal asynchronous phase, a border state, and one progressive synchronous focus, after which the same pattern repeats. Terminal closure may occur in any phase. We write $\mathcal F_{\Amod}\vdash S$ when $S$ has such a derivation.
\end{definition}

\begin{lemma}[Focused regularization]\label{lem:focused-regularization}
The transformations in Lemmas~\ref{lem:eigen-regularization} and~\ref{lem:neutral-parameter-normalization} preserve a focused ordinary derivation's creation indices, maximal asynchronous phase boundaries, progressive focus entry and exit points, and sequence of distinguished occurrences. Eigenparameter regularization also preserves height. Neutral parameter normalization preserves height and leaves the root unchanged whenever the root contains no neutral proof parameter.
\end{lemma}

\begin{proof}
Both transformations change only first-order names inside terms. They preserve logical constructors, nested positions, modal indices, rule ancestry, occurrence identities, and the distinction between newly introduced and retained formula occurrences. Hence creation indices and all phase and focus annotations remain unchanged. The underlying rule tree is unchanged, so height is preserved. The root clause for neutral parameter normalization is Lemma~\ref{lem:neutral-parameter-normalization}.
\end{proof}

Erasing the focused scheduling annotations gives the immediate direction of focalization.

\begin{lemma}[Erasure from focused proofs]\label{lem:focused-erasure}
If
$\mathcal F_{\Amod}\vdash S,$
then
$\mathcal N_{\Amod}\vdash S.$
\end{lemma}

\begin{proof}
Erase creation indices, phase markers, border markers, focus entry and exit markers, and distinguished occurrence annotations. Every asynchronous and synchronous inference is already a primitive rule of the ordinary certificate calculus. Terminal leaves are ordinary truth or atomic initial sequents. The underlying derivation tree and its root remain unchanged.
\end{proof}

\subsection{Proof transformations}\label{sec:focus-transformations}

The construction that transforms ordinary derivations into focused derivations uses three structural facts. Their statements are kept separate because their side conditions are needed explicitly in the focalization proof.

\begin{lemma}[Component renaming]\label{lem:component-renaming}
Let $\mathcal D$ be a finite ordinary derivation and let $g$ be an injective renaming of its finitely many component names. Uniformly apply $g$ to every nested position and every vertex in certificate walks. The result is an ordinary derivation of the renamed root with the same height.
\end{lemma}

\begin{proof}
We use induction on derivation height.

Base case. Truth and atomic initial sequents are invariant under component renaming.

Induction step. Assume the claim for the premise derivations of the final inference.

\begin{caselist}
\item Case Boolean or quantifier rule. These rules are invariant under component renaming. Apply the induction hypotheses and reapply the rule.
\item Case box or seriality. Injectivity preserves freshness of the generated child. Apply the induction hypothesis to the premise and reapply the rule.
\item Case certificate diamond propagation. The renamed walk has the same signed edge word as the original. The grammar derivation depends on that word and the fixed grammar, so the same grammar derivation certifies the renamed walk. Apply the induction hypothesis and reapply the propagation rule.
\end{caselist}
\end{proof}

\begin{lemma}[Deep formula weakening]\label{lem:deep-weakening}
Let $\mathcal D$ be a finite ordinary derivation of $S$. Add finitely many formula occurrences at component positions already present in $S$, without adding or removing nested edges, and call the result $S^+$. Then $S^+$ has an ordinary derivation of the same height as $\mathcal D$.
\end{lemma}

\begin{proof}
Let $Q$ be the finite set of proof parameters occurring in the added formulas. First apply Lemma~\ref{lem:eigen-regularization} with this set $Q$. Keep each added occurrence unchanged as passive context at its fixed component throughout the derivation. When creation indices are present, retain the indices of the original occurrences in $S$. Assign the added root occurrences the next consecutive indices in the fixed auxiliary order. Then process the inference occurrences in increasing depth from the root, independently within distinct premise branches. At each inference, assign the indices prescribed by the rule above to every formula occurrence newly created in its premises. Later occurrences are therefore reindexed relative to the enlarged root, so no collision with an added occurrence can arise.

Initial leaves remain initial because the original $\top$ or complementary atomic pair is still present. Boolean and existential rules reapply to their original principal occurrences with the new formulas passive. At a universal rule, its eigenparameter remains fresh because it was chosen outside $Q$. Box and seriality retain the freshness of their original child since no edge is added by formula weakening. Certificate diamond propagation keeps the same certificate because the indexed tree is unchanged. Thus every inference is copied without an additional proof level, and the height is preserved.
\end{proof}

\begin{theorem}[Height-preserving asynchronous invertibility]\label{thm:async-invertibility}
Let $\mathcal D$ be a finite ordinary derivation of height $h$ and let $\mathcal G\{\}$ be an arbitrary deep nested context.
\begin{enumerate}
\item From a derivation of $\mathcal G\{A\lor B\}$ one can construct a derivation of $\mathcal G\{A,B\}$ of height at most $h$.
\item From a derivation of $\mathcal G\{A\land B\}$ one can construct derivations of both $\mathcal G\{A\}$ and $\mathcal G\{B\}$, each of height at most $h$.
\item From a derivation of $\mathcal G\{\forall xA\}$ and a parameter $a$ fresh for the entire derivation and desired conclusion, one can construct a derivation of $\mathcal G\{A(a/x)\}$ of height at most $h$.
\item From a derivation of $\mathcal G\{\Box_iA\}$ and a component name $v$ absent from the entire derivation, one can construct a derivation of the exact box premise $\mathcal G\{[A]_i\}$ with the newly generated child named $v$, of height at most $h$.
\end{enumerate}
\end{theorem}

\begin{proof}
Choose the globally fresh parameter $a$ needed for the universal clause and apply Lemma~\ref{lem:eigen-regularization} with $Q=\{a\}$. Prove the four clauses simultaneously by induction on $h$.

Base case. At height zero the derivation is an initial sequent. The displayed asynchronous occurrence is passive relative to the original closing $\top$ or complementary atom pair. Performing any of the four backward transformations preserves that closing data, so the transformed state is again initial.

Induction step. Let $R$ be the final inference and assume the four claims for premise derivations of smaller height.

\begin{caselist}
\item Case $R$ is principal on the displayed occurrence. The disjunction and conjunction conclusions have precisely the required premises. For a universal principal inference, uniformly rename its eigenparameter to the prescribed fresh $a$ in the premise subderivation. For a box principal inference, use Lemma~\ref{lem:component-renaming} to rename its fresh child to $v$. These transformations preserve height.
\item Case $R$ is nonprincipal. Apply the appropriate induction hypothesis to every premise subderivation and then reapply $R$. Boolean and existential rules have no relevant freshness condition. If $R$ is universal with eigenparameter $b$, then $b$ was absent from the old conclusion. The transformed passive formula is built from an old conclusion formula and, in the universal inversion clause, the globally fresh parameter $a$, so it contains no $b$. If $R$ is box or seriality, the globally fresh component $v$ differs from its child and does not invalidate child freshness. If $R$ is certificate diamond propagation, the first three inversions leave the tree unchanged, while box inversion adds one fresh edge. Certificate persistence from Corollary~\ref{cor:certificate-persistence} preserves the old certificate in either situation. Reapplying $R$ adds no height beyond the old bound.
\end{caselist}
\end{proof}

\subsection{Syntactic focalization}\label{sec:syntactic-focalization}

\begin{theorem}[Syntactic focalization]\label{thm:syntactic-focalization}
For every indexed nested sequent $S$,
\[
 \mathcal N_{\Amod}\vdash S
 \Longleftrightarrow
 \mathcal F_{\Amod}\vdash S.
\]
The implication from ordinary to focused derivability is given by a syntactic construction.
\end{theorem}

\begin{proof}
The implication from focused to ordinary derivability follows from Lemma~\ref{lem:focused-erasure}. For the converse, let $\mathcal D$ be an ordinary derivation of $S$. Use the lexicographic measure $M(\mathcal D,S)=(h(\mathcal D),w(S)),$ with derivation height primary and asynchronous weight secondary. We use well-founded induction on this pair and preserve creation indices throughout the construction.

Base case. If $S$ is terminal, close immediately. This is legal in every phase.

Induction step. Assume the construction for every pair strictly smaller than $(h(\mathcal D),w(S))$.

\begin{caselist}
\item Case $S$ contains an asynchronous occurrence. Choose the least such occurrence by creation index. Depending on its head, choose a globally fresh eigenparameter or component name when required. Apply height-preserving asynchronous invertibility, Theorem~\ref{thm:async-invertibility}, to obtain ordinary derivations of the exact premises of the required asynchronous rule. Every resulting derivation has height at most $h(\mathcal D)$, while Lemma~\ref{lem:async-termination} gives strictly smaller asynchronous weight in every premise state. Hence every resulting pair consisting of a derivation and its state is smaller than $(h(\mathcal D),w(S))$. Apply the induction hypothesis to all premises and place the chosen asynchronous inference below them. Since it was the least asynchronous occurrence and new occurrences receive later creation indices, the resulting initial segment is precisely the required maximal asynchronous phase. The conjunction case continues independently in its two mandatory branches.
\item Case $S$ is a nonterminal border state and the final rule is seriality. Then $h(\mathcal D)>0$. The final inference cannot be asynchronous because its asynchronous principal formula would still occur in the conclusion. Use seriality as a one-step focus. Its premise derivation has strictly smaller height, so the induction hypothesis applies directly.
\item Case $S$ is a nonterminal border state and the final rule is existential or certificate diamond propagation. Use that inference as the first synchronous step and mark its newly created matrix occurrence. Its premise subderivation has height strictly below $h(\mathcal D)$. Continue the progressive focus while the distinguished occurrence has a synchronous head. For a distinguished existential formula, choose a fresh object variable $y$ as witness. The new state differs from the current state only by one added matrix occurrence at an old component. By Lemma~\ref{lem:deep-weakening}, that state has an ordinary derivation of height at most the current proof height. Use this derivation for the next stage of the focus. For a distinguished diamond with an eligible target, choose one certified target and perform the certificate rule. Again only one formula is added at an old component. By the same lemma, the resulting state has an ordinary derivation of height at most the current proof height. Focus stops under the exit conditions in Definition~\ref{def:progressive-focus}.
\end{caselist}

Every continuation step strictly lowers the rank of the distinguished formula by Lemma~\ref{lem:sync-termination}, so the continuation is finite. Let $S'$ be the final state and $\mathcal D'$ the ordinary derivation obtained after the focus. Repeated deep weakening preserves the strict height decrease inherited from the first synchronous premise, so $h(\mathcal D')<h(\mathcal D).$ The induction hypothesis applies to $(\mathcal D',S')$. Attach that focused derivation above the finite synchronous focus.

Every recursive call strictly decreases the well-founded measure. Thus the construction terminates and produces a finite focused derivation of $S$.
\end{proof}

Ordinary semantic completeness and the focalization theorem together give semantic completeness of the focused calculus.

\begin{corollary}[Focused semantic completeness]\label{cor:focused-semantic-completeness}
If an indexed nested sequent $S$ is valid in every constant-domain model with rigid terms based on $\Frames(\Amod)$, then
$\mathcal F_{\Amod}\vdash S.$
\end{corollary}

\begin{proof}
Kernel completeness, Theorem~\ref{thm:kernel-completeness}, gives $\mathcal N_{\Amod}\vdash S$. Syntactic focalization, Theorem~\ref{thm:syntactic-focalization}, then gives $\mathcal F_{\Amod}\vdash S$.
\end{proof}

\subsection{Symbolic focusing}\label{sec:focused-symbolic}

The symbolic focused calculus uses the same phase discipline. An existential focus introduces a fresh proof metavariable $U$ with permission set equal to the active eigenparameter set at focus entry. Certificate diamond and seriality use the ordinary focused nested transformations. Asynchronous rules are the symbolic rules from Definition~\ref{def:symbolic-derivation}. Truth and symbolic atomic closure are permitted in every phase. A \emph{focused symbolic derivation} is a symbolic derivation in the sense of Definition~\ref{def:symbolic-derivation} equipped with creation indices assigned as above and annotations for phases, borders, focuses, and distinguished occurrences, satisfying Definitions~\ref{def:maximal-asynchronous} and~\ref{def:progressive-focus}.

\begin{definition}[Focused answers]\label{def:focused-computed-answers}
A \emph{focused computation for $(P,G)$} is a symbolic computation for $(P,G)$ whose derivation is a focused symbolic derivation. It is successful when the underlying symbolic computation is successful. A \emph{focused reference computation} is a focused computation whose underlying derivation is a reference derivation in the sense of Definition~\ref{def:reference-symbolic-derivation}. Its computed answer is obtained by the same projection $\operatorname{proj}$. Let $\Comp_{\Amod}^{\mathrm F}(P,G)$ be the set of answers produced by successful focused computations for $(P,G)$. Its instance closure uses the ordinary output substitutions for $G$ from Definition~\ref{def:ordinary-output-substitution}.
\end{definition}

The focused phase boundaries must also respect the permission discipline. The next lemma states the only point at which a new eigenparameter can change the active permission set.

\begin{lemma}[Eigenparameter permissions]\label{lem:focus-permissions}
The active eigenparameter set is constant during a synchronous focus. Suppose a focus ends because its distinguished formula has universal head, and the following asynchronous phase introduces a fresh eigenparameter $a$. Every metavariable introduced before that universal inference has a permission set excluding $a$. Every metavariable introduced later in its premise subtree has a permission set containing $a$ whenever $a$ is active at the variable's introduction node.
\end{lemma}

\begin{proof}
Synchronous existential, diamond, and seriality steps introduce no eigenparameter. Therefore the active set does not change inside a focus. Permission sets are fixed when variables are introduced and are never enlarged later. The universal rule introduces $a$ only after the focus ends, so earlier variables exclude it. In the premise subtree, $a$ belongs to the active set and hence to the permissions assigned to subsequently created metavariables.
\end{proof}

Because permission assignment is stable across phase boundaries, erasing the scheduling annotations recovers the unfocused symbolic calculus.

\begin{lemma}[Focused erasure]\label{lem:focused-symbolic-erasure}
Erasing the creation indices and annotations for phases, borders, focuses, and distinguished occurrences from a successful focused computation gives a legal symbolic computation with the same root and terminal equation set.
\end{lemma}

\begin{proof}
Every focused asynchronous rule is the corresponding unfocused symbolic rule. The symbolic existential rule uses the same fresh metavariable and permission assignment as the unfocused calculus. Certificate diamond and seriality are unchanged. Truth and atomic closures available in every phase are exactly the symbolic terminal closures of Section~\ref{sec:symbolic-closure}. Branch locality depends only on ancestry. Erasure of the scheduling annotations therefore changes neither formulas nor generated equations.
\end{proof}

\begin{theorem}[Focused soundness]\label{thm:focused-answer-soundness}
Let $P$ be a program and $G$ a query. If $\theta\in\Comp_{\Amod}^{\mathrm F}(P,G)$, then $\theta\in\Corr^{\mathrm H}_{\Amod}(P,G)$.
\end{theorem}

\begin{proof}
Choose a successful focused computation producing $\theta$. Lemma~\ref{lem:focused-symbolic-erasure} erases its scheduling annotations. The constraint set and hence its scoped mgu and projected answer are unchanged. Apply unfocused answer soundness, Theorem~\ref{thm:answer-soundness}.
\end{proof}

For the converse direction, a focused witness derivation is needed. Proof abstraction preserves the focused schedule because it changes terms while retaining outer logical constructors, nested positions, and rule ancestry.

\begin{lemma}[Focus preservation]\label{lem:focus-preservation}
Apply Corollary~\ref{cor:query-proof-abstraction} to a finite regularized focused ordinary derivation. The resulting symbolic derivation is a legal focused symbolic derivation. It has the same creation indices, maximal asynchronous phase boundaries, progressive focus entry and exit points, and sequence of distinguished occurrences as the original focused ordinary derivation.
\end{lemma}

\begin{proof}
Inspect the rule cases of Lemma~\ref{lem:general-proof-abstraction}. Disjunction, conjunction, universal, and box retain their outer constructors and positions and therefore remain asynchronous at the same points. An ordinary existential focus step $\exists xA\mapsto A(t/x)$ becomes the symbolic step $\exists xA\mapsto A(U/x)$ at the same position. Both matrices have the same outer logical constructor because term replacement does not change formula structure. Certificate diamond propagation retains the same source, target, certificate, and matrix head. Seriality retains the same source and modality. Terminal leaves remain terminal leaves available in every phase. No rule occurrence is added or removed, so creation indices, phase boundaries, and the sequence of distinguished occurrences are unchanged.
\end{proof}

\begin{theorem}[Focused lifting]\label{thm:focused-scoped-lifting}
Let $\sigma\in\Corr^{\mathrm H}_{\Amod}(P,G)$. Then there exist a successful focused computation for $(P,G)$ with computed answer $\theta$ and an ordinary output substitution $\gamma$ for $G$ such that
$\sigma(X)=(\theta\gamma)(X)$
for every $X\in A_G$.
\end{theorem}

\begin{proof}
Choose the ordinary witness proof from Lemma~\ref{lem:correct-answer-proof}. Syntactic focalization, Theorem~\ref{thm:syntactic-focalization}, gives a focused ordinary proof of the same root. Apply eigenparameter regularization and neutral parameter normalization. Lemma~\ref{lem:focused-regularization} preserves the focused phase boundaries and distinguished occurrence structure under both transformations. Corollary~\ref{cor:query-proof-abstraction} then gives a symbolic derivation with an admissible solution. Lemma~\ref{lem:focus-permissions} verifies that the permission discipline is unchanged across the focused phase boundaries, and Lemma~\ref{lem:focus-preservation} makes the abstracted derivation a legal focused symbolic derivation. The scoped mgu theorem factors the admissible solution through the computed scoped mgu. The output instantiation argument from Theorem~\ref{thm:scoped-lifting} applies without change and yields an ordinary output substitution $\gamma$ for $G$ such that $\sigma(X)=(\theta\gamma)(X)$ for every $X\in A_G$.
\end{proof}

As in the unfocused search, finite branching and iterative deepening combine with focused lifting to establish completeness of enumeration after canonicalization. Here the \emph{primitive focused inference height} counts underlying calculus inferences on a root-to-leaf path and ignores phase markers, focus markers, and deterministic bookkeeping.

\begin{lemma}[Focused canonicalization]\label{lem:focused-reference-canonicalization}
Every successful focused computation for $(P,G)$ has a successful focused reference computation for $(P,G)$ with the same primitive focused inference height. The computed answer of the reference computation has the same instance closure on $A_G$ as the computed answer of the original computation, up to renaming of fresh output variables.
\end{lemma}

\begin{proof}
Perform the construction of Lemma~\ref{lem:reference-canonicalization} while retaining all scheduling annotations, namely the creation indices, phase and border markers, focus entry and exit markers, and annotations for distinguished occurrences. Use the maps $\kappa^{p}$ and $\kappa^{f}$ constructed in Lemma~\ref{lem:reference-canonicalization}, where $\kappa^{p}$ is a permutation of proof parameters and $\kappa^{f}$ is a bijection on flexible variables. The bijection preserves each reserve, and the two maps satisfy $\pi(\kappa^{f}(V))=\kappa^{p}[\pi(V)]$ for every $V\in\mathsf{FVar}$. The domain of $\kappa^{f}$ is all of $\mathsf{FVar}$, including auxiliary variables that may later be generated by focused scoped unification. These renamings change no logical constructor or rule ancestry, and replacing a certificate by the canonical certificate for the same source, target, and modality changes no formula insertion. Therefore maximal asynchronous phase boundaries, progressive focus boundaries, and the sequences of distinguished occurrences are preserved. The global constraint renaming and the two factorization clauses used in Lemma~\ref{lem:reference-canonicalization} apply unchanged, so the instance closure is preserved.
\end{proof}

\begin{theorem}[Focused fairness]\label{thm:focused-symbolic-fairness}
Using the deterministic choices of fresh names and certificates from Definition~\ref{def:reference-symbolic-derivation} together with the assignment of creation indices, iterative deepening over primitive focused inference height reaches a focused reference computation for every successful focused computation after finitely many stages. The computed answer of the resulting representative has the same instance closure on $A_G$ as the computed answer of the original focused computation, up to renaming of fresh output variables.
\end{theorem}

\begin{proof}
Apply Lemma~\ref{lem:focused-reference-canonicalization} and let $h$ be the primitive focused inference height of the resulting focused reference computation. Erasing the scheduling annotations gives an underlying reference derivation of primitive inference height $h$. By Lemma~\ref{lem:bounded-height-finiteness}, only finitely many underlying reference derivations from the fixed root have height at most $h$. For each such derivation, the creation indices are uniquely determined by the fixed assignment above. Since the derivation has finitely many rule and formula occurrences, there are only finitely many possible choices of phase boundaries, focus boundaries, and distinguished occurrences. Hence there are only finitely many focused reference derivations of primitive focused inference height at most $h$. The focused reference derivations at each fixed height bound can be enumerated effectively. At a border state there are finitely many existential occurrences and seriality sites, while Corollary~\ref{cor:finite-modal-choices} gives finitely many eligible propagation pairs. Asynchronous choice and the selection of fresh names are deterministic, and every progressive focus terminates by Lemma~\ref{lem:sync-termination}. Iterative deepening therefore exhausts the finite stage at bound $h$ and reaches the canonicalized computation.
\end{proof}

\begin{theorem}[Focused answer characterization]\label{thm:focused-answer-characterization}
For every program $P$ and query $G$,
\[
 \Inst(\Comp_{\Amod}^{\mathrm F}(P,G))
 =\Corr^{\mathrm H}_{\Amod}(P,G)
 =\Inst(\Comp_{\Amod}^{\mathrm M}(P,G)).
\]
\end{theorem}

\begin{proof}
Theorem~\ref{thm:focused-answer-soundness} and Lemma~\ref{lem:correct-instance-closure} give $\Inst(\Comp_{\Amod}^{\mathrm F}(P,G)) \subseteq\Corr^{\mathrm H}_{\Amod}(P,G).$ For the reverse inclusion, Theorem~\ref{thm:focused-scoped-lifting} provides a successful focused computation whose computed answer is more general than any chosen declaratively correct answer. Theorem~\ref{thm:focused-symbolic-fairness} reaches a focused reference computation with the same instance closure, so the chosen declaratively correct answer remains an ordinary output instance of an answer represented by that focused reference computation. This gives the first equality. The second equality is Corollary~\ref{cor:computed-answer-characterization}.
\end{proof}

The following corollary states the relation between individual focused and unfocused answers.

\begin{corollary}[Focused coverage of unfocused answers]\label{cor:focused-covers-unfocused}
For every program $P$, query $G$, and $\eta\in\Comp_{\Amod}^{\mathrm M}(P,G)$, there are $\theta\in\Comp_{\Amod}^{\mathrm F}(P,G)$ and an ordinary output substitution $\gamma$ for $G$ such that
$\eta=\theta\gamma$
on $A_G$.
\end{corollary}

\begin{proof}
By Definition~\ref{def:computed-answer}, $\eta\in\Comp_{\Amod}^{\mathrm M}(P,G)$ implies $\MMLP_{\Amod}(P,G)\Downarrow\eta$. Theorem~\ref{thm:answer-soundness} therefore gives $\eta\in\Corr^{\mathrm H}_{\Amod}(P,G)$. Apply Theorem~\ref{thm:focused-scoped-lifting} with $\sigma=\eta$.
\end{proof}

\section{Answer conservativity for \texorpdfstring{$\KSys$-\MProlog}{KDI4s5-MProlog}}\label{sec:conservativity}

The preceding sections establish the declarative and operational theory of \MMLP{}. We now compare the system with Nguyen's pure logical $\KSys$-\MProlog{} calculus~\cite{nguyen2006multimodal}. The comparison first establishes a semantic equivalence on the common fragment and then uses the computed answer completeness theorems for the two systems.

\subsection{Common fragment}\label{sec:mprolog-common-fragment}

Fix $\Idx=\{1,\ldots,m\}$. The \emph{common fragment} is the overlap determined by the following restrictions. It excludes Prolog built-ins, cut, negation as failure, side effects, and other operations outside the logic. Terms are interpreted over one constant domain with rigid constants and rigid function symbols. The comparison signature also contains a countably infinite reserve of fresh rigid constants, as required by Nguyen's metatheory~\cite{nguyen2006multimodal}. On the \MProlog{} side, we use Nguyen's SLD calculus for the $\KSys$ schema together with the restrictions on program clauses and goals in Definition~3.4 of his paper~\cite[Definition~3.4]{nguyen2006multimodal}.

Let an \MProlog{} goal be
$\mathfrak G=\leftarrow\alpha_1,\ldots,\alpha_n.$
When $n>0$, translate it to the \MMLP{} query
$\Phi_{\mathfrak G}=\alpha_1\land\cdots\land\alpha_n$.
For the empty goal, put $\Phi_{\mathfrak G}=\neg\bot$. Its NNF translation is $\top$. We use the source formula $\neg\bot$ because $\top$ belongs only to the core NNF language and is not a source formula.

Each \MProlog{} program clause is interpreted as the universally closed formula in the object language denoted by Nguyen's clause notation~\cite{nguyen2006multimodal}. The program $P$ is therefore the same set of closed formulas on both sides. Every free variable of $\Phi_{\mathfrak G}$ is designated as an answer variable, so the two systems return substitutions on the same original goal variables.

The comparison uses results stated and proved by Nguyen~\cite{nguyen2006multimodal}. Section~7.1 characterizes connected $\KSys$ frames by nonempty sets
$W_1\subseteq\cdots\subseteq W_m$
such that
$W=\{\tau\}\cup W_m$ and $R_i=W\times W_i$~\cite[Section~7.1]{nguyen2006multimodal}. Definition~3.9 defines correct answers by local semantic consequence at the actual world~\cite[Definition~3.9]{nguyen2006multimodal}. Theorem~7.1 proves that the concrete $\KSys$ schema satisfies the conditions required by the general framework~\cite[Theorem~7.1]{nguyen2006multimodal}. Consequently, Theorem~5.16 gives soundness of computed answers and Theorem~5.23 gives completeness in the form required below~\cite[Theorems~5.16 and~5.23]{nguyen2006multimodal}.

\subsection{Comparison module}\label{sec:comparison-module}

Write $\Kmod$ for the \MMLP{} modal module used in the comparison. It is defined by
\[
 \Kmod
 =\{\mathrm{D}_i:i\in\Idx\}
 \cup\{\mathrm I(i,j):i>j\}
 \cup\{4(i,j,i):i,j\in\Idx\}
 \cup\{5(i,i,i):i\in\Idx\}.
\]
The name $\KSys$ follows Nguyen's notation. In the present comparison, $\mathrm K$ denotes the normal modal base already built into $H_{\Kmod}$. The family $\mathrm D$ corresponds to seriality. The family $\mathrm I$ corresponds to relation inclusion. The $4_s$ component corresponds to the mixed instances $4(i,j,i)$. The $5$ component corresponds to the instances $5(i,i,i)$. The relational content of $\Kmod$ is seriality, nested relation inclusion, the mixed transitivity conditions $R_j\circ R_i\subseteq R_i$, and directed Euclideanness of each $R_i$.

\begin{lemma}[Relational correspondences for the comparison module]\label{lem:comparison-correspondences}
In this lemma, validity of the displayed modal schemes is understood under arbitrary propositional valuations on the frame. For every multimodal frame:
\begin{enumerate}
\item $\mathrm{D}_i$ is valid precisely on serial $R_i$.
\item $\mathrm I(i,j)$ is valid precisely when $R_j\subseteq R_i$.
\item $4(i,j,i)$ is valid precisely when $R_j\circ R_i\subseteq R_i$.
\item $5(i,i,i)$ is valid precisely when
$uR_iv\land uR_iw\Longrightarrow vR_iw.$
\end{enumerate}
\end{lemma}

\begin{proof}
The implications from the displayed frame conditions to validity follow by the same modal semantic arguments as in Proposition~\ref{prop:modal-correspondence}. For the converses, use a propositional variable $p$ with the valuations specified below.

If seriality fails at $u$, then $\Box_ip$ is vacuously true and $\Diamond_ip$ is false at $u$, so $\mathrm{D}_i$ fails. If $uR_jv$ and $u\not R_iv$, make $p$ true at every $R_i$ successor of $u$ and false at $v$. Then $\mathrm I(i,j)$ fails at $u$.

If $uR_jv$, $vR_iw$, and $u\not R_iw$, make $p$ true at every $R_i$ successor of $u$ and false at $w$. Then $\Box_ip$ is true at $u$, whereas $\Box_j\Box_ip$ is false. Hence $4(i,j,i)$ fails. Finally, if $uR_iv$, $uR_iw$, and $v\not R_iw$, make $p$ true only at $w$. Then $\Diamond_ip$ is true at $u$ and false at $v$, so $\Box_i\Diamond_ip$ is false at $u$ and $5(i,i,i)$ fails.
\end{proof}

Lemma~\ref{lem:comparison-correspondences} establishes the frame correspondences for the individual modal principles. The next theorem shows that, once the frame is generated from the distinguished point, their combined relational content is exactly Nguyen's connected normal form.

\begin{theorem}[Connected frame characterization]\label{thm:connected-frame-characterization}
Let $(W,\tau,(R_i)_{1\le i\le m})$ be a pointed multimodal frame in which every world is reachable from $\tau$ through a path in the union of the relations. The following conditions are equivalent.
\begin{enumerate}
\item The relation family satisfies the frame conditions generated by $\Kmod$.
\item There are nonempty sets
$W_1\subseteq\cdots\subseteq W_m$
such that
$W=\{\tau\}\cup W_m$ and $R_i=W\times W_i$
for every $i$.
\end{enumerate}
Thus the connected \MMLP{} frame class for $\Kmod$ coincides with Nguyen's connected $\KSys$ frame presentation.
\end{theorem}

\begin{proof}
Assume the first condition and set $W_i=R_i(\tau)=\{v:\tau R_i v\}.$ Seriality makes every $W_i$ nonempty. If $i>j$, relation inclusion gives $W_j\subseteq W_i$, so the sets form the required chain.

We next prove $W=\{\tau\}\cup W_m$. Let $u\ne\tau$. By connectedness there is a path from $\tau$ to $u$ whose edges belong to relations $R_{k_1},\ldots,R_{k_n}$. Every $R_k$ is contained in $R_m$ when $k<m$, while $R_m$ is contained in itself. Hence all edges of the path are $R_m$ edges. The selected instance $4(m,m,m)$ gives $R_m\circ R_m\subseteq R_m$, so repeated composition along the finite path yields $\tau R_mu$. Therefore $u\in W_m$.

It remains to prove $R_i(u)=W_i$ for every $u$ and $i$. The equality holds for $u=\tau$ by definition. Let $u\ne\tau$. We already have $\tau R_mu$. If $uR_iv$, the condition $R_m\circ R_i\subseteq R_i$ gives $\tau R_iv$, so $v\in W_i$. Hence $R_i(u)\subseteq W_i$.

For the reverse inclusion, seriality gives some $v$ with $uR_iv$. The preceding inclusion makes $v\in W_i$, so $\tau R_iv$. Let $w\in W_i$ be arbitrary, so $\tau R_iw$. Euclideanness gives $vR_iw$. The selected instance $4(i,i,i)$ then combines $uR_iv$ and $vR_iw$ into $uR_iw$. Thus $W_i\subseteq R_i(u)$, and $R_i=W\times W_i$.

Conversely, assume the second condition. Nonemptiness of $W_i$ makes $R_i=W\times W_i$ serial. If $i>j$, the set inclusion $W_j\subseteq W_i$ gives $R_j\subseteq R_i$. If $uR_jv$ and $vR_iw$, then $w\in W_i$, hence $uR_iw$. Thus $R_j\circ R_i\subseteq R_i$. Finally, if $uR_iv$ and $uR_iw$, then $w\in W_i$ and therefore $vR_iw$. These checks complete the converse implication.
\end{proof}

\begin{corollary}[First-order model correspondence]\label{cor:first-order-model-correspondence}
Under the domain discipline with rigid terms of the common fragment, pointed connected Nguyen models and pointed connected \MMLP{} models over $\Kmod$ share relational frames and the constant-domain requirement. They also agree on rigidity of constants and function symbols, world-dependent predicate interpretations, and assignments that are independent of the world.
\end{corollary}

\begin{proof}
The frame statement is Theorem~\ref{thm:connected-frame-characterization}. The clauses for first-order semantics add no further frame condition, and the two presentations use the same domain and rigidity assumptions in the common fragment.
\end{proof}

\subsection{Generated submodels}\label{sec:generated-submodels}

Nguyen evaluates local consequence in a pointed connected model. \MMLP{} validity quantifies globally over all worlds of all frames in the selected class. The following lemma on generated submodels relates these two notions of validity.

\begin{lemma}[Generated connected submodel]\label{lem:generated-submodel}
Let
$\mathcal M=(W,(R_i),D,\mathcal J)$
be a constant-domain model with rigid terms based on a frame in $\Frames(\Kmod)$ and let $w\in W$. Let $W_w$ contain $w$ and every world reachable from $w$ through a path in the union of the $R_i$. Restrict all relations and world-dependent predicate interpretations to $W_w$, while keeping the domain and rigid term interpretations unchanged. Call the resulting model $\mathcal M\upharpoonright w$. Then:
\begin{enumerate}
\item the restricted frame is connected from $w$.
\item it still belongs to $\Frames(\Kmod)$.
\item for every source formula $A$, every $u\in W_w$, and every assignment $g$,
$\mathcal M,u,g\models A\Longleftrightarrow \mathcal M\upharpoonright w,u,g\models A$.
\end{enumerate}
\end{lemma}

\begin{proof}
Connectedness is immediate from the definition of $W_w$. For seriality, let $u\in W_w$ and choose an $R_i$ successor $v$ in the original model. Such a successor exists by seriality, and $v$ also belongs to $W_w$ because its edge can be appended to a path from $w$ to $u$. The inclusion, mixed transitivity, and Euclidean conditions are universal implications between relation atoms, so restriction preserves them. Thus the restricted frame remains in the selected class.

For truth preservation, use structural induction on $A$.

Base case. Rigid terms keep the same denotations, and predicate interpretations agree at retained worlds, so atomic formulas and truth constants have the same truth values.

Induction step. Assume truth preservation for the immediate subformulas.

\begin{caselist}
\item Case Boolean constructor. Apply the induction hypotheses and the corresponding Boolean semantic clause.
\item Case quantifier. The domain is unchanged, so the quantifier ranges over the same elements in both models. Apply the induction hypothesis under each updated assignment.
\item Case box or diamond. At a retained world $u$, every original successor of $u$ is also retained by successor closure, so the successor set is the same before and after restriction. Apply the induction hypothesis at each successor.
\end{caselist}
\end{proof}

\subsection{Local consequence}\label{sec:local-consequence-equivalence}

Write $P\models^{\mathrm N}F$ for Nguyen's local semantic consequence at the actual world.

\begin{theorem}[Local consequence equivalence]\label{thm:local-consequence-equivalence}
For every finite set $P$ of closed program formulas in the common fragment and every closed formula $F$ in the common fragment,
$\Frames(\Kmod)\models(\widehat P\to F)\Longleftrightarrow P\models^{\mathrm N}F$.
\end{theorem}

\begin{proof}
Assume first that the global \MMLP{} implication is valid. Take an arbitrary Nguyen model with actual world $\tau$ in which every formula of $P$ is true at $\tau$. By Corollary~\ref{cor:first-order-model-correspondence}, forgetting only the distinguished marker for the actual world yields a model over $\Frames(\Kmod)$. Hence $\widehat P\to F$ is true at $\tau$, and $F$ follows there. The Nguyen model was arbitrary.

Conversely, assume $P\models^{\mathrm N}F$ and suppose the global implication is invalid. Then some \MMLP{} model and world $w$ satisfy $\widehat P$ and falsify $F$. Apply Lemma~\ref{lem:generated-submodel}. The generated restriction, pointed at $w$, is connected and remains in the selected frame class, so Corollary~\ref{cor:first-order-model-correspondence} makes it a Nguyen model. By truth preservation, every program formula remains true at $w$, and $F$ remains false there, contradicting the assumed local consequence.
\end{proof}

The preceding semantic equivalence extends directly to answer correctness by universally closing each instantiated goal.

Let $\Corr^{\mathrm N}(P,\mathfrak G)$ denote Nguyen's correct answer set for the pair on the common fragment. The superscript $\mathrm N$ marks Nguyen's answer notation in the common fragment.

\begin{theorem}[Correct answer equivalence]\label{thm:correct-answer-equivalence}
For every $(P,\mathfrak G)$ in the common fragment,
$\Corr^{\mathrm H}_{\Kmod}(P,\Phi_{\mathfrak G}) =\Corr^{\mathrm N}(P,\mathfrak G).$
\end{theorem}

\begin{proof}
Take a legal user answer substitution $\theta$ and let $F_\theta$ be the universal closure of $\Phi_{\mathfrak G}\theta$. By Definition~\ref{def:hilbert-correct}, $\theta$ is \MMLP{} correct precisely when $H_{\Kmod}\vdash\widehat P\to F_\theta.$ Corollary~\ref{cor:three-way-adequacy} converts this condition to global validity over $\Frames(\Kmod)$. Theorem~\ref{thm:local-consequence-equivalence} then converts it to $P\models^{\mathrm N}F_\theta.$ This is Nguyen's correct answer condition for $\theta$ from Definition~3.9~\cite[Definition~3.9]{nguyen2006multimodal}. The equivalence holds for every legal answer substitution.
\end{proof}

\subsection{Computed answers in \MProlog}\label{sec:mprolog-computed-answers}

Let $\Comp^{\mathrm N}(P,\mathfrak G)$ be the computed answer set of Nguyen's concrete calculus. Given a computed answer $\eta$, let $\gamma$ be a capture-avoiding ordinary substitution whose domain is contained in the variables still free in $\Phi_{\mathfrak G}\eta$ and whose range consists of ordinary user terms. Every free variable of $\Phi_{\mathfrak G}$ is an answer variable, so such a $\gamma$ satisfies the output condition of Definition~\ref{def:ordinary-output-substitution}.

The \emph{instance closure} $\Inst(\Comp^{\mathrm N}(P,\mathfrak G))$ consists of all restrictions $(\eta\gamma)\upharpoonright\FV(\Phi_{\mathfrak G})$ obtained in this way from computed answers $\eta$.

\begin{lemma}[Instance closure of Nguyen answers]\label{lem:nguyen-instance-closure}
If $\eta\in\Corr^{\mathrm N}(P,\mathfrak G)$ and $\gamma$ satisfies the preceding conditions, then
$(\eta\gamma)\upharpoonright\FV(\Phi_{\mathfrak G})\in\Corr^{\mathrm N}(P,\mathfrak G).$
\end{lemma}

\begin{proof}
By Theorem~\ref{thm:correct-answer-equivalence}, $\eta$ is an \MMLP{} declaratively correct answer. Since every free variable of $\Phi_{\mathfrak G}$ is designated as an answer variable, Lemma~\ref{lem:correct-instance-closure} preserves correctness under substitutions satisfying the same domain and range conditions after restriction to the original goal variables. A second application of Theorem~\ref{thm:correct-answer-equivalence} gives the claim.
\end{proof}

\begin{lemma}[Goal substitution reflection]\label{lem:goal-substitution-reflection}
Let $\sigma$ and $\sigma'$ be ordinary substitutions on the original variables of $\mathfrak G$. If
$\mathfrak G\sigma=\mathfrak G\sigma'$
as syntactic goal lists, then $\sigma(X)=\sigma'(X)$ for every $X\in\FV(\Phi_{\mathfrak G})$. Consequently, if
$\mathfrak G\theta=\mathfrak G\eta\delta,$
then
$\theta=\eta\delta_0$
on the original answer variables, where
$\delta_0=\delta\upharpoonright\FV(\Phi_{\mathfrak G}\eta).$
The substitution $\delta_0$ satisfies the conditions used in the instance closure.
\end{lemma}

\begin{proof}
Fix $X\in\FV(\Phi_{\mathfrak G})$ and choose one free occurrence of $X$ in the goal. Follow the fixed path in the syntax tree from the root of that goal atom to the occurrence. Simultaneous substitution preserves all constructors on this path and replaces its leaf by $\sigma(X)$ or $\sigma'(X)$. Equality of the substituted goal lists therefore gives $\sigma(X)=\sigma'(X)$.

For the consequence, apply the first assertion to $\theta$ and the ordinary composition $\eta\delta$. Every variable on which $\delta$ can affect $\eta(X)$ occurs in $\Phi_{\mathfrak G}\eta$. Restricting $\delta$ to those remaining free variables therefore leaves its action on every $\eta(X)$ unchanged. The domain and range conditions for legality follow from the definition of the restriction.
\end{proof}

The two lemmas establish the closure and reflection steps for substitutions needed to translate Nguyen's soundness and completeness theorems stated at the level of goals into an equality between the two instance closures.

\begin{theorem}[\MProlog{} answer characterization]\label{thm:mprolog-answer-characterization}
For every $(P,\mathfrak G)$ in the common fragment,
$\Inst(\Comp^{\mathrm N}(P,\mathfrak G)) =\Corr^{\mathrm N}(P,\mathfrak G).$
\end{theorem}

\begin{proof}
For the forward inclusion, let $\theta=(\eta\gamma)\upharpoonright\FV(\Phi_{\mathfrak G})$ be an instance of an \MProlog{} computed answer. Nguyen's Theorem~5.16, applied to the concrete schema verified by Theorem~7.1, shows that $\eta$ is correct~\cite[Theorems~5.16 and~7.1]{nguyen2006multimodal}. Lemma~\ref{lem:nguyen-instance-closure} then shows that $\theta$ is correct.

For the reverse inclusion, let $\theta\in\Corr^{\mathrm N}(P,\mathfrak G)$. By Nguyen's Theorem~5.23, applied to the concrete schema verified by Theorem~7.1, there exist a computed answer $\eta$ and an ordinary substitution $\delta$ with $\mathfrak G\theta=\mathfrak G\eta\delta$~\cite[Theorems~5.23 and~7.1]{nguyen2006multimodal}. Lemma~\ref{lem:goal-substitution-reflection} yields a restriction $\delta_0$ satisfying those conditions, with $\theta=\eta\delta_0$ on the original goal variables. Thus $\theta$ belongs to the instance closure.
\end{proof}

\subsection{Answer conservativity}\label{sec:answer-conservativity}

\begin{theorem}[$\KSys$ answer conservativity]\label{thm:ksys-answer-conservativity}
For every $(P,\mathfrak G)$ in the common fragment,
\[
\begin{aligned}
&\Inst(\Comp_{\Kmod}^{\mathrm F}(P,\Phi_{\mathfrak G}))\\
={}&\Inst(\Comp_{\Kmod}^{\mathrm M}(P,\Phi_{\mathfrak G}))\\
={}&\Corr^{\mathrm H}_{\Kmod}(P,\Phi_{\mathfrak G})\\
={}&\Corr^{\mathrm N}(P,\mathfrak G)\\
={}&\Inst(\Comp^{\mathrm N}(P,\mathfrak G)).
\end{aligned}
\]
Both \MMLP{} answer procedures use the same certificate kernel instantiated with $\Kmod$. The focused procedure adds only the scheduling discipline of Section~\ref{sec:focusing}.
\end{theorem}

\begin{proof}
Theorem~\ref{thm:focused-answer-characterization} gives equality of the first three expressions. Theorem~\ref{thm:correct-answer-equivalence} gives the equality with Nguyen's correct answers. Theorem~\ref{thm:mprolog-answer-characterization} gives the final equality.
\end{proof}

The following definition states the extension claim.

\begin{definition}[Language extension with answer conservativity]\label{def:language-extension-conservativity}
Suppose a logic programming system $S_2$ translates every program and query pair of a system $S_1$ on a common fragment. We say that $S_2$ is a \emph{language extension of $S_1$ with answer conservativity} when two conditions hold.
\begin{enumerate}
\item The combined source language for programs and queries in $S_2$ strictly contains the translated common fragment.
\item For every program and query pair in the common fragment, the two systems have the same instance closure of computed answers on the user answer variables.
\end{enumerate}
\end{definition}

\begin{corollary}[Language extension of $\KSys$-\MProlog]\label{cor:language-extension}
\MMLP{} is a language extension of Nguyen's pure logical $\KSys$-\MProlog{} calculus with answer conservativity on the common fragment fixed in Section~\ref{sec:mprolog-common-fragment}.
\end{corollary}

\begin{proof}
The instance closure equality is Theorem~\ref{thm:ksys-answer-conservativity}. Strict language containment follows from Definition~\ref{def:source-formulas}. A nonempty common \MProlog{} goal is a list translated to a conjunction of admissible goal atoms, with the empty goal translated as $\neg\bot$, while \MMLP{} permits arbitrary source formulas as queries. For example, an outermost disjunction $p\lor q$ is a legal \MMLP{} query and lies outside that common \MProlog{} goal syntax.
\end{proof}

\section{Conclusion}\label{sec:conclusion}

\MMLP{} gives a first-order multimodal logic programming system in which both programs and queries may be arbitrary formulas of the source language. The Hilbert system $H_{\Amod}$ specifies the declarative semantics independently. The calculus $\mathcal N_{\Amod}$ is a certificate proof kernel, and scoped unification is used for symbolic answer computation. The main metatheory proves cut-free completeness and scoped lifting, and it characterizes declaratively correct answers as the instance closure of computed answers.

The focused calculus $\mathcal F_{\Amod}$ preserves ordinary derivability by a syntactic focalization theorem and yields the same instance closure of computed answers. On the common fragment with Nguyen's pure logical $\KSys$-\MProlog{}, that closure also agrees with the instance closure of the computed answers produced by \MProlog{}. Thus the larger program and query language of \MMLP{} preserves the established answer semantics on the comparison fragment while extending the formulas available to users.

\section*{Code availability}\label{sec:code-availability}

The source code for the implementation of \MMLP{} described in this paper is publicly available at \url{https://github.com/kenjitokuo/mmlp}. The version used for this paper is release \texttt{v0.1.0}, archived on Zenodo at \url{https://doi.org/10.5281/zenodo.21917365}.

\appendix

\section{Normalization results}\label{app:normalization}

The core NNF language and the recursive translations are defined in Definitions~\ref{def:core-nnf} and~\ref{def:nnf-translations}. This appendix contains the normalization and substitution proofs used by the main calculus. Their explicit forms matter because answer projection and proof abstraction use finite simultaneous substitutions rather than informal replacement of one variable.

\subsection{Semantic adequacy}\label{app:nnf-adequacy}

We begin by separating two tasks that were used together in the main text: closure of the translations inside the core grammar and semantic equivalence with the source formulas. The first is purely syntactic, while the second requires the simultaneous semantic induction.

\begin{proposition}[Core well-formedness]\label{prop:nnf-wellformed}
For every source formula $A$, both $\NNF^+(A)$ and $\NNF^-(A)$ are core formulas.
\end{proposition}

\begin{proof}
Use simultaneous structural induction on $A$.

Base case. Atomic formulas and $\bot$ follow from the first two clauses of Definition~\ref{def:nnf-translations}.

Induction step. Assume both translations of every immediate subformula are core formulas.

\begin{caselist}
\item Case negation. The two translations exchange the induction hypotheses.
\item Case binary connective. Each translation combines the corresponding induction hypotheses with $\land$ or $\lor$.
\item Case quantifier or modality. The induction hypothesis makes each translation of the immediate subformula a core formula. Applying the corresponding quantifier or modal constructor therefore yields a core formula.
\end{caselist}
\end{proof}

Well-formedness ensures that both translations stay inside the core language. Theorem~\ref{thm:normalization-adequacy} now follows by the same simultaneous induction, which proves the truth and falsity clauses in parallel.

\begin{proof}[Proof of Theorem~\ref{thm:normalization-adequacy}]
We prove the positive and negative semantic equivalences simultaneously by structural induction on $A$.

Base case. If $A=p(\bar t)$, the positive translation is the same atom and the negative translation is true precisely when that atom is false by the semantics of primitive negative literals. If $A=\bot$, the positive clause is immediate and the negative translation $\top$ is true precisely when $\bot$ is false.

Induction step. Assume the positive and negative equivalences for the immediate subformulas.

\begin{caselist}
\item Case $A=\neg B$. The positive translation is $\NNF^-(B)$, which is true precisely when $B$ is false by the negative induction hypothesis. This is the truth condition for $\neg B$. The negative translation is $\NNF^+(B)$, which is true precisely when $B$ is true, equivalently when $\neg B$ is false.
\item Case conjunction or disjunction. For conjunction, the positive translation is true precisely when both positive translations of the conjuncts are true, hence precisely when both conjuncts are true. The negative translation is a disjunction of the negative translations, so it is true precisely when at least one conjunct is false. The disjunction case is dual.
\item Case implication. The positive translation $\NNF^-(B)\lor\NNF^+(C)$ is true precisely when $B$ is false or $C$ is true, which is the classical truth condition for $B\to C$. The negative translation $\NNF^+(B)\land\NNF^-(C)$ is true precisely when $B$ is true and $C$ is false.
\item Case universal or existential quantifier. For $A=\forall xB$, the positive translation is true at $(\mathcal M,w,g)$ precisely when $\NNF^+(B)$ is true under every assignment $g[x\mapsto d]$, $d\in D$. By the induction hypothesis, this is equivalent to truth of $B$ under every such assignment. The negative translation is existential and is true precisely when $B$ is false for some $d$. The existential source case is dual.
\item Case box or diamond. For $A=\Box_iB$, the positive translation is true precisely when the positive translation of $B$ is true at every $R_i$ successor, and the induction hypothesis gives the ordinary box condition. The negative translation is $\Diamond_i\NNF^-(B)$ and is true precisely when some $R_i$ successor falsifies $B$, which is the condition for the box to be false. The diamond source case is dual.
\end{caselist}

For the assertion about free variables, use a second simultaneous structural induction.

Base case. The atomic translations leave the terms unchanged and therefore preserve their free variables. The formula $\bot$ has no free variables.

Induction step.

\begin{caselist}
\item Case Boolean or modal constructor. The induction hypotheses preserve the free variables of the immediate subformulas. Boolean constructors take the union of these sets, while modal constructors leave the set unchanged.
\item Case quantifier. Both translations bind the same variable as the source formula, so the induction hypothesis for the matrix gives the result.
\end{caselist}

For finite simultaneous substitution, prove $\NNF^\pm(A\theta)\equiv_\alpha \NNF^\pm(A)\theta$ by simultaneous structural induction on $A$.

Base case. The atomic case follows from simultaneous term substitution, and $\bot$ is immediate.

Induction step.

\begin{caselist}
\item Case Boolean or modal constructor. Apply the induction hypotheses componentwise.
\item Case quantifier. Alpha convert the bound variable so that it lies outside the finite domain of $\theta$ and does not occur freely in any term in the range of $\theta$. Capture-avoiding substitution then commutes with the quantifier, and the induction hypothesis applies to the matrix.
\end{caselist}
\end{proof}

\subsection{Hilbert normalization}\label{app:hilbert-nnf}

We now prove the Hilbert normalization theorem stated in Section~\ref{sec:hilbert-specification}. The clauses for negation, quantifiers, and modalities involve both translations, so the proof below uses simultaneous induction.

\begin{proof}[Proof of Theorem~\ref{thm:hilbert-nnf}]
We establish two statements by structural induction on $A$. The positive statement is $H_{\Amod}\vdash A\leftrightarrow\NNF^+(A)$, and the negative statement is $H_{\Amod}\vdash\neg A\leftrightarrow\NNF^-(A)$.

Base case. Atomic formulas use the definitional Hilbert interpretation of $\overline p(\bar t)$ as $\neg p(\bar t)$. Falsity uses $\neg\bot\leftrightarrow\top$.

Induction step. Assume both equivalences for the immediate subformulas.

\begin{caselist}
\item Case negation. Use the two induction hypotheses and classical double negation.
\item Case conjunction or disjunction. Use propositional replacement together with the two De Morgan laws.
\item Case implication. Use the classical equivalences $A\to B\leftrightarrow(\neg A\lor B)$ and $\neg(A\to B)\leftrightarrow(A\land\neg B)$.
\item Case universal or existential quantifier. Apply the induction equivalences under the quantifier and use the classical quantifier dualities $\neg\forall xA\leftrightarrow\exists x\neg A$ and $\neg\exists xA\leftrightarrow\forall x\neg A$. These equivalences are derivable from the first-order Hilbert base fixed in Definition~\ref{def:hilbert-base}.
\item Case box or diamond. For a box, lift the positive induction equivalence through $\Box_i$ by necessitation and $\mathrm{K}_i$. For the negative side use $\neg\Box_iA\leftrightarrow\Diamond_i\neg A$ and lift the negative induction equivalence through $\Diamond_i$, using the definitional duality and normality of $\Box_i$. The diamond case is symmetric.
\end{caselist}
\end{proof}

\subsection{Substitution lemmas}\label{app:substitution}

Fix a finite set $Q\subseteq\mathsf{Par}$ and use the temporary rigid signature $\Sigma\cup Q$ from Convention~\ref{conv:parameter-expansion}. The following substitution lemmas therefore apply both to source terms and formulas and to the parameterized proof terms and proof core formulas used inside certificate derivations.

\begin{lemma}[Term substitution]\label{lem:term-substitution}
Let $r$ and $t$ be terms over $\Sigma\cup Q$, with $t$ free for $x$ in $r$, and fix a model with rigid interpretations for $\Sigma\cup Q$ and an assignment $g$. Put $g'=g[x\mapsto\sem{t}_g]$. Then
$\sem{r(t/x)}_{g}=\sem{r}_{g'}$.
The same statement holds for a finite simultaneous capture-avoiding term substitution.
\end{lemma}

\begin{proof}
For the statement concerning one variable, use structural induction on $r$.

Base case. Variables are the defining cases of assignment update. Rigid constants, including the temporary proof parameters in $Q$, are unchanged.

Induction step. Suppose $r=f(r_1,\ldots,r_n)$. Apply the induction hypotheses to the arguments and then the fixed interpretation of $f$.

For the finite simultaneous version, use structural induction on $r$ with the simultaneous assignment that maps every substituted variable to the denotation of its replacement term under the original assignment.

Base case. Variables satisfy the claim by the definition of the simultaneous assignment, and rigid constants satisfy it because their interpretations are unchanged.

Induction step. Suppose $r=f(r_1,\ldots,r_n)$. Apply the induction hypotheses to the arguments and then the interpretation of $f$.
\end{proof}

The term lemma proves the atomic case for substitution in formulas. The remaining constructors follow from the semantic clauses.

\begin{lemma}[Formula substitution]\label{lem:formula-substitution}
Let $A$ be a source formula or proof core formula over the temporary signature $\Sigma\cup Q$, and let the term $t$ over that signature be free for $x$ in $A$. For every model $\mathcal M$ over that temporary signature with rigid interpretations of the members of $Q$, every world $w$, and every assignment $g$,
$\mathcal M,w,g\models A(t/x)\Longleftrightarrow \mathcal M,w,g[x\mapsto\sem{t}_g]\models A$.
The analogous equivalence holds for every finite simultaneous capture-avoiding substitution.
\end{lemma}

\begin{proof}
Use structural induction on $A$ for the statement concerning a single variable.

Base case. Positive and negative atoms follow from Lemma~\ref{lem:term-substitution}. Truth constants are immediate.

Induction step. Assume the equivalence for the immediate subformulas.

\begin{caselist}
\item Case Boolean constructor. Apply the induction hypotheses and the corresponding Boolean semantic clause.
\item Case quantifier. Alpha convert the bound variable away from $x$ and the free variables of $t$. The quantifier then ranges over the same constant domain on both sides, and the induction hypothesis applies to every updated assignment.
\item Case modality. Modal operators do not alter assignments, so the induction hypothesis applies at every accessible world.
\end{caselist}

For the finite simultaneous statement, use the same structural induction with the corresponding simultaneous assignment update.

Base case. Atomic formulas follow from the simultaneous term substitution statement, and truth constants are immediate.

Induction step.

\begin{caselist}
\item Case Boolean or modal constructor. Apply the induction hypotheses componentwise.
\item Case quantifier. Alpha convert the bound variable so that it lies outside the finite substitution domain and does not occur freely in any term in the substitution range, then apply the induction hypothesis to every updated assignment.
\end{caselist}
\end{proof}

\subsection{Root normalization}\label{app:root-normalization}

The final normalization fact relates the source implication used for declarative correctness to the one-sided root used by certificate and symbolic execution. It also fixes the convention for an empty program at the level of the proof calculus.

\begin{lemma}[Root normalization]\label{lem:root-normalization}
For a nonempty program $P=\{P_1,\ldots,P_n\}$,
$\NNF^+(\widehat P\to G)$
is, up to association of disjunction,
$\NNF^-(P_1)\lor\cdots\lor\NNF^-(P_n)\lor\NNF^+(G).$
Its one-sided root presentation is therefore
$\NNF^-(P_1),\ldots,\NNF^-(P_n),\NNF^+(G).$
For the empty program convention $\widehat P=\neg\bot$, the root is $\bot,\NNF^+(G)$.
\end{lemma}

\begin{proof}
The positive NNF clause for implication gives $\NNF^+(\widehat P\to G)=\NNF^-(\widehat P)\lor\NNF^+(G).$ Repeated application of the negative conjunction clause gives the disjunction of the $\NNF^-(P_i)$. In the one-sided sequent semantics, formulas in a component are combined disjunctively. For the empty program, $\NNF^+((\neg\bot)\to G) =\NNF^-(\neg\bot)\lor\NNF^+(G) =\bot\lor\NNF^+(G).$ \end{proof}

\section{Hilbert certification of the certificate calculus}\label{app:hilbert-certification}

The completeness proof in the main text first establishes the relation between certificate derivability and Kripke semantics. The converse direction, from certificate derivability to Hilbert theoremhood, requires a separate syntactic argument. This appendix contains the propositional canonical theorem and the Hilbert lemmas for deep contexts used by that argument.

\subsection{Canonical completeness}\label{app:propositional-canonical}

Fix a countably infinite effectively enumerable reserve of propositional variables, disjoint from the first-order syntax, for the propositional abstractions in this appendix. Let $H_{\Amod}^{\mathrm{prop}}$ be the propositional fragment of $H_{\Amod}$ over this reserve. A propositional valuation assigns to each propositional variable a set of worlds. Propositional satisfaction uses the classical Boolean clauses and the modal clauses of Definition~\ref{def:modal-model}. A propositional formula is valid on a frame when it is true at every world under every propositional valuation. The notions of class validity and the semantics based on position maps extend to propositional formulas and to nested structures whose formula occurrences are propositional formulas. In particular, the proof of Lemma~\ref{lem:position-map-semantics} applies without change to such propositional nested structures.

\begin{theorem}[Propositional canonical completeness]\label{thm:prop-canonical-completeness}
If a propositional multimodal formula $F$ is valid on every frame in $\Frames(\Amod)$, then
$H_{\Amod}^{\mathrm{prop}}\vdash F.$
\end{theorem}

\begin{proof}
Call a set $\Gamma$ \emph{consistent} when no finite subset $\Gamma_0\subseteq\Gamma$ satisfies $H_{\Amod}^{\mathrm{prop}}\vdash(\bigwedge\Gamma_0)\to\bot.$ If $F$ is not derivable, $\{\neg F\}$ is consistent by classical propositional reasoning.

Enumerate the propositional multimodal formulas as $E_0,E_1,\ldots$. Starting from any consistent $\Gamma_0$, extend recursively by adding $E_n$ when consistency is preserved and by adding $\neg E_n$ otherwise. At least one extension is consistent. If both were inconsistent, finite inconsistency witnesses for the two alternatives and propositional reasoning by cases would make $\Gamma_n$ inconsistent. The union is consistent because every finite subset occurs at some finite stage. It is maximal because it contains exactly one of $E_n$ and $\neg E_n$ for every $n$.

Let $W$ be the set of maximal consistent sets. We first verify the propositional closure properties that will be used below. If $H_{\Amod}^{\mathrm{prop}}\vdash A$ and $A\notin w$, then the construction of $w$ gives $\neg A\in w$. The finite set $\{\neg A\}$ is then inconsistent because $A$ is a theorem, contradicting consistency of $w$. Hence every theorem belongs to every $w\in W$. If $A\in w$ and $A\to B\in w$ but $B\notin w$, then $\neg B\in w$, and the finite set $\{A,A\to B,\neg B\}$ is inconsistent by propositional reasoning. Thus each $w$ is closed under modus ponens. The construction places at least one of $A$ and $\neg A$ in $w$, while consistency forbids both, so exactly one belongs to $w$.

For conjunction, if $A\land B\in w$ and $A\notin w$, then $\neg A\in w$, making $\{A\land B,\neg A\}$ inconsistent; the same argument gives $B\in w$. Conversely, if $A,B\in w$ and $A\land B\notin w$, then $\neg(A\land B)\in w$, and $\{A,B,\neg(A\land B)\}$ is inconsistent. Hence $A\land B\in w$ exactly when both $A$ and $B$ belong to $w$. For disjunction, if $A\lor B\in w$ while neither disjunct belongs to $w$, then $\neg A,\neg B\in w$, and $\{A\lor B,\neg A,\neg B\}$ is inconsistent. Conversely, if $A\in w$ and $A\lor B\notin w$, then $\neg(A\lor B)\in w$, and $\{A,\neg(A\lor B)\}$ is inconsistent; the case $B\in w$ is symmetric. Thus $A\lor B\in w$ exactly when at least one disjunct belongs to $w$.

For implication, suppose $A\to B\in w$. If $A\in w$, closure under modus ponens gives $B\in w$. Hence $A\notin w$ or $B\in w$. Conversely, suppose $A\notin w$ or $B\in w$. If $A\notin w$, then $\neg A\in w$. Classical propositional reasoning gives the theorem $\neg A\to(A\to B)$, so modus ponens yields $A\to B\in w$. If $B\in w$, the instance $B\to(A\to B)$ of $\mathrm{H1}$ and modus ponens again yield $A\to B\in w$. Thus $A\to B\in w$ exactly when $A\notin w$ or $B\in w$. No maximal consistent set contains $\bot$, since $\{\bot\}$ is inconsistent.

Define $\Delta_i(w)=\{A:\Box_iA\in w\}$. Put $wR_iv$ exactly when $\Delta_i(w)\subseteq v$, and let a propositional variable $p$ be true at $w$ precisely when $p\in w$. Let $\mathcal M^c$ denote this frame equipped with the displayed propositional valuation, and write $\mathcal M^c,w\models A$ for the propositional satisfaction relation just specified.

We first prove the modal existence property. Suppose $\Diamond_iA\in w$. If $\Delta_i(w)\cup\{A\}$ were inconsistent, choose a finite inconsistent subset. If it contains no member of $\Delta_i(w)$, propositional reasoning gives $H_{\Amod}^{\mathrm{prop}}\vdash\neg A$. Otherwise, let $B_1,\ldots,B_n$, with $n\geq1$, be its members from $\Delta_i(w)$. Whether or not the subset contains $A$, propositional reasoning gives $H_{\Amod}^{\mathrm{prop}}\vdash (B_1\land\cdots\land B_n)\to\neg A$. In the first case, necessitation gives $\Box_i\neg A$, while modal duality and $\Diamond_iA\in w$ give $\neg\Box_i\neg A\in w$, contradicting consistency.

In the second case, we use only normal modal reasoning to combine the boxed assumptions. From the propositional theorem $C\to(D\to C\land D)$, necessitation and two applications of $\mathrm{K}_i$ give $H_{\Amod}^{\mathrm{prop}}\vdash (\Box_iC\land\Box_iD)\to\Box_i(C\land D)$. Iterating this theorem yields
$H_{\Amod}^{\mathrm{prop}}\vdash (\Box_iB_1\land\cdots\land\Box_iB_n)\to\Box_i(B_1\land\cdots\land B_n).$
Each $B_j$ belongs to $\Delta_i(w)$, so each $\Box_iB_j$ belongs to $w$. The conjunction property and modus ponens therefore give $\Box_i(B_1\land\cdots\land B_n)\in w$. Necessitation and $\mathrm{K}_i$ applied to $(B_1\land\cdots\land B_n)\to\neg A$ give the theorem $\Box_i(B_1\land\cdots\land B_n)\to\Box_i\neg A$, so $\Box_i\neg A\in w$. Modal duality and $\Diamond_iA\in w$ give $\neg\Box_i\neg A\in w$, again contradicting consistency. Thus $\Delta_i(w)\cup\{A\}$ is consistent and extends to some maximal consistent $v$. Then $wR_iv$ and $A\in v$.

We next prove the truth lemma $\mathcal M^c,w\models A\Longleftrightarrow A\in w$ by structural induction on $A$.

Base case. For propositional variables, the assertion is the valuation definition. For $\bot$, it follows from the fact that no maximal consistent set contains $\bot$.

Induction step. Assume the truth lemma for the immediate subformulas.

\begin{caselist}
\item Case Boolean constructor. The result follows from the corresponding property of maximal consistent sets and the induction hypotheses.
\item Case $A=\Box_iB$. If $\Box_iB\in w$, every $R_i$ successor contains $B$, so the induction hypothesis makes $B$ true there. If $\Box_iB\notin w$, then $\neg\Box_iB\in w$, hence $\Diamond_i\neg B\in w$. The modal existence property gives an $R_i$ successor containing $\neg B$, where $B$ is false by the induction hypothesis.
\item Case $A=\Diamond_iB$. If $\Diamond_iB\in w$, the modal existence property gives some $v$ with $wR_iv$ and $B\in v$. By the induction hypothesis, $\mathcal M^c,v\models B$, so $\mathcal M^c,w\models\Diamond_iB$. Conversely, suppose $\mathcal M^c,w\models\Diamond_iB$. Choose $v$ with $wR_iv$ and $\mathcal M^c,v\models B$. The induction hypothesis gives $B\in v$. If $\Diamond_iB\notin w$, then $\neg\Diamond_iB\in w$. Modal duality gives $\Box_i\neg B\in w$, and $wR_iv$ therefore gives $\neg B\in v$, contradicting consistency of $v$. Hence $\Diamond_iB\in w$.
\end{caselist}

It remains to verify every selected frame condition in the canonical relation family.

If $\mathrm{T}_i\in\Amod$ and $\Box_iA\in w$, axiom $\mathrm{T}_i$ gives $A\in w$. Hence $\Delta_i(w)\subseteq w$ and $wR_iw$.

If $\mathrm I(i,j)\in\Amod$ and $wR_jv$, take $\Box_iA\in w$. The selected axiom gives $\Box_jA\in w$, and $wR_jv$ gives $A\in v$. Therefore $wR_iv$, so $R_j\subseteq R_i$.

If $\mathrm B(i,j)\in\Amod$ and $wR_iv$, take $\Box_jA\in v$. If $A\notin w$, maximality gives $\neg A\in w$. The selected $\mathrm B(i,j)$ instance gives $\Box_i\Diamond_j\neg A\in w$, hence $\Diamond_j\neg A\in v$. This contradicts $\Box_jA\in v$ by duality. Thus $A\in w$ and $vR_jw$.

If $4(i,j,k)\in\Amod$, $wR_jv$, and $vR_ku$, take $\Box_iA\in w$. The selected axiom gives $\Box_j\Box_kA\in w$, hence $\Box_kA\in v$, and then $A\in u$. Therefore $wR_iu$.

If $5(i,j,k)\in\Amod$, $wR_iu$, and $wR_jv$, take $\Box_kA\in v$ and suppose $A\notin u$. Then $\neg A\in u$. If $\Diamond_i\neg A\notin w$, duality and maximality give $\Box_iA\in w$, contradicting $wR_iu$. Hence $\Diamond_i\neg A\in w$. The selected axiom gives $\Box_j\Diamond_k\neg A\in w$, so $\Diamond_k\neg A\in v$, contradicting $\Box_kA\in v$. Thus $A\in u$ and $vR_ku$.

Suppose $\mathrm{D}_i\in\Amod$. We prove that $\Delta_i(w)$ is consistent. Otherwise finitely many $A_1,\ldots,A_n\in\Delta_i(w)$ would make $(A_1\land\cdots\land A_n)\to\bot$ a theorem. By the boxed conjunction argument above, $\Box_iA_1,\ldots,\Box_iA_n\in w$ give $\Box_i(A_1\land\cdots\land A_n)\in w$. Necessitation and $\mathrm{K}_i$ applied to $(A_1\land\cdots\land A_n)\to\bot$ then give $\Box_i\bot\in w$. The selected seriality axiom gives $\Diamond_i\bot\in w$. Since $\neg\bot$ is a theorem, necessitation gives the theorem $\Box_i\neg\bot$, which belongs to $w$. Modal duality then gives $\neg\Diamond_i\bot\in w$, a contradiction. Extend $\Delta_i(w)$ to a maximal consistent $v$. Then $wR_iv$, proving seriality.

The canonical frame therefore belongs to $\Frames(\Amod)$. If $F$ were valid on every selected frame while underivable, extend $\{\neg F\}$ to a maximal consistent world. The truth lemma would falsify $F$ there in the canonical model, a contradiction.
\end{proof}

\subsection{Certificate propagation}\label{app:certificate-hilbert-proof}

This subsection contains the detailed Hilbert certification of certificate propagation deferred from the main text. The proof abstracts first-order formulas propositionally while preserving the correspondence between formula occurrences that persist from the conclusion to the premise.

\begin{proof}[Detailed proof of Lemma~\ref{lem:certificate-hilbert}]
Fix a finite $Q\subseteq\mathsf{Par}$ containing all proof parameters in the rule instance, and work throughout in $H_{\Amod}[Q]$. Fix an instance of $(\Diamond_i^{\mathrm{cert}})$ with premise $\mathcal G\{\Diamond_iA\}_{u}\{A,\Delta\}_{v}$ and conclusion $\mathcal G\{\Diamond_iA\}_{u}\{\Delta\}_{v}$. The source and target may coincide.

Introduce one propositional variable $p$ for the distinguished matrix $A$. Replace the retained source occurrence $\Diamond_iA$ by $\Diamond_ip$ and the inserted target occurrence $A$ by the same $p$. For every other formula occurrence that persists from the conclusion to the premise, use one fresh propositional variable for the corresponding occurrences on both sides. Use distinct variables for distinct persistent occurrences. Keep all nested brackets and modal indices unchanged. Let $F^{p}$ and $F^{c}$ be the recursive interpretations of the resulting propositional premise and conclusion.

We prove $\Frames(\Amod)\models F^{p}\to F^{c}.$ Take a model based on a selected frame and suppose $F^{c}$ is false at a world $w$. Semantics of position maps, Lemma~\ref{lem:position-map-semantics}, gives a position map $\iota$ rooted at $w$ that respects every edge and makes every conclusion occurrence false at its assigned component. The attached certificate and certificate semantic soundness give $\iota(u)R_i\iota(v).$ The retained occurrence $\Diamond_ip$ is false at $\iota(u)$, so every $R_i$ successor of $\iota(u)$ falsifies $p$. In particular, $p$ is false at $\iota(v)$. Every formula already present at the target was false there, so adding the new $p$ occurrence leaves all premise occurrences false under the same position map. The certificate rule changes no edge, so the map remains edge respecting. Semantics of position maps then makes $F^{p}$ false at $w$. This completes the contraposition argument.

By Theorem~\ref{thm:prop-canonical-completeness}, $H_{\Amod}^{\mathrm{prop}}\vdash F^{p}\to F^{c}.$ Uniformly substitute the original first-order formulas for the abstract propositional variables. Every Hilbert scheme and rule is closed under uniform formula substitution in the temporary signature, so $H_{\Amod}[Q]\vdash \sem{\text{premise}}\to\sem{\text{conclusion}}.$ Hence $H_{\Amod}[Q]$ theoremhood of the premise interpretation entails $H_{\Amod}[Q]$ theoremhood of the conclusion interpretation.
\end{proof}

\subsection{Hilbert lemmas for deep contexts}\label{app:remaining-hilbert-proof}

Let $Q\subseteq\mathsf{Par}$ be an arbitrary finite set, let $\mathcal G\{\}$ be a finite one-hole nested context over the temporary signature $\Sigma\cup Q$, and let
$\Phi_{\mathcal G}(A)=\sem{\mathcal G\{A\}}$.
The fixed formulas and brackets of the context are treated as parameters.

\begin{lemma}[Hilbert principles for deep contexts]\label{lem:deep-hilbert-principles}
The following statements hold in $H_{\Amod}[Q]$.
\begin{enumerate}
\item If $H_{\Amod}[Q]\vdash A\to B$, then
$H_{\Amod}[Q]\vdash\Phi_{\mathcal G}(A)\to\Phi_{\mathcal G}(B).$
\item If $H_{\Amod}[Q]\vdash A$, then $H_{\Amod}[Q]\vdash\Phi_{\mathcal G}(A)$.
\item
$H_{\Amod}[Q]\vdash (\Phi_{\mathcal G}(A)\land\Phi_{\mathcal G}(B))\to\Phi_{\mathcal G}(A\land B).$
\item If $y$ is absent from the fixed part of $\mathcal G\{\}$, then
$H_{\Amod}[Q]\vdash \forall y\,\Phi_{\mathcal G}(B(y))\to\Phi_{\mathcal G}(\forall yB(y)).$
\item If $a\in\mathsf{Par}\setminus Q$ is fresh for the fixed context and for $B$ except through $B(a)$, and $H_{\Amod}[Q\cup\{a\}]$ proves $\Phi_{\mathcal G}(B(a))$, then $H_{\Amod}[Q]$ proves $\Phi_{\mathcal G}(\forall yB(y))$ for a fresh object variable $y$.
\end{enumerate}
\end{lemma}

\begin{proof}
We prove clauses 1--4 simultaneously by induction on the number of bracket boundaries between the hole and the root.

Base case. At depth zero there is a fixed disjunctive remainder $R$ such that, modulo the fixed association convention, $\Phi_{\mathcal G}(A)=R\lor A$.

\begin{caselist}
\item Case deep monotonicity. Use the propositional theorem $(A\to B)\to((R\lor A)\to(R\lor B))$.
\item Case theorem lifting. Use disjunction introduction.
\item Case deep conjunction. Use $((R\lor A)\land(R\lor B))\to(R\lor(A\land B))$.
\item Case universal movement. First derive $\forall y(R\lor F(y))\to(R\lor\forall yF(y))$ when $y\notin\FV(R)$. Let $C_0=\neg R\land\forall y(R\lor F(y))$. Universal instantiation and propositional reasoning give $C_0\to F(y)$. Since $y\notin\FV(C_0)$, theorem generalization followed by the first-order distribution scheme gives $C_0\to\forall yF(y)$. Propositional rearrangement yields the required formula.
\end{caselist}

Induction step. Assume clauses 1--4 for a smaller context $\mathcal G'\{\}$. At positive depth there are a fixed $R$, an index $i$, and such a context with $\Phi_{\mathcal G}(A)=R\lor\Box_i\Phi_{\mathcal G'}(A)$.

\begin{caselist}
\item Case deep monotonicity. The induction hypothesis gives $\Phi_{\mathcal G'}(A)\to\Phi_{\mathcal G'}(B)$. Necessitation and $\mathrm{K}_i$ lift it through $\Box_i$, and propositional reasoning lifts the result through $R\lor(-)$.
\item Case theorem lifting. The induction hypothesis gives $\Phi_{\mathcal G'}(A)$ as a theorem, necessitation gives $\Box_i\Phi_{\mathcal G'}(A)$, and disjunction introduction adds $R$.
\item Case deep conjunction. Write $F_1=\Phi_{\mathcal G'}(A)$, $F_2=\Phi_{\mathcal G'}(B)$, and $F_3=\Phi_{\mathcal G'}(A\land B)$. The induction hypothesis gives $(F_1\land F_2)\to F_3$. Normal modal logic proves $(\Box_iF_1\land\Box_iF_2)\to\Box_i(F_1\land F_2)$. Necessitation and $\mathrm{K}_i$ lift the induction implication to $\Box_i(F_1\land F_2)\to\Box_iF_3$. Combine the two modal implications and lift the result through $R\lor(-)$ propositionally.
\item Case universal movement. Set $F_1(y)=\Phi_{\mathcal G'}(B(y))$ and $F_2=\Phi_{\mathcal G'}(\forall yB(y))$. The induction hypothesis gives $\forall yF_1(y)\to F_2$. Quantifier movement at the root gives $\forall y(R\lor\Box_iF_1(y))\to R\lor\forall y\Box_iF_1(y)$. The Barcan scheme of the constant-domain base gives $\forall y\Box_iF_1(y)\to\Box_i\forall yF_1(y)$. Necessitation and $\mathrm{K}_i$ lift the induction implication to $\Box_i\forall yF_1(y)\to\Box_iF_2$. Compose these implications and lift through the fixed disjunction.
\end{caselist}

For clause 5, assume a finite $H_{\Amod}[Q\cup\{a\}]$ proof of $\Phi_{\mathcal G}(B(a))$. Choose an object variable $y$ absent from that proof and from the context, alpha converting bound variables where necessary. Uniformly replace $a$ by $y$ throughout the finite proof. The transformation preserves axiom instances, modus ponens, and necessitation. After the preliminary alpha conversion, it also preserves theorem generalization. Hence $H_{\Amod}[Q]$ proves $\Phi_{\mathcal G}(B(y))$. Generalize to $\forall y\Phi_{\mathcal G}(B(y))$ and use clause 4.
\end{proof}

The preceding principles for deep contexts provide the syntactic tools needed for the remaining primitive kernel rules. We use them to prove Lemma~\ref{lem:remaining-rule-hilbert}.

\begin{proof}[Detailed proof of Lemma~\ref{lem:remaining-rule-hilbert}]
Fix a finite $Q\subseteq\mathsf{Par}$ containing every proof parameter in the rule instance, let $H=H_{\Amod}[Q]$, and consider a deep context $\mathcal G\{\}$ around the active occurrence.

For a truth initial, $H\vdash\top$, so theorem lifting from Lemma~\ref{lem:deep-hilbert-principles} gives $H\vdash\Phi_{\mathcal G}(\top)$. For atomic identity, classical logic proves $p(\bar t)\lor\neg p(\bar t)$, which is the Hilbert formula corresponding to the complementary core literals. Theorem lifting places this theorem in the surrounding context.

The disjunction premise $\mathcal G\{A,B\}$ and conclusion $\mathcal G\{A\lor B\}$ have propositionally equivalent recursive interpretations. For conjunction, theoremhood of both premises gives $\Phi_{\mathcal G}(A)\land\Phi_{\mathcal G}(B)$, and the deep conjunction clause gives $\Phi_{\mathcal G}(A\land B)$.

For the retaining existential rule, first-order existential introduction gives $A(t/x)\to\exists xA.$ Propositional reasoning gives $(\exists xA\lor A(t/x))\to\exists xA.$ Deep monotonicity lifts this implication through the context.

For the universal eigenparameter rule with premise $\mathcal G\{A(a/x)\}$ and conclusion $\mathcal G\{\forall xA\}$, put $Q_0=Q\setminus\{a\}$. Freshness of $a$ in the conclusion makes the premise theoremhood judgment an $H_{\Amod}[Q_0\cup\{a\}]$ judgment to which the clause for a fresh parameter of Lemma~\ref{lem:deep-hilbert-principles} applies. It yields the conclusion already in $H_{\Amod}[Q_0]$, and hence also in $H_{\Amod}[Q]$. Use alpha equivalence of the quantified matrix as needed.

For box, the premise child $[A]_i$ contributes the same formula $\Box_iA$ as the principal formula in the conclusion, so the recursive interpretations coincide.

For selected seriality, the premise adds an empty $i$ child whose local contribution is $\Box_i\bot$. The formula $\top$ denotes $\neg\bot$ in Hilbert judgments. Since $\top$ is a theorem, necessitation gives $\Box_i\neg\bot$. Modal duality and classical reasoning then give $\neg\Diamond_i\bot$. The selected instance $\mathrm D_i$ at $\bot$ gives $\Box_i\bot\to\Diamond_i\bot$. Classical reasoning yields $\Box_i\bot\to\bot$. Deep monotonicity lifts this implication through the surrounding context. No other primitive rule remains.
\end{proof}

\section{Structural soundness}\label{app:raw-structural}

The raw calculus is defined in Definition~\ref{def:raw-calculus}. This appendix proves the local validity preservation needed for the admissibility result in the main text. Theorem~\ref{thm:fanout-counterexample} gives a finite modal module for which the raw calculus is incomplete.

\subsection{Semantic soundness}\label{app:raw-semantics}

\begin{proof}[Detailed proof of Proposition~\ref{prop:raw-semantic-soundness}]
We argue contrapositively using counterinterpretations based on position maps.

For $[\mathrm{t}_i]$, map the active component of an invalid conclusion to $w$. Reflexivity gives $wR_iw$. Map the new $i$ child of the premise to the same world. All formulas remain false.

For $[\mathrm{d}_i]$, map the active component to $w$. Seriality gives $v$ with $wR_iv$. Map the fresh empty child to $v$.

For $[\mathrm I(i,j)]$, an invalid conclusion has a displayed $j$ edge $wR_jv$. The selected inclusion condition gives $wR_iv$, so the same component assignment respects the premise.

For $[\mathrm b(i,j)]$, let the parent map to $w$ and its displayed $i$ child to $v$. The selected symmetry condition gives $vR_jw$. Keep the $i$ child at $v$ and map the premise's new $j$ child to $w$. The formulas moved into that child remain false at the world where they were false in the conclusion.

For $[4(i,j,k)]$, let the parent map to $w$, the displayed $j$ child to $v$, and its $k$ child to $u$. The selected condition gives $wR_iu$. Map the premise's $i$ child to $u$ and its $j$ child to $v$. Every formula remains at its previous counterexample world.

For $[5(i,j,k)]$, let the parent map to $w$, the $i$ child to $u$, and the $j$ child to $v$. The selected condition gives $vR_ku$. Map the premise's nested $k$ child to $u$ and keep the $j$ child at $v$. Again every formula remains false at the same world.
\end{proof}


\begin{thebibliography}{99}

\bibitem{brunnler2009deep}
K. Br\"unnler,
``Deep Sequent Systems for Modal Logic,''
\emph{Archive for Mathematical Logic}, 48(6), 551--577, 2009.
\doi{10.1007/s00153-009-0137-3}.

\bibitem{brunnler2010nested}
K. Br\"unnler,
\emph{Nested Sequents},
Habilitationsschrift, Institut f\"ur Informatik und angewandte Mathematik, Universit\"at Bern, 2010, arXiv:1004.1845.
\doi{10.48550/arXiv.1004.1845}.

\bibitem{demri2005regular}
S. Demri and H. de Nivelle,
``Deciding Regular Grammar Logics with Converse Through First-Order Logic,''
\emph{Journal of Logic, Language and Information}, 14(3), 289--329, 2005.
\doi{10.1007/s10849-005-5788-9}.

\bibitem{tiu2012grammar}
A. Tiu, E. Ianovski, and R. Gor\'e,
``Grammar Logics in Nested Sequent Calculus: Proof Theory and Decision Procedures,''
in T. Bolander, T. Bra\"uner, S. Ghilardi, and L. Moss (eds.), \emph{Advances in Modal Logic 9},
pp. 516--537, College Publications, 2012.

\bibitem{lyon2022firstorder}
T. S. Lyon,
``Nested Sequents for First-Order Modal Logics via Reachability Rules,''
arXiv:2210.00789, 2022, revised 2023.
\doi{10.48550/arXiv.2210.00789}.

\bibitem{lyon2023quantified}
T. S. Lyon and E. Orlandelli,
``Nested Sequents for Quantified Modal Logics,''
in R. Ramanayake and J. Urban (eds.), \emph{Automated Reasoning with Analytic Tableaux and Related Methods},
Lecture Notes in Computer Science, 14278, pp. 449--467, Springer, Cham, 2023.
\doi{10.1007/978-3-031-43513-3_24}.

\bibitem{lyon2026horn}
T. S. Lyon and E. Orlandelli,
``Nested Sequents for Horn-Characterizable Quantified Modal Logics with Equality via Reachability Rules,''
arXiv:2604.18403, 2026.
\doi{10.48550/arXiv.2604.18403}.

\bibitem{farinas1986molog}
L. Fari\~nas del Cerro,
``MOLOG: A System That Extends PROLOG with Modal Logic,''
\emph{New Generation Computing}, 4(1), 35--50, 1986.
\doi{10.1007/BF03037381}.

\bibitem{balbiani1988declarative}
P. Balbiani, L. Fari\~nas del Cerro, and A. Herzig,
``Declarative Semantics for Modal Logic Programs,''
in \emph{Proceedings of the International Conference on Fifth Generation Computer Systems 1988},
pp. 507--514, ICOT, 1988.

\bibitem{debart1992multimodal}
F. Debart, P. Enjalbert, and M. Lescot,
``Multimodal Logic Programming Using Equational and Order-Sorted Logic,''
\emph{Theoretical Computer Science}, 105(1), 141--166, 1992.
\doi{10.1016/0304-3975(92)90290-V}.

\bibitem{nonnengart1994modalities}
A. Nonnengart,
``How to Use Modalities and Sorts in Prolog,''
in C. MacNish, D. Pearce, and L. M. Pereira (eds.), \emph{Logics in Artificial Intelligence},
Lecture Notes in Computer Science, 838, pp. 365--378, Springer, Berlin, 1994.
\doi{10.1007/BFb0021985}.

\bibitem{orgun1994overview}
M. A. Orgun and W. Ma,
``An Overview of Temporal and Modal Logic Programming,''
in D. M. Gabbay and H. J. Ohlbach (eds.), \emph{Temporal Logic},
Lecture Notes in Computer Science, 827, pp. 445--479, Springer, Berlin, 1994.
\doi{10.1007/BFb0014004}.

\bibitem{baldoni1996framework}
M. Baldoni, L. Giordano, and A. Martelli,
``A Framework for Modal Logic Programming,''
in M. Maher (ed.), \emph{Proceedings of the Joint International Conference and Symposium on Logic Programming},
pp. 52--66, MIT Press, 1996.

\bibitem{baldoni1998modal}
M. Baldoni, L. Giordano, and A. Martelli,
``A Modal Extension of Logic Programming: Modularity, Beliefs and Hypothetical Reasoning,''
\emph{Journal of Logic and Computation}, 8(5), 597--635, 1998.
\doi{10.1093/logcom/8.5.597}.

\bibitem{nguyen2003fixpoint}
L. A. Nguyen,
``A Fixpoint Semantics and an SLD-Resolution Calculus for Modal Logic Programs,''
\emph{Fundamenta Informaticae}, 55(1), 63--100, 2003.
\doi{10.3233/FUN-2003-55105}.

\bibitem{nguyen2006multimodal}
L. A. Nguyen,
``Multimodal Logic Programming,''
\emph{Theoretical Computer Science}, 360(1--3), 247--288, 2006.
\doi{10.1016/j.tcs.2006.03.026}.

\bibitem{nguyen2007databases}
L. A. Nguyen,
``Foundations of Modal Deductive Databases,''
\emph{Fundamenta Informaticae}, 79(1--2), 85--135, 2007.
\doi{10.3233/FUN-2007-791-206}.

\bibitem{nadathur1993harrop}
G. Nadathur,
``A Proof Procedure for the Logic of Hereditary Harrop Formulas,''
\emph{Journal of Automated Reasoning}, 11(1), 115--145, 1993.
\doi{10.1007/BF00881902}.

\bibitem{miller1991uniform}
D. Miller, G. Nadathur, F. Pfenning, and A. Scedrov,
``Uniform Proofs as a Foundation for Logic Programming,''
\emph{Annals of Pure and Applied Logic}, 51(1--2), 125--157, 1991.
\doi{10.1016/0168-0072(91)90068-W}.

\bibitem{harland1997goal}
J. Harland,
``On Goal-Directed Provability in Classical Logic,''
\emph{Computer Languages}, 23(2--4), 161--178, 1997.
\doi{10.1016/S0096-0551(97)00013-1}.

\bibitem{nadathur1998uniform}
G. Nadathur,
``Uniform Provability in Classical Logic,''
\emph{Journal of Logic and Computation}, 8(2), 209--229, 1998.
\doi{10.1093/logcom/8.2.209}.

\bibitem{gabbay2000goal}
D. M. Gabbay and N. Olivetti,
\emph{Goal-Directed Proof Theory},
Applied Logic Series, 21, Kluwer Academic Publishers, Dordrecht, 2000.
\doi{10.1007/978-94-017-1713-7}.

\bibitem{stone2005disjunction}
M. Stone,
``Disjunction and Modular Goal-Directed Proof Search,''
\emph{ACM Transactions on Computational Logic}, 6(3), 539--577, 2005.
\doi{10.1145/1071596.1071599}.

\bibitem{miller2022survey}
D. Miller,
``A Survey of the Proof-Theoretic Foundations of Logic Programming,''
\emph{Theory and Practice of Logic Programming}, 22(6), 859--904, 2022.
\doi{10.1017/S1471068421000533}.

\bibitem{miller1992unification}
D. Miller,
``Unification under a Mixed Prefix,''
\emph{Journal of Symbolic Computation}, 14(4), 321--358, 1992.
\doi{10.1016/0747-7171(92)90011-R}.

\bibitem{nadathur1995scoping}
G. Nadathur, B. Jayaraman, and K. Kwon,
``Scoping Constructs in Logic Programming: Implementation Problems and Their Solution,''
\emph{The Journal of Logic Programming}, 25(2), 119--161, 1995.
\doi{10.1016/0743-1066(95)00037-K}.

\bibitem{andreoli1992focusing}
J.-M. Andreoli,
``Logic Programming with Focusing Proofs in Linear Logic,''
\emph{Journal of Logic and Computation}, 2(3), 297--347, 1992.
\doi{10.1093/logcom/2.3.297}.

\bibitem{liang2009focusing}
C. Liang and D. Miller,
``Focusing and Polarization in Linear, Intuitionistic, and Classical Logics,''
\emph{Theoretical Computer Science}, 410(46), 4747--4768, 2009.
\doi{10.1016/j.tcs.2009.07.041}.

\bibitem{chaudhuri2008logical}
K. Chaudhuri, F. Pfenning, and G. Price,
``A Logical Characterization of Forward and Backward Chaining in the Inverse Method,''
\emph{Journal of Automated Reasoning}, 40(2--3), 133--177, 2008.
\doi{10.1007/s10817-007-9091-0}.

\bibitem{chaudhuri2016focusednested}
K. Chaudhuri, S. Marin, and L. Stra{\ss}burger,
``Focused and Synthetic Nested Sequents,''
in B. Jacobs and C. L\"oding (eds.), \emph{Foundations of Software Science and Computation Structures},
Lecture Notes in Computer Science, 9634, pp. 390--407, Springer, Berlin, 2016.
\doi{10.1007/978-3-662-49630-5_23}.

\bibitem{chaudhuri2016modularfocused}
K. Chaudhuri, S. Marin, and L. Stra{\ss}burger,
``Modular Focused Proof Systems for Intuitionistic Modal Logics,''
in D. Kesner and B. Pientka (eds.), \emph{1st International Conference on Formal Structures for Computation and Deduction},
Leibniz International Proceedings in Informatics, 52, pp. 16:1--16:18, Schloss Dagstuhl--Leibniz-Zentrum f\"ur Informatik, 2016.
\doi{10.4230/LIPIcs.FSCD.2016.16}.

\end{thebibliography}
\end{document}